\documentclass[]{siamart171218}

\usepackage{lipsum}
\usepackage{amsmath}
\usepackage{amsfonts}
\usepackage{amssymb}
\usepackage{mathtools}
\usepackage{commath}
\usepackage{graphicx}
\usepackage{epstopdf}
\usepackage{algorithmic}
\usepackage{todonotes}
\usepackage{pgfplots}
\usepackage{enumitem}
\usepackage{hyperref}
\usepackage{wasysym}

\usepackage{booktabs}
\usepackage{siunitx}
\ifpdf
  \DeclareGraphicsExtensions{.eps,.pdf,.png,.jpg}
\else
  \DeclareGraphicsExtensions{.eps}
\fi
 
\newcommand{\consistQ}[1]{\mathbf{E}_C^{#1}}
\newcommand{\consist}[1]{\hyperref[eq:consist]{\mathbf{E}_C^{#1}}}
\newcommand{\interpolQ}[1]{\mathbf{E}_I^{#1}}
\newcommand{\interpol}[1]{\hyperref[eq:approx]{\mathbf{E}_I^{#1}}}
\newcommand{\errQ}[1]{\mathbf{e}^{#1}}
\newcommand{\err}[1]{\hyperref[eq:split]{\mathbf{e}^{#1}}}
\newcommand{\errpQ}[1]{\tilde{\mathbf{e}}^{#1}}
\newcommand{\errp}[1]{\hyperref[eq:split proj]{\tilde{\mathbf{e}}^{#1}}}
\newcommand{\errdQ}[1]{\mathbf{e}_h^{#1}}
\newcommand{\errd}[1]{\hyperref[eq:split]{\mathbf{e}_h^{#1}}}

\newcommand{\erriQ}[1]{\mathbf{e}_I^{#1}}
\newcommand{\erri}[1]{\hyperref[eq:split]{\mathbf{e}_I^{#1}}}
\newcommand{\erripQ}[1]{\tilde{\mathbf{e}}_I^{#1}}
\newcommand{\errip}[1]{\hyperref[eq:split proj]{\tilde{\mathbf{e}}_I^{#1}}}

\newcommand{\N}{\mathbb{N}}
\newcommand{\eps}{\varepsilon}
\newcommand{\R}{\mathbb{R}}
\providecommand{\C}{\mathcal{C}}
\renewcommand{\C}{\mathcal{C}}
\newcommand{\V}{\mathcal{V}}
\newcommand{\W}{\mathcal{W}}
\renewcommand{\H}{\mathcal{H}}
\renewcommand{\L}{\mathcal{L}}
\renewcommand{\P}{\mathcal{P}}
\newcommand{\F}{\mathcal{F}}
\renewcommand{\S}{\mathcal{S}}
\newcommand{\I}{\mathcal{I}}
\newcommand{\E}{\mathcal{E}}

\newcommand{\T}{\mathcal{T}}
\newcommand{\x}{\mathbf{x}}
\newcommand{\dt}{\Delta t}
\newcommand{\n}{\mathbf{n}}
\newcommand{\w}{\mathbf{w}}

\renewcommand{\O}{\mathcal{O}}
\newcommand{\Om}{\Omega}
\providecommand{\G}{\Gamma}
\renewcommand{\G}{\Gamma}

\renewcommand{\i}{\infty}
\newcommand{\remove}[1]{}
\newcommand{\id}{\operatorname{id}}

\renewcommand{\div}{\operatorname{div}}
\newcommand{\dist}{\operatorname{dist}}
\newcommand{\meas}{\operatorname{meas}}

\newcommand{\finaltime}{\mathbf{T}}

\newcommand{\extendedonly}[1]{}
\newcommand{\shortonly}[1]{#1}

\newcommand{\Vh}{\hyperref[def:Vh]{\V_h}}
\newcommand{\Vhbq}[1]{\hyperref[def:Vh]{\V_h^{(#1)}}}
\newcommand{\Vhn}[1]{\hyperref[def:Vhn]{\V_h^{#1}}}
\newcommand{\Omhn}[1]{\hyperref[def:Vhn]{\Om_h^{#1}}}
\newcommand{\Omihn}[2]{{\Om_{#1,h}^{#2}}}
\newcommand{\Ghn}[1]{\hyperref[def:Vhn]{\G_h^{#1}}}
\newcommand{\Thn}[1]{\hyperref[def:Vhn]{\T_h^{#1}}}
\newcommand{\Vdn}[2]{\hyperref[def:Vdn]{\V_{#1}^{#2}}}
\newcommand{\Sdn}[2]{\hyperref[def:Sdn]{\S_{#1}^{#2}}}
\newcommand{\Srn}[2]{\hyperref[def:Srn]{\S_{#1}^{#2}}}
\newcommand{\Sdelta}{\hyperlink{def:Sdelta}{S_{\delta}}}
\newcommand{\TdnS}[2]{\hyperref[def:TdnSOdnS]{\T_{#1}^{#2\!,\!\,\S}}}
\newcommand{\OdnS}[2]{\hyperref[def:TdnSOdnS]{\mathcal{O}_{#1}^{#2\!,\!\,\S}}}
\newcommand{\Omdn}[2]{\hyperref[def:OmdnTdnOdn]{\Omega_{#1}^{#2}}}
\newcommand{\Odn}[2]{\hyperref[def:OmdnTdnOdn]{\mathcal{O}_{#1}^{#2}}}
\newcommand{\Tdn}[2]{\hyperref[def:OmdnTdnOdn]{\mathcal{T}_{#1}^{#2}}}

\newcommand{\Omn}[1]{\hyperlink{def:Omn}{\Om^{#1}}}
\newcommand{\Gn}[1]{\hyperlink{def:Gn}{\G^{#1}}}

\newcommand{\Omepsn}[1]{\hyperlink{def:Omepsn}{\Om_{\epsilon}^{#1}}}

\newcommand{\Th}{\hyperlink{def:Th}{\T_h}}

\newcommand{\Q}{\hyperref[def:Q]{Q}}
\newcommand{\Qe}{\hyperref[def:Q]{Q_\epsilon}}
\newcommand{\B}{\hyperlink{def:B}{\widetilde{\Omega}}}
\newcommand{\Omphi}[2]{\hyperlink{def:Omphi}{\Omega_{#1}^{#2}}}
\newcommand{\Omphih}[2]{\hyperlink{def:Omphih}{\Omega_{#1}^{#2}}}

\newcommand{\Orn}[2]{\hyperlink{def:Orn}{\O_{#1}^{#2}}}
\newcommand{\OrnS}[2]{\hyperlink{def:OrnS}{\O_{#1}^{#2,\S}}}
\newcommand{\Vrn}[2]{\hyperlink{def:Vrn}{\V_{#1}^{#2}}}
\newcommand{\xxi}{\hyperref[eq:timestep]{\xi}}
\newcommand{\Ep}{\hyperlink{def:ep}{\mathcal{E}^{\mathcal{P}}}}

\newcommand{\enorm}[2]{\hyperref[eq:norm]{|\!|\!|}#1\hyperref[eq:norm]{|\!|\!|_{#2}}}
\newcommand{\Frn}[2]{\hyperref[eq:facetset]{\mathcal{F}_{#1}^{#2}}}
\newcommand{\Firn}[3]{{\mathcal{F}_{#1,#2}^{#3}}}
\newcommand{\bh}[1]{\hyperref[def:bh]{b_h^{#1}}}
\newcommand{\bih}[2]{{b_{#1,h}^{#2}}}
\newcommand{\fhn}[1]{\hyperlink{def:fhn}{f_h^{#1}}}
\newcommand{\srn}[2]{\hyperref[eq:ghostpenalty]{s_{#1}^{#2}}}
\newcommand{\sirn}[3]{{s_{#1,#2}^{#3}}}
\newcommand{\ext}{\hyperlink{def:ext}{e}}
\newcommand{\omx}[1]{\hyperref[eq:nodepatch]{\omega(#1)}}
\newcommand{\omF}[1]{\hyperref[eq:facetpatch]{\omega(#1)}}
\newcommand{\omT}[1]{\hyperref[eq:elementpatch]{\omega(#1)}}
\newcommand{\homega}[1]{\hyperlink{def:hom}{\hat{\omega}(#1)}}

\newcommand{\Thneps}[2]{\hyperref[lemma:dmax]{\hat{T}^{#1}_{#2}}}
\newcommand{\Tneps}[2]{\hyperlink{def:Tneps}{T^{#1}_{#2}}}

\renewcommand{\d}{\delta}
\newcommand{\dd}{\hyperlink{cond:delta}{\delta}}
\newcommand{\K}{\hyperref[ass:paths]{K}}

\newcommand{\LagL}{\hyperlink{def:L}{L}}

\newcommand{\dtBDFb}{\hyperref[eq:stencil:BDF2]{\partial_{\dt}^{2}}}

\newcommand{\dtBDFr}[1]{\hyperref[eq:stencils]{\partial_{\dt}^{#1}}}

\newcommand{\uIn}[1]{\hyperlink{def:uIn}{u_I^{#1}}}
\newcommand{\Ihn}[1]{\hyperlink{def:Ihn}{\I_h^{#1}}}
\newcommand{\PPsi}{\hyperlink{def:Psi}{\Psi}}
\newcommand{\Psin}[1]{\hyperlink{def:Psi}{\Psi^{#1}}}
\newcommand{\lift}{\hyperlink{def:lift}{\ell}}

\newcommand{\Ty}{\hyperlink{def:Taylor}{\mathfrak{T}}}

\newcommand{\ltwo}{\hyperlink{def:l2}{l_2}}

\newcommand{\Pinm}[2]{\hyperref[eq:defprojr]{\Pi^{#1}_{#2}}}
\newcommand{\Pimn}[2]{\hyperref[eq:defprojr]{\Pi_{#1}^{#2}}}
\newcommand{\Pin}[1]{\hyperlink{def:proj}{\Pi^{#1}}}

\newcommand{\TnOm}[1]{\hyperlink{def:TnOm}{T_\Om^{#1}}}

\newcommand{\bdftwonorm}[4]{\hyperref[eq:bdf2norm]{\Vert ( } #1 \hyperref[eq:bdf2norm]{,} #2 \hyperref[eq:bdf2norm]{) \Vert_{#3}^{#4}}}

\newcommand{\LinfW}[1]{\hyperlink{def:LinfW}{\L^\infty(\W^{#1}_{\infty})}}
\newcommand{\LinfH}[1]{\hyperlink{def:LinfH}{\L^\infty(\H^{#1})}}

\newcommand{\yihat}{\hyperlink{def:yihat}{\hat{y}_i}}

\newcommand{\new}[1]{#1}
\newcommand{\newc}[1]{#1}

\newsiamremark{remark}{Remark}
\newsiamremark{hypothesis}{Hypothesis}
\crefname{hypothesis}{Hypothesis}{Hypotheses}
\newsiamthm{claim}{Claim}

\headers{Isoparametric unfitted BDF-FEM on evolving domains}{Yimin Lou and Christoph Lehrenfeld}

\title{Isoparametric unfitted BDF -- Finite element method for PDEs on evolving domains
\thanks{Submitted to the editors DATE.
\funding{This work was funded by the German Science Foundation (DFG) within the project "LE 3726/1-1".}}}

\author{Yimin Lou\thanks{Institute for Numerical and Applied Mathematics, University of G\"{o}ttingen, Germany 
    (\email{lou@math.uni-goettingen.de}).}
\and Christoph Lehrenfeld\thanks{Institute for Numerical and Applied Mathematics, University of G\"{o}ttingen, Germany 
  (\email{lehrenfeld@math.uni-goettingen.de}).}
}

\usepackage{amsopn}

\newtheorem{assumption}{assumption}

\ifpdf
\hypersetup{
  pdftitle={Isoparametric unfitted BDF-FEM on evolving domains},
  pdfauthor={Yimin Lou and Christoph Lehrenfeld}
}
\fi

\begin{document}

\maketitle

% REQUIRED
\begin{abstract}
We propose a new discretization method for PDEs on moving domains in the setting of unfitted finite element methods, which is provably higher-order accurate in space and time. 
In the considered setting, the physical domain that evolves essentially arbitrarily through a time-independent computational background domain, is represented by a level set function.
For the time discretization, the application of standard time stepping schemes that are based on finite difference approximations of the time derivative is not directly possible, as the degrees of freedom may get active or inactive across such a finite difference stencil in time.
In [Lehrenfeld, Olshanskii. An Eulerian finite element method for PDEs in time-dependent domains. ESAIM: M2AN, 53:585--614, 2019] this problem is overcome by extending the discrete solution at every timestep to a sufficiently large neighborhood so that all the degrees of freedom that are relevant at the next time step stay active. But that paper focuses on low-order methods. We advance these results with introducing and analyzing realizable techniques for the extension to higher order. 
To obtain higher-order convergence in space and time, we combine the BDF time stepping with the isoparametric unfitted FEM. 
The latter has been used and analyzed for several stationary problems before. 
However, for moving domains the key ingredient in the method, the transformation of the underlying mesh, becomes time-dependent which gives rise to some technical issues. We treat these with special care, carry out an a priori error analysis and two numerical experiments. 
\end{abstract}

% REQUIRED

% REQUIRED

\section{Introduction}
\label{sec:intro}
Partial differential equations (PDEs) posed on time-de\-pen\-dent domains appear in many problems in physics, chemistry, biology and engineering. Famous problem classes of that sort are two-phase flow and free surface problems. 
In recent years, geometrically unfitted finite element methods (FEM) such as CutFEM \cite{cutFEM} have become very popular. In these methods, the geometry is described separately from the computational background mesh, which allow us to handle domains that may exhibit strong deformations or even topology changes. 
In the following we represent the physical domains, embedded and evolving smoothly in a time-independent background domain, implicitly by a level set function. 
While for PDEs on stationary domains, suitable unfitted finite element methods have been designed, analyzed, implemented and validated for a broad range of problems, the treatment of unfitted moving domains is less well-explored. 
A method-of-lines approach is not directly applicable as the domain of definition of the discrete solution and hence the corresponding unfitted finite element space changes between time instances. There are (at least) three approaches to solving this problem: 

  1. In \cite{hansbo2015characteristic} and very recently in \cite{MZZ_ARXIV_2021} a \emph{characteristic Galerkin} or \emph{semi-Lagrangian} formulation for a convection-diffusion problem on an evolving surface and an evolving bulk domain have been considered, respectively. Here, instead of discretizing the partial time derivative $\partial_t$, the material derivative is used and approximated by backtracking trajectories at required integration points. 

  2. A space-time reformulation of the problem with a moving domain 
  allows to transfer the main concepts of unfitted FEM from the case of stationary to that of time-dependent domains. Such an approach has been considered, e.g., in \cite{LR_SINUM_2013,L_SISC_2015} for scalar interface problems and recently extended to higher order in space and time in \cite{preussmaster,heimannmaster}. A variant that reduces the complexity that comes with a space-time formulation is the quadrature-in-time approach introduced in \cite{zahedi2017space} and applied in \cite{hansbo2016cut,frachon2019cut}.

  3. An alternative strategy stays in the framework of the usual method of lines. To this end, an extension is applied at every time step to make previous solutions well-defined in subsequent time steps. Without analysis and with the restriction of applying the extension to direct neighbors only, such a strategy has been considered in \cite{schottstabilized}. And in \cite{LO_ESAIM_2019}, for scalar problems, this approach has been generalized, studied more systematically and put on a mathematically rigorous foundation. Developments to unsteady Stokes problems on moving domains have been considered in \cite{burman2022eulerian,vWRL_ARXIV_2020}.

In this paper, we restrict ourselves to the third class of methods and extend \cite{LO_ESAIM_2019} with respect to two major limitations. First of all, in \cite{LO_ESAIM_2019} only an implicit Euler time discretization has been analyzed, while we upgrade this to the Backward Differentiation Formulas (BDF) for higher order of accuracy (with focus on BDF2 in the analysis). 
The second limitation of \cite{LO_ESAIM_2019} that we remove, is the abstract assumption of an arbitrarily accurate geometry handling. In \cite{LO_ESAIM_2019} it is assumed that the domain integrals on the implicitly -- through level set functions -- described geometries can be carried out robustly and arbitrarily accurate. In practice, however, this is hard to achieve. Existing  strategies for numerical integration on the cut geometries are typically either low-order accurate or fail to guarantee positive quadrature weights and hence stable quadrature. In \cite{lehrenfeld2015cmame} the concept of geometrically unfitted \emph{isoparametric} finite elements has been introduced as a remedy to combine guaranteed positive quadrature weights with arbitrarily accurate numerical integration on smooth domains. Afterwards, this approach has been successfully applied and analyzed to several \emph{stationary} problems in a.o. \cite{LR_IMAJNA_2018,lehrenfeld20162,LPWL_PAMM_2016,L_GUFEMA_2017,grande2018analysis}. 
To obtain a computational feasible but still higher-order accurate approach for the handling of the implicit geometry, we also consider the use of the isoparametric unfitted FEM in this work, but in the context of moving domain. 

\subsection*{Content and structure of the paper}
The major contribution in this work is the development of the method in \cite{LO_ESAIM_2019} to higher order of accuracy. This includes: 
\begin{itemize}
  \item Introduction of an arbitrarily high order in space and first to third order in time method for a scalar convection-diffusion equation on an evolving domain.
  \item Handling and estimates for mesh deformations changing in time in the isopa\-rametric unfitted FEM. These results are valuable not only for the considered BDF-based time stepping schemes, and play a major part of this work.
  \item A priori error analysis of the BDF2-based method that yields arbitrarily high order of accuracy in space and up to second-order convergence in time. 
  \item Numerical examples that confirms the predicted convergence rates and applications beyond the scope of the considered numerical analysis. 
\end{itemize}

This paper is organized as follows. 
In \Cref{sec:model} we introduce the PDE problem. In preparation of the definition of the method, in \Cref{sec:prelim} we gather notation and properties of the computational mesh, its time-dependent active parts, the isoparametric mesh transformation and a preliminary description of a transfer operator between different meshes. In \Cref{sec:method} the discretizations in space and time are given. The transfer operator is dicussed in more detail in \Cref{sec:projection} and several important results on the transfer operator are stated. Together with \Cref{sec:ana} where the a priori error analysis for the scheme is carried out, these two sections represent the most important pieces of this study. The main part of the paper concludes with \Cref{sec:experiments} that validates the theoretical findings and extends beyond them.
\extendedonly{This manuscript requires a significant mount of notation. To facilitate working with the many different symbols we make heavy use of the hyperlink capabilities of pdf documents.}

\section{Mathematical model}
\label{sec:model}

%\todo[inline]{Don't need to introduce the mapping $\Phi$}
For ease of presentation we mainly consider the convection-diffusion equation posed on an evolving domain. The method, however, has been verified feasible for some more complicated models, such as two-phase interface problem tested in \Cref{sec:experiments} or unsteady incompressible flows in \cite{burman2022eulerian,vWRL_ARXIV_2020}, and with some restriction on the time step size also in \cite{schottstabilized}.

Let $\Om(t) \subset \R^d,\; d=2,3,~$ be a time-dependent domain with Lipschitz boundary $\Gamma := \partial \Om$ evolving in a time interval $t \in [0,\finaltime], ~\finaltime \in \R_+$, which \hypertarget{defB}{is embedded in a polygonal, time-independent background domain $\widetilde{\Omega}$}.
For instance, $\Om(t)$ may be regarded as a volume of fluid under motion and deformation, with a material velocity field \extendedonly{of particles }$\w: \B \to \R^d$ that has a proper meaning on the whole \extendedonly{background }domain $\B$. 
The conservation of a scalar quantity $u(\x,t)$ of the fluid with a diffusive flux is governed by
  \begin{equation} \label{eq:eqn}
	\partial_t u + \nabla \cdot (u \w - \nu \nabla u) = g \quad \text{in} \;\Om(t),\; \qquad \nabla u \cdot \n = 0 \quad \text{on} \;\G(t), ~~t \in (0,\finaltime], 
  \end{equation}
where $\nu > 0$ denotes the diffusion coefficient, $g$ is a source term
and $\n(\x,t)$ is the unit normal on $\G(t) = \partial \Om(t)$. 
Here, for the sake of simplicity we apply boundary conditions that ensure the global conservation of $u$ in $\Om(t)$. For the treatment of Dirichlet-type boundary conditions or interface conditions using Nitsche's method we refer to \cite{vWRL_ARXIV_2020} and the numerical examples. 
Further, we assume proper given initial conditions $u(\x,0) = u_0(\x)~\text{in}~\Om(0)$.

In order to describe the time-dependent domain $\Om(t)$, a level set function $\phi(\x,t): \B \times [0,\finaltime] \to \R$ is utilized such that the boundary of the domain is represented by the zero level and the domain is described by the negative levels, i.e. 
  \begin{equation}
  	\Om(t) = \{\x \in \B: \phi(\x,t) < 0\}, ~\text{s.t.}~
    \G(t) = \{\x \in \B: \phi(\x,t) = 0\}.
  \end{equation}
In addition, we define an $\epsilon$-neighborhood of the domain $\Om_\epsilon(t) := \{ \x \in \B: \phi(\x,t) < \epsilon \}$ for some $\epsilon > 0$ corresponding space-time domains
\begin{equation} \label{def:Q}
	Q := \bigcup_{t \in (0,\finaltime)} \Om(t) \times \{t\}, \qquad 
	Q_\epsilon := \bigcup_{t \in (0,\finaltime)} \Om_\epsilon(t) \times \{t\}, \qquad 
	\Q \subset \Qe \subset \R^{d+1}. 
\end{equation}
  
\section{Preliminaries for the discretization}
\label{sec:prelim}
For the problem discussed in \Cref{sec:model} we seek a proper discretization\extendedonly{~ using
 an \emph{isoparametric unfitted finite element method} based on \emph{volumetric discrete extensions} as in \cite{LO_ESAIM_2019}. 
Here, a backward differentiation formula (BDF) for time stepping is coupled with a ghost penalty extension  on top of an unfitted finite element discretization in space}. 
%
%\todo[inline]{Summarize idea of Lehrenfeld-Olshanksii here again}
\extendedonly{

}
In this section we prepare notation, concepts, and assumptions for the definition of the method given in \Cref{sec:method}, esspecially w.r.t. the geometrical approximation, finite element spaces and active meshes. 
\extendedonly{
In \Cref{sec:fes and geom appr} we define notation w.r.t. time discretization and the geometrical approximation with using a parametric mesh transformation. 
The properties of the isoparametric mapping are then summarized in \Cref{sec:prop param1} and \Cref{sec:prop param2}. 
Further notation for active meshes and discrete neighborhoods is introduced in \Cref{sec:discrete neigh}. % prior to \Cref{sec:ass} where we build some assumptions. 
}

\subsection{Finite element spaces and geometrical approximation} \label{sec:fes and geom appr}
First of all, we introduce notation for the discrete time levels of the time stepping procedure. Let $\dt := \finaltime/N, ~N \in \N$ be the uniform time step of an equally-spaced subdivision of the time interval $(0,\finaltime]$ under investigation. 
Let $t_n := n \Delta t$ be a time instance, then we denote by quantities with upper index $n$ the corresponding quantity with restriction $t=t_n$, e.g. \hypertarget{def:Omn}{$\Om^n := \Om(t_n)$}, 
\hypertarget{def:Gn}{$\G^n := \G(t_n)$}, 
\hypertarget{def:phin}{$\phi^n := \phi(\cdot, t_n)$}, 
or 
\hypertarget{def:Omepsn}{$\Om_\epsilon^n := \Om_\epsilon(t_n)$}, $~n=0,...,N$.

\hypertarget{def:Th}{Let $\{\T_h\}$ be an admissible quasi-uniform family of simplicial triangulations with a diameter $h > 0$ on the background domain $\B$.}
On each of these triangulations $\T_h$ we define the time-independent, standard finite element space with polynomials of order $k$ as 
  \begin{equation} \label{def:Vh}
	\V_h = \V_h^{(k)} := \{v_h \in C(\B): v_h|_T \in \P_k(T), \forall T \in \Th\}.
  \end{equation}

\begin{remark}[Inequalities up to constants]
  In order to simplify the inequalities with generic constants $c$ that are independent of the mesh size $h$, time step $\dt$ and time $t_n$, in the following $x \lesssim y$ ($x \gtrsim y$) denotes $x \le c y$ ($x \ge c y$), and $x \simeq y$ indicates $x \lesssim y$ and $x \gtrsim y$. The hidden constant $c$ may be refered to by $c_{(\texttt{x}.\texttt{y})}$ where $(\texttt{x}.\texttt{y})$ is the label of the corresponding inequality using the $\lesssim$ or $\gtrsim$ notation. 
\end{remark}

In general, for $n=0,..,N$, only a good approximation $\phi_h^n \in \Vhbq{q},~q\geq 1$, e.g., a higher-order piecewise polynomial approximation to $\phi^n$, is given. 
We assume that $\partial \Omn{n}$ is sufficiently smooth and
  \begin{equation}
  	\operatorname{dist}\left(\partial \Omn{n}, \partial \Om_{\phi_h}^n\right) \lesssim ~ h^{q+1}, \qquad \forall t \in [0,\finaltime]
  \end{equation}
holds for \hypertarget{def:Omphi}{$\Om_{\phi_h}^n := \{ \x \in \B \mid \phi_h^n(\x) < 0 \}$}. 
A well-known issue with the implicit description of the level set functions is that realizations of quadrature rules that preserve the geometrical order of accuracy are difficult to achieve, cf. the discussion in \cite{lehrenfeld2015cmame}. 
In \extendedonly{the remainder of }this work we consider 
the \emph{isoparametric} approach introduced in \cite{lehrenfeld2015cmame} to tackle this problem. 
The underlying idea is that an only second-order approximation of $\Omphi{\phi_h}{n}$ based on the piecewise linear interpolation $\hat{\phi}_h^n$ of $\phi_h^n$ simplifies the realization of quadrature rules dramatically. 
This configuration then serves as a reference configuration on which quadrature rules can easily be constructed (e.g., by simple geometrical decomposition rules). 
\hypertarget{def:Omphih}{To improve the accuracy of this low-order approximation $\Om_{\hat{\phi}_h}^n$} an additional transformation $\Theta_n \in [\Vhbq{q}]^d$ is constructed at each time step $n$ such that
  \begin{equation} 
  	\operatorname{dist}\Big(\partial \Omn{n}, \Theta^n\big(\partial \Omphih{\hat{\phi}_h}{n}\big)\Big) \lesssim 
    \operatorname{dist}\Big(\partial \Omn{n}, \partial \Omphi{\phi_h}{n}\Big) + \operatorname{dist}\Big(\partial \Omphi{\phi_h}{n}, \Theta^n\big(\partial \Omphih{\hat{\phi}_h}{n}\big)\Big) \lesssim 
    h^{q+1}. %, \qquad \forall t \in [0,\finaltime].
  \end{equation} 
This transformation is itself a finite element function w.r.t. the (undeformed) background mesh which renders the task of accurate numerical integration feasible. 
The deformation is \emph{local}, i.e., only in the vicinity of cut elements it deviates from the identity, and \emph{small} everywhere in the sense that $\| \Theta^n \|_\infty \lesssim h^2$ (in detail \Cref{sec:prop param1}). 
However, the fact that the deformed meshes and the properly adapted finite element spaces are in general time-dependent, results in several technicalities\extendedonly{ as we will see}. 
Based on this configuration we define the high-order approximations of geometry, the deformed meshes and the time-dependent finite element spaces as (with $\Theta^{-n}:=(\Theta^n)^{-1}$)
  \begin{equation} \label{def:Vhn}
	\Om_h^n := \Theta^n\big(\Omphih{\hat{\phi}_h}{n}\big), \qquad
	\G_h^n := \partial \Om_h^n, \qquad
	\T_h^n := \Theta^n(\Th), \qquad
	\Vhn{n} := \Vh \circ \Theta^{-n}.
  \end{equation}
\new
{
We note that the mesh deformations from different time steps do not accumulate and remain \emph{local} and \emph{small}. This is in contrast to most body-fitted methods where the deformations (from the initial domain) accumulate until a remeshing takes place.
}

\subsection{Discrete neighborhoods and active meshes}\hypertarget{sec-discrete-neigh}{}\label{sec:discrete neigh}
As usual in unfitted finite element methods, only a part of the background mesh is involved in the computation at each time step. 
We therefore define \emph{active} parts of meshes and finite element spaces as those parts corresponding to the elements that overlap the physical domain $\Om(t)$ or its \emph{discrete} $\d$-neighborhood. We refer to \Cref{fig:domains} for a sketch of the different domains and meshes introduced next. 
\begin{figure}
  \begin{center}
  \scalebox{0.68}{
  % \documentclass{standalone}
% \usepackage{xr}
% \usepackage{tikz}
% \usepackage{amsmath}
% \usepackage{amssymb}
% \usepackage{amsfonts}
% \usepackage{multirow}
% \usetikzlibrary{shapes,arrows,snakes,calendar,matrix,backgrounds,folding,calc,positioning,patterns}

\definecolor{myblue}{HTML}{4383BD}%
\definecolor{myred}{HTML}{a11006}%
\definecolor{myblueb}{HTML}{1414f8}%
\definecolor{mypetrol}{HTML}{00c5bd}%
\definecolor{myyellow}{HTML}{c8c55e}%
\definecolor{mygreen}{HTML}{00c54c}%
\definecolor{mygray}{HTML}{c8c5bd}%
\definecolor{mypurple}{HTML}{c883bd}%
\definecolor{myviolet}{HTML}{5500ff}%

\begin{tabular}{@{}|@{}l@{}l@{}|@{}l@{}l@{}|@{}l@{}l@{}||@{}l@{}l@{}|@{}l@{}l@{}|@{}l@{}l@{}||@{}l@{}l@{}|@{}l@{}l@{}|@{}l@{}l@{}|@{}}
  \hline
    \multicolumn{2}{@{}|@{}c@{}|@{}}{\includegraphics[width=2cm,trim=  4.00cm 8cm 93.00cm 7.4cm,clip=true]{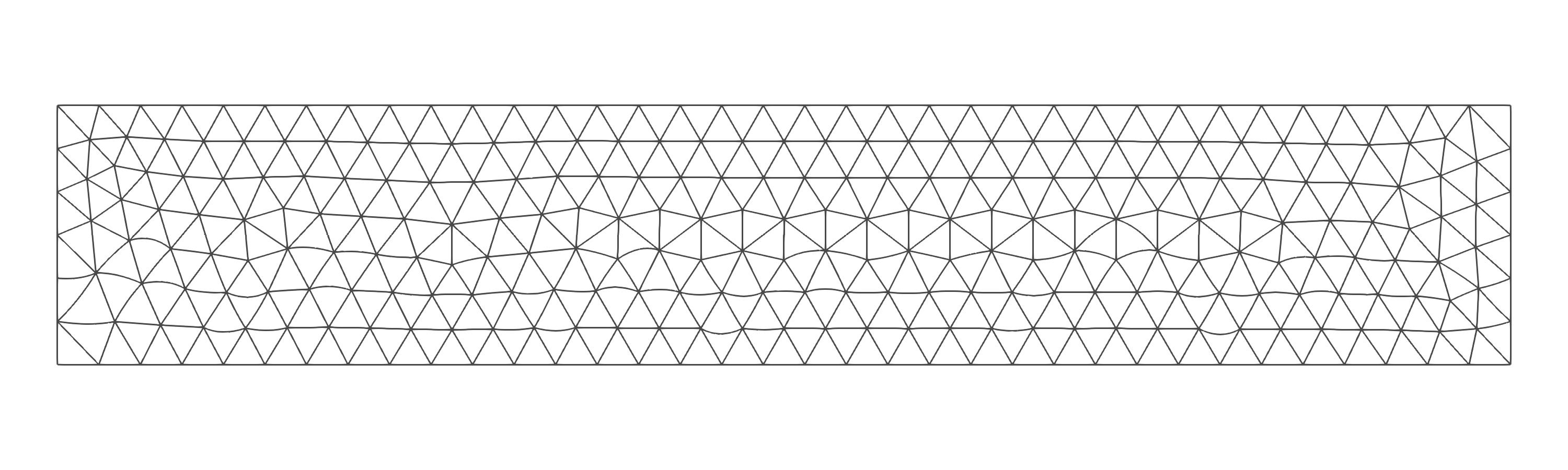}}
  & \multicolumn{2}    {@{}c@{}|@{}}{\includegraphics[width=2cm,trim= 12.86cm 8cm 84.14cm 7.4cm,clip=true]{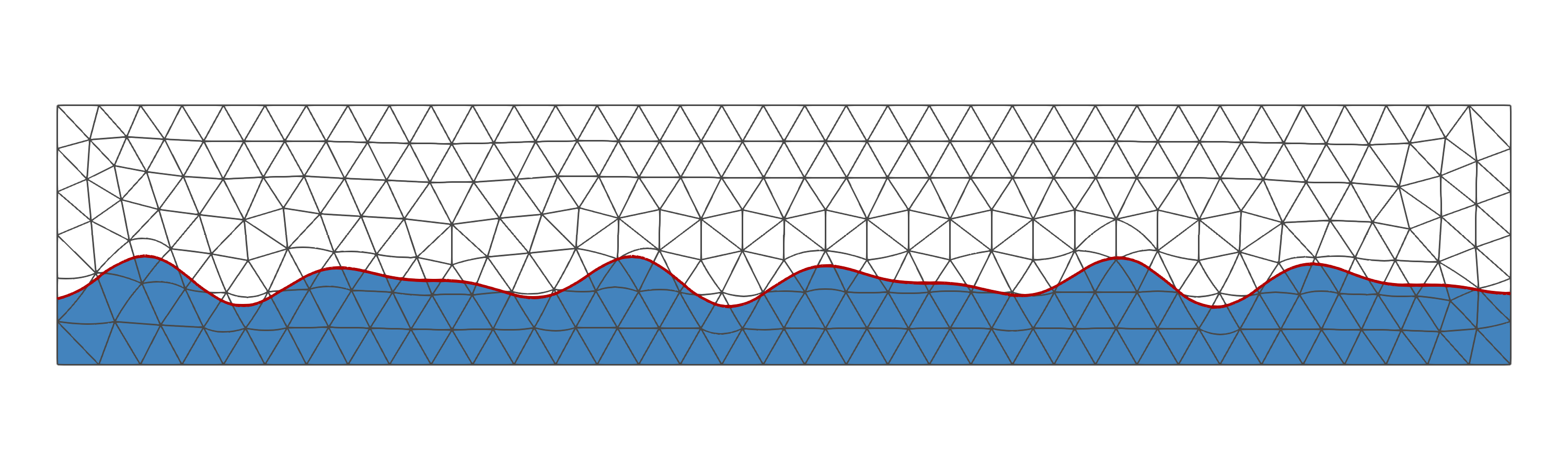}}
  & \multicolumn{2}    {@{}c@{}||@{}}{\includegraphics[width=2cm,trim= 21.72cm 8cm 75.28cm 7.4cm,clip=true]{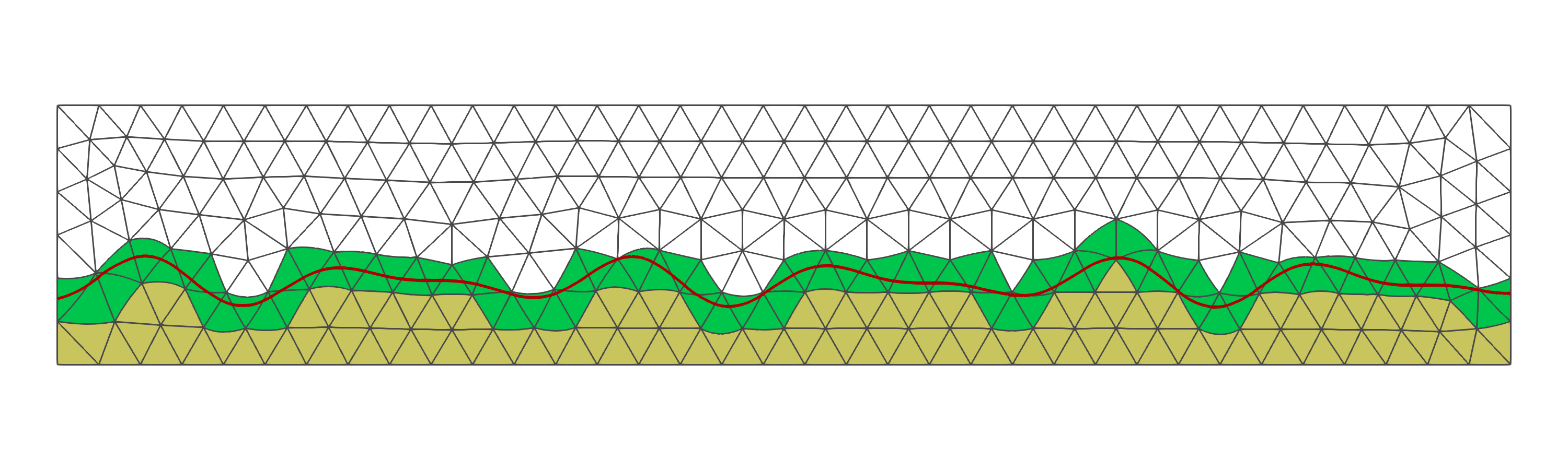}}
  & \multicolumn{2}    {@{}c@{}|@{}}{\includegraphics[width=2cm,trim= 30.58cm 8cm 66.42cm 7.4cm,clip=true]{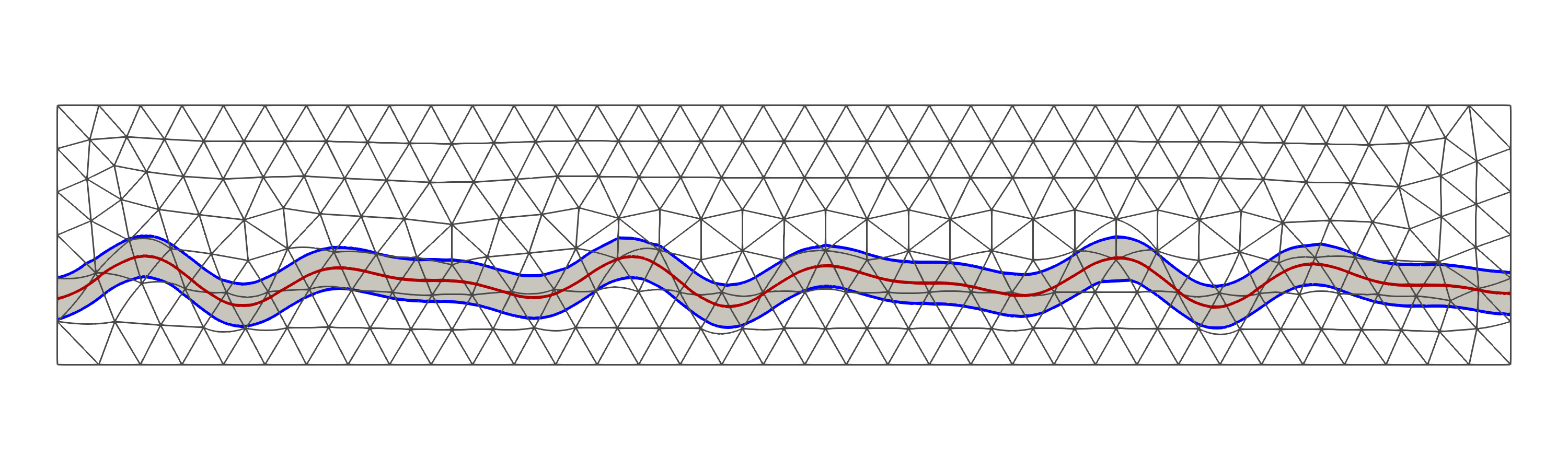}}
  & \multicolumn{2}    {@{}c@{}|@{}}{\includegraphics[width=2cm,trim= 39.44cm 8cm 57.56cm 7.4cm,clip=true]{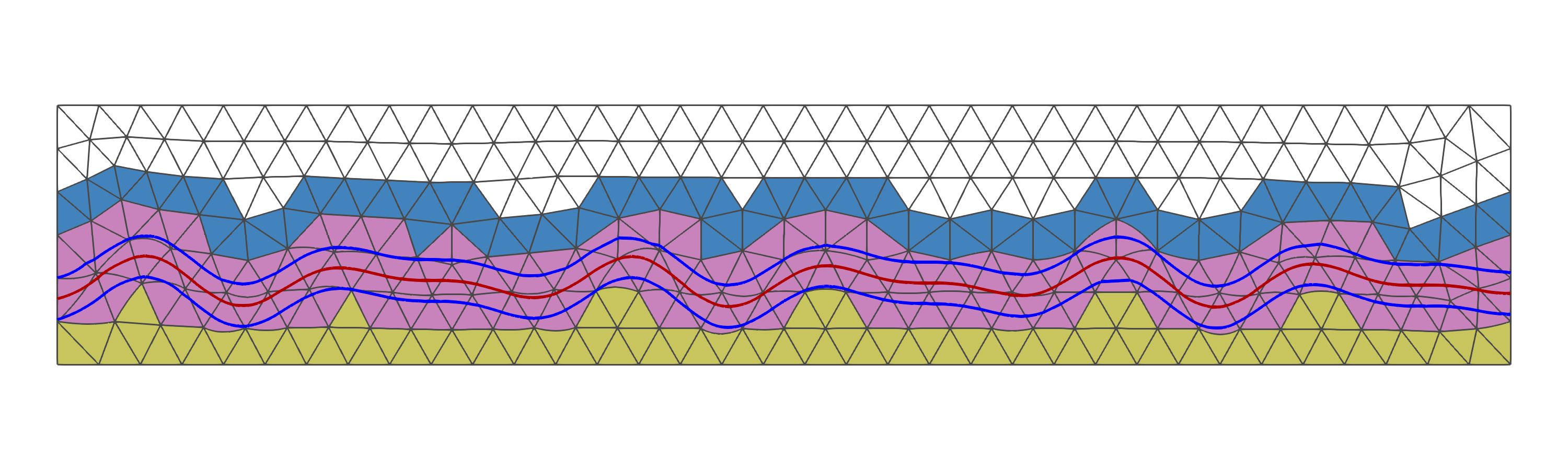}}
  & \multicolumn{2}    {@{}c@{}||@{}}{\includegraphics[width=2cm,trim= 48.30cm 8cm 48.70cm 7.4cm,clip=true]{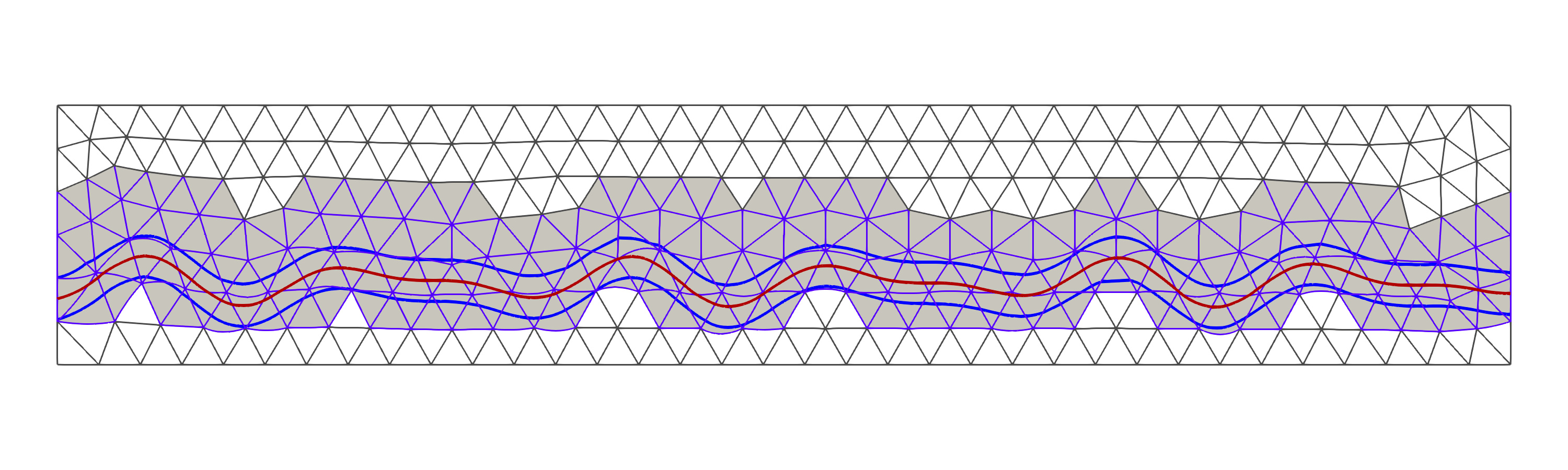}}
  & \multicolumn{2}    {@{}c@{}|@{}}{\includegraphics[width=2cm,trim= 57.16cm 8cm 39.84cm 7.4cm,clip=true]{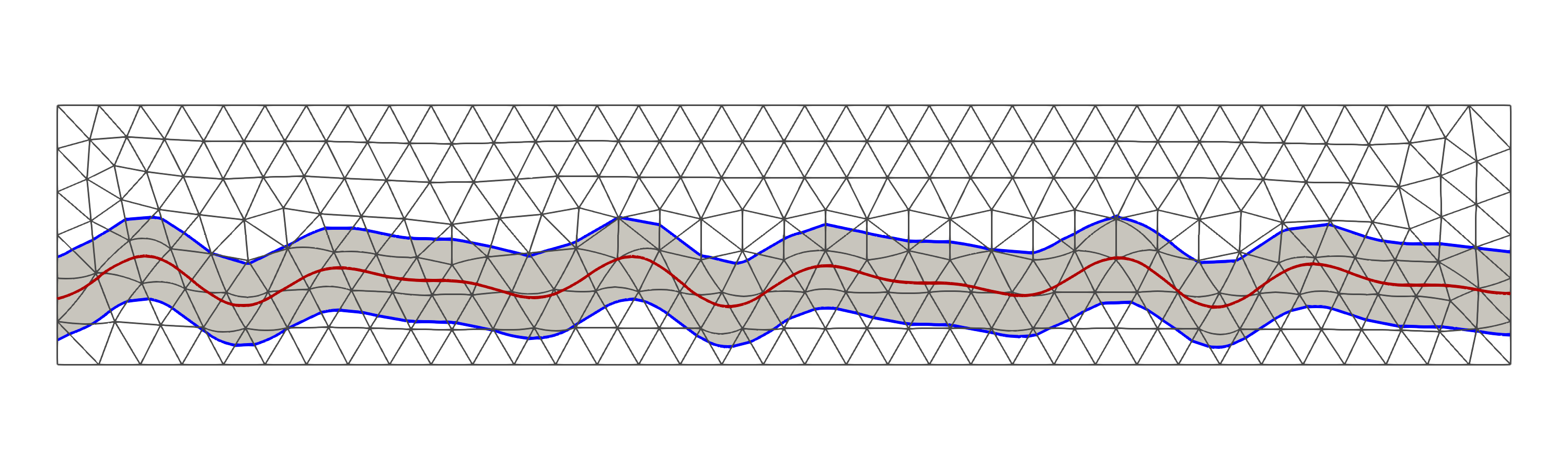}}
  & \multicolumn{2}    {@{}c@{}|@{}}{\includegraphics[width=2cm,trim= 66.02cm 8cm 30.98cm 7.4cm,clip=true]{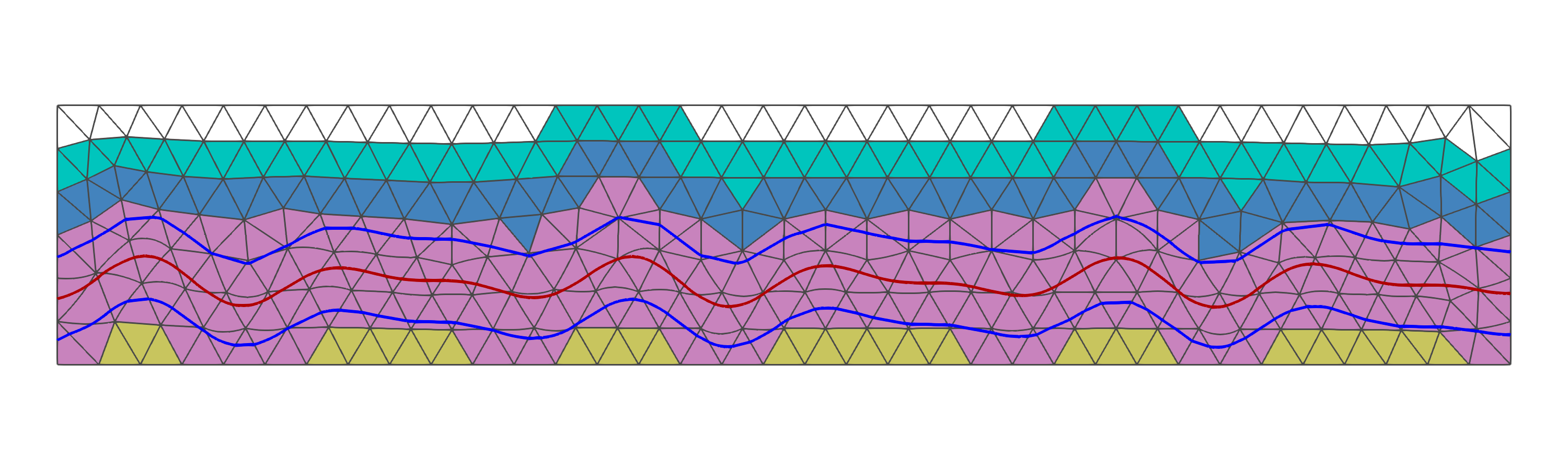}}
  & \multicolumn{2}    {@{}c@{}|@{}}{\includegraphics[width=2cm,trim= 74.88cm 8cm 22.12cm 7.4cm,clip=true]{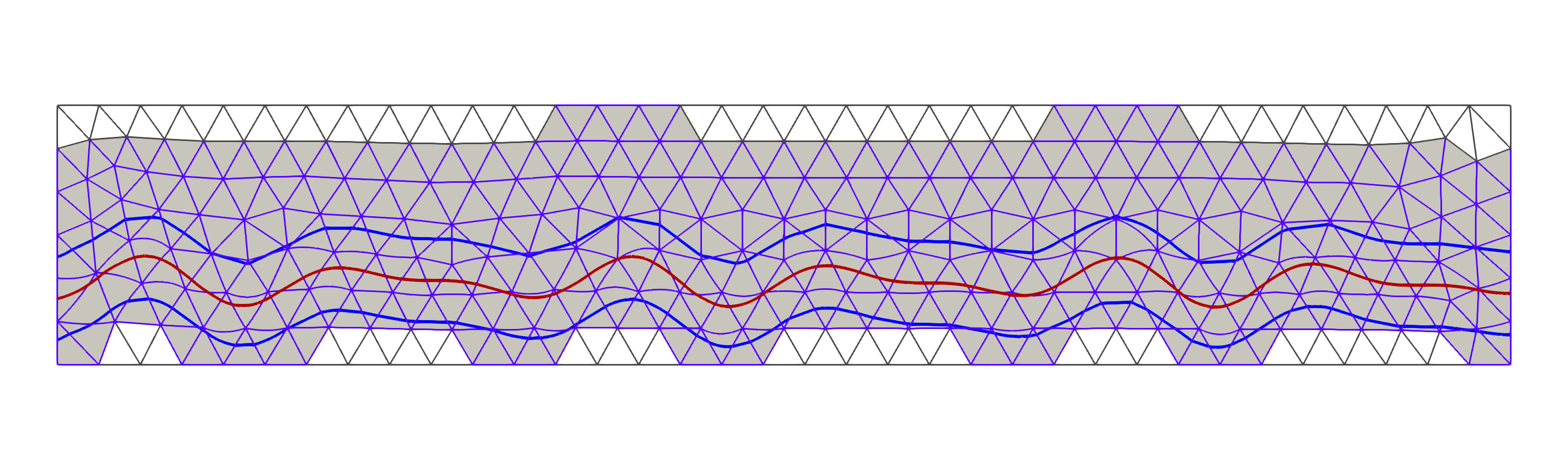}}
  \\[-0.12cm]
  \hline ~ \\[-2.5ex]
    \hspace*{0.14cm}\color{black}- \hspace*{-0.05cm}& $\!:\! \Thn{n}$
  &
    \hspace*{0.14cm}\color{myred}\textbf{-} \hspace*{-0.05cm}& $\!:\!\partial \Omhn{n}$
  &
    $\,\bigcup_{\!{\color{mygreen}\blacktriangle}}$
  &
    $\!:\! \TdnS{0}{n}$
  &
    \hspace*{0.14cm}\color{myblueb}\textbf{-} \hspace*{-0.05cm}& $\!:\!\partial \Sdn{\delta}{n}$
  &
    $\,\bigcup_{\!{\color{mypurple}\blacktriangle}}$
  &
    $\!:\! \TdnS{\delta}{n}$
  &
    $\,\bigcup_{\!{\color{mygray}\blacktriangle}}$
 & $\!:\! \TdnS{\delta,+}{n} $
  &
    \hspace*{0.14cm}\color{myblueb}\textbf{-} \hspace*{-0.05cm}& $\!:\!\partial \Sdn{2\delta}{n}$
  &
    $\,\bigcup_{\!{\color{mypurple}\blacktriangle}}$
  &
    $\!:\!\TdnS{2\delta}{n}$
  &
    $\,\bigcup_{\!{\color{mygray}\blacktriangle}}$
 & $\!:\!\! \TdnS{2\delta,2+}{n}$
  \\
     & 
  &
    ~\color{myblue}$\bullet$ \hspace*{-0.05cm} & $\!:\! \Omhn{n}$
  &
    $\,\bigcup_{\!{\color{mygreen}\blacktriangle}\!\!{\color{myyellow}\blacktriangledown}}$
  &
    $\!:\! \Tdn{0}{n}$
  &
    ~$\color{mygray}\bullet$~ & $\!:\! \Sdn{\delta}{n}$
  &
    $\,\bigcup_{\!{\color{mypurple}\blacktriangle}\!\!{\color{myyellow}\blacktriangledown}}$
  &
    $\!:\! \Tdn{\delta}{n}$
  &
    \hspace*{0.16cm}\color{myviolet}{$\mid$} \hspace*{-0.05cm}& $\!:\!\Frn{1}{n}$
  &
    ~$\color{mygray}\bullet$~ & $\!:\! \Sdn{2\delta}{n}$
  &
    $\,\bigcup_{\!{\color{mypurple}\blacktriangle}\!\!{\color{myyellow}\blacktriangledown}}$
  &
    $\!:\!\! \Tdn{2\delta}{n}$
  &
    \hspace*{0.16cm}\color{myviolet}{$\mid$} \hspace*{-0.05cm} & $\!:\!\Frn{2}{n}$
  \\
     & 
  &
     & 
  &
     & 
  &
     & 
  &
    $\,\bigcup_{\!{\color{mypurple}\blacktriangle}\!\!{\color{myyellow}\blacktriangledown}\!\!{\color{myblue}\blacktriangle}}$
  &
    $\!:\! \Tdn{\delta,+}{n}$
  &
     & 
  &
     & 
  &
    $\,\bigcup_{\!{\color{mypurple}\blacktriangle}\!\!{\color{myyellow}\blacktriangledown}\!\!{\color{myblue}\blacktriangle}}$
  &
    $\!:\!\!\Tdn{2\delta,+}{n}$
  &
     &  
  \\
     & 
  &
    &
  &
     & 
  &
     & 
  &
     & 
  &
     & 
  &
     & 
  &
    $\,\bigcup_{\!{\color{mypurple}\blacktriangle}\!\!{\color{myyellow}\blacktriangledown}\!\!{\color{myblue}\blacktriangle}\!\!{\color{mypetrol}\blacktriangledown}}$
  &
    $\!:\!\! \Tdn{2\delta,2+}{n}$
  &
     &  
  \\
  \hline
\end{tabular}

%\end{document} 
  }
  \end{center}\vspace*{-0.2cm}
  \caption{Sketch of discrete domains and different selections of elements and facets. The first three columns display the mesh, the discrete domain $\Omhn{n}$ and the set of interior and cut elements. The three columns in the center display a strip domain related to an extension by $\d$ and a corresponding element and facet selection while in the last three columns an extension by $2\d$ is considered.}
  \label{fig:domains}
  \end{figure}
First, let us define a \emph{discrete} $\d$-strip for some $\d \in \R_{+}$, 
\begin{equation}\label{def:Sdn}
\S_{\d}^n := \Theta^n \Big(\Omphih{\hat{\phi}_h-\d}{n} \setminus \Omphih{\hat{\phi}_h+\d}{n}\Big)
\end{equation}
and the corresponding part of the set of elements and the corresponding domain
\begin{equation} \label{def:TdnSOdnS}
\T_{\d}^{n,\S} := \{ T \in {\T}_{h}^n \mid \operatorname{meas}_d (T \cap \Sdn{\d}{n} ) > 0 \}, \quad \mathcal{O}_{\d}^{n,\S} := \{\x \in \overline{T}, T \in {\T}_{\d}^{n,\S}\}.
\end{equation}
%\todo[inline]{Put this condition somewhere else? also need r delta smaller than...}
Note that $\T_{0}^{n,\S}$ and $\mathcal{O}_{0}^{n,\S}$ denote the set of all \emph{cut elements} and the corresponding domain, respectively, i.e., the elements that are cut by the discrete boundary $\partial \Om_h$.
For the discrete extension of the domain that includes the domain interior, 
the \emph{active} part of the mesh and its domain we have
\begin{equation} \label{def:OmdnTdnOdn}
  \Om_{\d}^n \!:= \Theta^n \!\big(\Omphih{\hat{\phi}_h-\d}{n}\big)\!,
  {\T}_{\d}^n \!:= \{ T \in {\T}_{h}^n\!\mid\!\operatorname{meas}_d(T \cap \Om_{\d}^n)\!>\!0\}, 
\mathcal{O}_{\d}^n := \!\{\x \!\in \overline{T}, T \in {\T}_{\d}^n \}.
\end{equation}
Corresponding to $\Tdn{\d}{n}$ we define the time-dependent finite element spaces on the \emph{active meshes} as continuous, piecewise mapped polynomials of degree $k$:
%\todo[inline]{should be here "the finite element space of the active part" instead of "the active part of the finite element space"?} 
\begin{equation}\label{def:Vdn}
\V_{\d}^n := \Vhn{n} |_{\Odn{\d}{n}} = \{v_h \in C(\Odn{\d}{n}): v_h|_T \in \P_k(T)\circ \Theta^{-n}, \forall T \in {\T}_{\d}^n\}. 
\end{equation}
%Please note that these domains, meshes and spaces depend on $\d$ that represents the discrete sense implicitly, we hence hide the subscript $h$. 
We furthermore add a subscript ``$+$'' to expand a set of elements or domain by all neighboring elements\footnote{An element is considered a neighbor if both share a vertex}, e.g., the neighboring elements in addition to the cut elements are denoted by $\T_{0,+}^{n}$. This extension can also be stacked $r$ times\footnote{with $r$ a small integer}, e.g., $\T_{\d,2+}^{n}:=\T_{\d,++}^{n}$ and $\O_{\d,3+}^{n}:=\O_{\d,+++}^{n}$. Obviously, there holds
\begin{equation}
\operatorname{dist}\left(\O_{r\d,r+}^n,\Orn{0}{n}\right) \gtrsim r(\d + h). 
\end{equation}
For notational simplicity, the abbreviations \hypertarget{def:Trn}{$\T^n_{r} := \T^n_{r\d,r+}$}, 
\hypertarget{def:TrnS}{$\T^{n,\S}_{r} := \T^{n,\S}_{r\d,r+}$}, 
\hypertarget{def:Orn}{$\O^n_{r} := \O^n_{r\d,r+}$}, 
\hypertarget{def:OrnS}{$\O^{n,\S}_{r} := \O^{n,\S}_{r\d,r+}$}, 
and 
\hypertarget{def:Vrn}{$\V^n_{r} := \V^n_{r\d,r+}$ for $r \in \N$} will be frequently used below. 
We note that the introduced notation implies the following identities: 
\begin{align*}
  \Omhn{n}=\Omdn{\d}{n}\big|_{\d=0}&=\Om_{0}^n, & 
  \T_{\d}^n\big|_{\d=0}=\T_{r}^n\big|_{r=0}&=\T_{0}^n, & 
    \O_{\d}^n\big|_{\d=0}=\O_{r}^n\big|_{r=0}&=\O_{0}^n 
  \\
  \Thn{n}=\T_{\d}^n\big|_{\d\to\infty}&=\T_{\infty}^n, &
  \Vhn{n}=\V_{\d}^n\big|_{\d\to\infty}&=\V_{\infty}^n
\end{align*} 
Next, from $\F_h^n$, the set of all facets in the mesh $\Thn{n}$, we introduce a set of \emph{active facets} that is later on used for stabilization and extension purposes. To this end, we mark all facets between elements in the $r \d$ strip and the interior:
\begin{equation} \label{eq:facetset}
\F_{r}^n := \{\overline{T}_1 \cap \overline{T}_2 \mid T_1 \in \T_{r}^n, T_2 \in \T_{r}^{n,\S}, T_1 \neq T_2 \}. 
\end{equation}
Note that this selection of facets connects the domain interior of $\Omhn{n}$ with $\O_r^n$, i.e., the region obtained by applying an extension by $r \d$ plus $r$ additional element layers. 

We further introduce a \emph{patch} $\omega(\cdot): \B \cup \F_h^n \cup \Thn{n} \to \Thn{n}$ that maps a point, a facet, or an element to a set of neighboring elements
\begin{subequations} 
	\begin{align}
		\omega(x) &:= \{\cup_{T \in \Thn{n}} T, ~ x \in \overline{T} \}  &&\text{for a point} ~x \in \B, \label{eq:nodepatch} \\
		\omega(F) &:= \{\cup_{T \in \Thn{n}} T, ~ F \subset \overline{T} \}  &&\text{for a facet} ~F \in \F_h^n, \label{eq:facetpatch} \\
		\omega(T) &:= \{\cup_{T' \in \Thn{n}} T' \mid \overline{T} \cap \overline{T}' \ne \emptyset \}  &&\text{for an element} ~T \in \Thn{n}. \label{eq:elementpatch}
	\end{align}
\end{subequations}
\hypertarget{def:hom}{Similarly we use the notation $\hat{\omega}(\cdot)$ for patches on the undeformed mesh $\Th$ where the neighboring elements are picked correspondingly from $\Th$.}

We conclude this subsection with the following definition: 
\begin{definition}[Trivial finite element extension]
We identify discrete functions on restricted meshes with functions on the whole mesh by setting all degrees of freedom outside the restriction to zero, s.t. there holds for instance $\V_{\d}^n \subset \V_{h}^n$ for any $\d \geq 0$.
\end{definition}

\subsection{Properties of the parametric mapping as a function in space} \label{sec:prop param1}
Let $\phi(\x,t)$ be a given function which is smooth in space and Lipschitz-continuous in time at least in the vicinity of its zero level.
The time-dependent mappings $\Theta^n: \B \to \B$ are constructed for each $n=0,..,N$ based on the strategies for stationary domains described in \cite{lehrenfeld2015cmame,LR_IMAJNA_2018}.
We only summarize the most important features.
The mapping acts mainly on \emph{cut elements}, i.e. on $\OrnS{0}{n}$, where the construction ensures that the image of the zero level of $\hat{\phi}_h^n$ under the mapping is (in a higher-order sense) close to the zero level of $\phi_h^n$.
Because the piecewise linearized level set function $\hat{\phi}_h^n$ is already exact on vertices and second order accurate elsewhere, the mapping is the identity on vertices and $\O(h^2)$ small on cut elements.
On elements neighboring to cut elements, $\OrnS{0,+}{n} \setminus \OrnS{0}{n}$, a transition to the identity is realized so that overall the mapping $\Theta^n$ is small and local.
\hypertarget{def:Psi}{By $\Psi^n$ an \emph{ideal} mapping is denoted that maps the zero level of $\hat{\phi}_h^n$ onto $\Gamma^n$ exactly.}
As $\Theta^n$ the ideal mapping $\Psi^n$ only deviates from the identity in $\OrnS{0,+}{n}$.
\extendedonly{Note that although not reflected in the notation $\Theta^n$ and $\Psi^n$ depend on the triangulation $\Th$.} 

We summarize the accuracy of the mapping in the following lemma: 
% \todo[inline]{introduce notation for cut elements ($+$/$-$ transition elements) and remainder}
\begin{lemma} \label{propertiesdh}
Let $n \in \{0,..,N\}$ be fixed and $\OrnS{0,+}{n}$ be the domain of cut elements and direct neighbors. For $h$ sufficiently small, there holds
\begin{subequations}
  \begin{align}
    & \Theta^n(\x)  =\x, \quad \text{for \ $\x=\x_V$ \text{\rm vertex in} $\Th$ or $\x \in \B\setminus \OrnS{0,+}{n}$}, \label{resd3} \\
    & \| \Theta^n - \id \|_{\infty,\B} \lesssim h^2, \quad ~\| D \Theta^n(\x) - \operatorname{I}\|_{\infty,\B} \lesssim h, \label{resd4} \\
    & \| \Theta^n - \Psin{n} \|_{\infty,\B} + h \| D (\Theta^n - \Psin{n}) \|_{\infty,\B} \lesssim h^{q+1}, \\
    & \text{where the latter implies} ~\operatorname{dist}(\partial \Om(t),\partial \Om_h(t)) \lesssim h^{q+1}. 
  \end{align}
\end{subequations}
\end{lemma}
\begin{proof} See \cite[Lemmas 3.4, 3.6 and 3.7]{LR_IMAJNA_2018}. \end{proof}

\begin{figure}
  \begin{center}
    \includegraphics[trim=0.2cm 0.0cm 0.2cm 0.0cm, clip=true, width=0.5\textwidth]{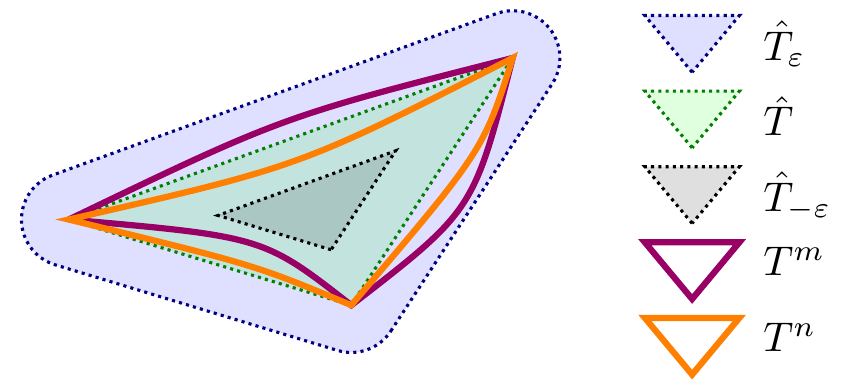}\vspace*{-0.3cm} 
  \end{center}
  \caption{Sketch of domains involved in \Cref{lemma:dmax}.}
  \label{fig:teps}
\end{figure}

Next, we characterize inclusion relations between inflated and deflated elements: 
\begin{lemma} \label{lemma:dmax}
  Let $\hat{T} \in \Th$, we define the (slightly) inflated version $\hat{T}_{\eps} := \{ x \in \B \mid \operatorname{dist}(x,\partial \hat{T}) \leq \eps \}$%
\footnote{Note that $\eps$ is not to be confused with $\epsilon$ introduced in \Cref{sec:model} for the domain extension.} 
  and the (slightly) deflated version $\hat{T}_{-\eps} := \{ x \in \hat{T} \mid \operatorname{dist}(x,\partial \hat{T}) \ge \eps \}$ for some $\eps > 0$. Further, for $m=0,..,N$, let $\Theta_{T}^{m*}: \hat{T}_{\eps} \to \R^d$ denotes the canonical extension of the polynomial function $\Theta^{m}|_{\hat{T}}$ to $\hat{T}_{\eps}$.
	For $h$ sufficiently small there is $c_{L\ref{lemma:dmax}} > 0$ (independent of $h$, $\hat{T}$ and $m,~n$) such that with $\eps = c_{L\ref{lemma:dmax}} h^2$ the following inclusion properties hold for $m,n=0,\dots,N$ and $T^m_{\pm \eps} := \Theta_{T}^{m*}(\hat{T}_{\pm \eps})$, cf. \Cref{fig:teps}
	\begin{align}
	T^n \subset \hat{T}_{\frac{\eps}{2}} \subset T^m_{\eps},~~ \hat{T}_{-\eps} \subset T^n \cap T^m \text{ with } \meas_d\left( (T^n \cup T^m) \setminus \hat{T}_{-\eps} \right) \lesssim h^{d+1}.
	\end{align}
\end{lemma}

\begin{proof}
Due to norm equivalences on the space of polynomials on a reference element and its extension, and standard scaling arguments, we have with $\eps \simeq h^2 < 1$ that $\| \Theta_{T}^{m*} - \operatorname{id} \|_{\infty,\hat{T}_\eps} \lesssim \| \Theta^{m}|_{\hat{T}} - \operatorname{id} \|_{\infty,\hat{T}} \lesssim h^2$ and
$\| D \Theta_{T}^{m*} - I \|_{\infty,\hat{T}_\eps} \lesssim \| D \Theta^{m}|_{\hat{T}} - I \|_{\infty,\hat{T}} \lesssim h$. Hence, the properties of \eqref{resd4} carry over to the extended function $\Theta_{T}^{m*}$ which ensures the inclusion properties and the measure of the $\eps$-band, i.e. $(T^n \cup T^m) \setminus \hat{T}_{-\eps}$, with a bound $\eps h^{d-1}$ where $\eps \simeq h^2$. 
\end{proof}

A direct conclusion of the two previous lemmas and standard scaling arguments is that for $\hat{T} \in \T_h,~T^n = \Theta_T^n(\hat{T})$ and $\hat{v} \in \mathcal{P}^k(\hat{T})$ there hold the following equivalences
\begin{alignat}{4}
  h^{\frac{d}{2}} \Vert \hat{v} \circ  \Theta_T^{-n} \Vert_{\L^{\infty}(T^n_{\hphantom{-\eps}})} &\simeq & \Vert \hat{v} \circ  \Theta_T^{-n} \Vert_{T^n_{\hphantom{-\eps}}} &\simeq &\Vert \hat{v} \Vert_{\hat{T}_{\hphantom{-\eps}}} &\simeq& \Vert \hat{v} \Vert_{\L^\infty(\hat{T}_{\hphantom{-\eps}})}\hphantom{.} \nonumber \\[-1ex]
  \rotatebox[]{270}{$\simeq$}\hspace*{2cm} && \rotatebox[]{270}{$\simeq$}\hspace*{1.25cm} && \rotatebox[]{270}{$\simeq$}\hspace*{0.65cm} && \rotatebox[]{270}{$\simeq$}\hspace*{1.5cm} \label{eq:normequiv}\\[-1ex]
  h^{\frac{d}{2}} \Vert \hat{v} \circ \Theta_T^{-n} \Vert_{\L^{\infty}(\Tneps{n}{\pm \eps})} &\simeq& \Vert \hat{v} \circ \Theta_T^{-n} \Vert_{\Tneps{n}{\pm \eps}} &\simeq& \Vert \hat{v} \Vert_{\Thneps{}{\pm \eps}} &\simeq& \Vert \hat{v} \Vert_{\L^\infty(\Thneps{}{\pm \eps})}. \nonumber
\end{alignat}

\subsection{Properties of the parametric mapping as a function in time}\label{sec:prop param2}

As mentioned above for the isoparametric approximation of the geometry we have slightly different meshes between consecutive time steps. To do proper time stepping in such an approach we need to project solutions from one deformed mesh to another. The details about this projection are discussed in the subsequent section. As such a projection has to be applied in every time step, one may expect projection errors accumulating with the number of time steps $N \sim \frac{1}{\Delta t}$. To be able to show (in the analysis section) that this is \emph{not} the case we take a careful look at how the deformation depends on time. More specifically, we characterize where and when the deformation depends continuously on time and where and when not. This will then be exploited when analyzing the accumulation of the projection errors later in the analysis. 

Based on the properties discussed above, for a fixed time $t$, there are three different types of mapped elements: cut elements, transition elements (neighboring to cut elements) and undeformed elements. 
The cut elements are transformed based on the desired property $\hat{\phi}_h(\cdot,t) \approx \phi_h(\cdot,t) \circ \Theta(\cdot,t)$, while the undeformed elements, sufficiently far away from cut elements, have $\Theta(\cdot,t) = \id$. The remainders are transition elements which realize a proper blending between these two zones, cf. \Cref{fig:sketch:eltypes} for a sketch in the spatially one-dimensional situation.

\begin{figure}
\centering 
  \includegraphics[width=0.95\textwidth]{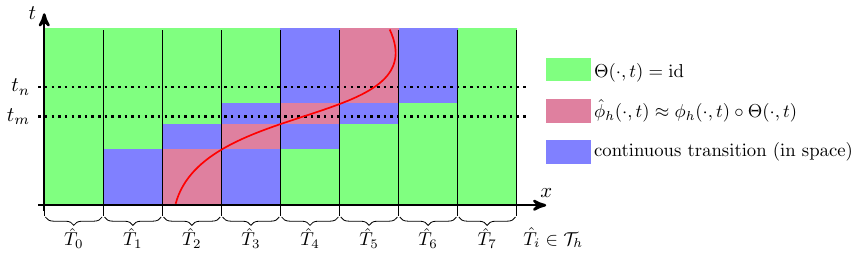} \vspace*{-0.25cm}
\caption{Sketch of different regions for the mesh deformation. At a fixed time $t$ an element is in exactly one of the three classes: cut (purple), transition (blue) or undeformed (green). Between two time instances $t_m < t_n$, for each fixed element we distinguish two situations: the type of the element and all its direct neighbors stay in the same class for all $t \in [t_m,t_n]$ or not. 
}
\label{fig:sketch:eltypes}
\end{figure}

\shortonly{
  Now, we consider a fixed element $\hat{T} \in \Th$ at two time instances $t_m < t_n$ and distinguish two cases. Either $\hat{T}$ or all neighboring elements remain of the same type of deformed elements. In this case the change in the deformation is Lipschitz-continuous and there holds 
  \begin{equation} \label{eq:const}
		\| \Theta^m - \Theta^n \|_{\infty,\B} \leq |t_n - t_m|; 
  \end{equation}
  If either $\hat{T}$ or one of its neighboring elements changes the type, the deformation will in general no longer be Lipschitz-continuous in time, and we fall back to the smallness of the deformation 
  \begin{equation} \label{eq:change}
		\| \Theta^m - \Theta^n \|_{\infty,\B} \leq \| \Theta^m - \id \|_{\infty,\B} + \| \id - \Theta^n \|_{\infty,\B} \lesssim h^2. 
  \end{equation}

  For a fixed time step size $\Delta t$, we define to every $\hat{T} \in \Th$
  the integer $N_D^T$ that counts the number of occasions where $\hat{T}$ or one of its neighbors changes type within a time step. Taking the maximum over the mesh, we further define $N_D = \max_{\hat{T}\in\Th}\{N_D^T\}$.
}

\begin{assumption} \label{ass:change}
  In the remainder we assume that for a fixed time interval $(0,\finaltime]$ and a fixed computational mesh, the number $N_D$ is bounded independent of the partition of time, but only depends on the motion of the domain. 
\end{assumption}

%{\color{gray}
\subsection{Transfer operator between meshes at different time steps}\label{sec:transfer}
As mentioned above we have slightly different meshes between consecutive time steps. 
We therefore have to specify a transfer operator of finite element functions from one mesh to another. 
To this end, below in \Cref{sec:projection} we design a projection operator $\Pi^n: \Vhn{n-1} \to \Vhn{n}$. 
We note that for locality and computational efficiency we choose a projection operator that deviates from a direct $L^2$ projection.
% which however would have also been possible in view of the analysis later on. 
%}

\section{Definition of the stabilized Eulerian finite element method}
\label{sec:method}
Based on suitably adapted versions of the method of lines we introduce a full discretization. For ease of presentation we start with the low-order discretization in space and time, i.e., a piecewise linear finite element space with an implicit Euler time stepping in \Cref{sec:fulllinear}. This allows us to present the spatial discretization with the involved stabilization and extension in its simplest configuration. 
The development to higher-order approximation in space is then tackled in \Cref{sec:fullhospace}, which is followed by the extension also to higher-order approximation in time in \Cref{sec:fullhotime}. 

\subsection{A fully discrete low order prototypical formulation}\label{sec:fulllinear}
Let $k=q=1$ in which case $\phi_h=\hat{\phi}_h$ and $\Theta^n=\operatorname{id},~\Vhn{n}=\Vh,~\Thn{n}=\Th,~\Omhn{n}=\Omphi{\phi_h}{n}$ for all $n=0,..,N$. 
Each step in the low-order version as introduced in \cite{LO_ESAIM_2019} consists of three parts: 
(i) the approximation of the partial time derivative through the finite difference stencil $\tfrac{u_h^n-u_h^{n-1}}{\Delta t}$; 
(ii) the spatially discrete operator $b_h^n$ for convection and diffusion; 
(iii) a ghost-penalty-type operator $s_r^n$ for extension and stabilization. The weak form reads: \\
Find $u_h^n \in \Vrn{1}{n},\; n = 1,...,N$ for a given $u_h^0 \in \Vrn{1}{0}$, such that
\begin{align}
\int_{\Omhn{n}} \frac{u_h^n - u_h^{n-1}}{\Delta t} v_h \; dx + \bh{n}(u_h^n,v_h) + \gamma \srn{1}{n}(u_h^n,v_h) = f_h^n(v_h), \qquad \forall v_h \in \Vrn{1}{n}. 
\end{align}
Here, the bilinear form for convection and diffusion makes use of a skew-symmetrized form for the convection part
\begin{align} 
b_h^n(u_h,v_h) :=& \int_{\Omhn{n}} \nu \nabla u_h \cdot \nabla v_h \; dx + \frac{1}{2} \int_{\Omhn{n}} \left( (\w^e \cdot \nabla u_h)v_h - (\w^e \cdot \nabla v_h)u_h \right) \; dx  \nonumber \\ 
&+ \frac{1}{2} \int_{\Omhn{n}} (\nabla \cdot \w^e) u_h v_h \; dx +  \frac{1}{2} \int_{\Ghn{n}} (\w^e \cdot \n) u_h v_h \; ds, \quad \forall u_h,v_h \in \H^1(\Omhn{n}). \label{def:bh}
\end{align}
\hypertarget{def:ext}{where $(\cdot)^e$ denotes a smooth extension from $\Omn{n}$ to $\Omhn{n}$ which we assume to exist.}

The bilinear form $\srn{1}{n}(\cdot,\cdot)$ for extension and stabilization is applied with a parameter $\gamma(h,\d)$, which is yet to be defined below. 
Here, this term uses the \emph{ghost penalty} stabilization mechanism \cite{B10} where different versions to realize the same effect exists, cf. \cite[Section 4.3]{LO_ESAIM_2019}. 
We make use of the \emph{direct} or \emph{volumetric jump} formulation introduced in \cite{preussmaster} which takes the form (with $r=1$ for the implicit Euler)
\begin{equation} \label{eq:ghostpenalty}
s_{r}^{n}(u_h,v_h) := \sum_{F^n \in \Frn{r}{n}} s_{F}^n (u_h,v_h) \text{ with } s_{F}^n (u_h,v_h):= \frac{1}{h^2} \int_{\omF{F^n}} \!\!\!\!\!\!\!\! (u_1 - u_2) (v_1 - v_2) dx,
\end{equation}
where $\omF{F^n}$ is the patch of elements around $F^n$, cf. \Cref{eq:facetpatch}, and
$u_i,~v_i,~i=1,2$ are canonical extensions of \emph{mapped} polynomials, i.e. $u_i = \big( \mathcal{E}^\P  (u_h|_{T^n_i} \circ \Theta_{T_i}^n) \big) \circ \Theta_{T_i}^{-n*}$ (and similarly for $v_i$)
with $\Theta_{T_i}^{n*} = \mathcal{E}^\P(\Theta_{T_i}^n)$ where \hypertarget{def:ep}{$\mathcal{E}^\P: \P_k(\hat{T}_i) \to \P_k(\mathbb{R}^d), ~\hat{T}_i = \Theta_{T_i}^{-n}(T^n_i)$ is the canonical extension of a polynomial to the whole
space}\footnote{Note that $u_h|_{T^n_i},~i=1,2$ are only \emph{mapped} polynomials but $u_h|_{T^n_1} \circ \Theta_{T^n_i}^n$ are standard ones.}.  
The ghost penalty is responsible for two effects. On the one hand, it stabilizes the formulation to achieve robustness w.r.t. the position of the geometry within the elements. 
On the other hand, it implicitly realizes a discrete extension from $\Omhn{n}$ to $\Orn{1}{n} \supset \Omdn{\d}{n}$. 
\hypertarget{cond:delta}{This extension is required for instance to make $u_h^{n}$ well-defined for the domain $\Omhn{n+1} \subset \Omdn{\d}{n} \subset \Orn{1}{n}$. 
To this end, we make the following assumption on $\d$:
\begin{equation}\label{ass:delta}
\d \geq \Delta t \Vert \mathbf{w} \Vert_{\L^\infty((0,\finaltime),\L^\infty(\B))}
\end{equation}
}
Let us note that we take a global (in space and time) choice for $\dd$ to keep the presentation feasible, but a more localized definition of an extension region would easily be possible by considering different values for $\dd$ in different time steps and different spatial regions. 
Next, note that the solution is extended away from $\Omhn{n}$ by at least one layer of elements, i.e., by at least a distance proportional to $h$ so that for a constant $c > 0$, depending only on the shape regularity there holds
\begin{equation} \label{eq:distnnpo}
  \operatorname{dist}(\partial \Orn{1}{n}, \partial \Omhn{n}) \geq \dd + c h \gtrsim \Delta t + h.
\end{equation}  
With $\Omn{n+1} \subset \Omn{n}$ and $\dist(\partial \Omhn{n},\partial \Omn{n}) \lesssim h^2$ for all $n = 0, .., N$ we can guarantee the inclusion $\Omhn{n+1} \subset \Omdn{\dd}{n} \subset \Orn{1}{n}$ (for sufficiently small $h$).
%\todo[inline]{Shall we introduce the assumption of $\gamma$ at this position?}
%
\hypertarget{def:fhn}{The linear functional $f_h^n$ is simply
$ f_h^n(v_h) := \int_{\Omhn{n}} g v_h dx$ for $v_h \in \H^1(\Omhn{n})$}. %
%\todo[inline]{Discuss unique solvability only briefly}
%
If the time step is bounded by
\begin{equation} \label{eq:timestep}
\dt < \xi^{-1} := 2 \left(\|\div(\w^{\ext})\|_{\L^\i(\Omhn{n})} + \nu + c_{\Om_h}^2 \|\w^{\ext} \cdot \n\|_{\L^\i(\Omhn{n})} / 4\nu\right)^{-1}, 
\end{equation}
where $c_{\Om_h}$ is the constant of the multiplicative trace inequality, $\bh{n}$ has a lower bound
\begin{equation} \label{eq:bilinear}
	\bh{n}(v,v) \ge \frac{\nu}{2} \|\nabla v\|_{\Omhn{n}}^2 - \xxi \|v\|_{\Omhn{n}}^2. 
\end{equation}
\extendedonly{
  Here we assume that the (extended) velocity $\w^{\ext}$ has bounded divergence and normal flux. For details cf. \cite[Lemma 3.1 and Remark 4.1]{LO_ESAIM_2019}.
}
The coercivity of the overall l.h.s. bilinear form on $\Vrn{1}{n}$ w.r.t. the norm 
\begin{equation} \label{eq:norm}
|\!|\!| v |\!|\!|_n := \left(\frac{\nu}{2} \|\nabla v\|_{\Omhn{n}}^2 + \|v\|_{\Omhn{n}}^2 + \gamma \srn{r}{n}(v,v)\right)^\frac{1}{2} 
\end{equation}
guarantees the unique solvability based on the Lax-Milgram theorem.

Before specifying the parameter $\gamma$ we introduce the following assumption. 
\begin{assumption} \label{ass:paths}
Let $\T_{r}^{n,\S+}$ denote the subset of $\T_{r}^{n,\S}$ where for at least one point $\x \in T$ there holds $\phi_h(\x) > 0$. To every element in $\T_{r}^{n,\S+}$ we require an element in $\T_{r}^{n}\setminus\T_{r}^{n,\S+}$ that can be reached by repeatedly passing through facets in $\Frn{r}{n}$. We assume that the number of facets passed through during this path is bounded by $K \lesssim (1 + \frac{\dd}{h})$. Further, every ``interior'' element in the active domain, i.e. $T \in \T_{r}^{n}\setminus\T_{r}^{n,\S+}$, provides at most $M$ paths in which it serves as the terminal element of such paths, where $M$ is a number that is bounded independently of $h$ and $\dt$. 
\end{assumption}
\hypertarget{def:cgamma}{
With this definition of $\K$, we specify -- following \cite[Section 4.4]{LO_ESAIM_2019} -- $\gamma(h,\d) = c_\gamma K \lesssim 1 + \d/h$ for a constant $c_\gamma$ independet of $\dt$ and $h$. This completes the fully discrete low-order scheme. 
}
\new{   
\begin{remark}\label{rem:nu}  
  From \eqref{eq:timestep} we can already see that the analysis of the method relies on a diffusion coefficient $\nu$ which does not become arbitrarily small. This already holds for the semi-discrete discretization which is only discrete in time, cf. \cite{LO_ESAIM_2019}. Furthermore, the numerical studies for the slightly more difficult problem and discretization in \cite{vWRL_ARXIV_2020} suggest that the method is indeed not robust for vanishing $\nu$. Hence, we will assume in the remainder of this manuscript that $\nu$ is bounded from below by a constant of size $\mathcal{O} (1)$.  
\end{remark}
}

\subsection{Higher order space discretization} \label{sec:fullhospace}
The discretization above can be advanced trivially to higher order of accuracy in space if exact geometry handling is assumed or sufficiently accurate quadrature on $\Omphi{\phi_h}{n}$ is given. 
As the former is typically not realistic and the latter is hard to guarantee, we consider the application of the isoparametric mapping $\Theta^n$ to achieve higher order of geometrical accuracy. However, with the time-dependent deformation of the mesh, which implies $\Thn{n-1} \neq \Thn{n}$ and hence $u_h^{n-1} \not\in \Vhn{n}$, the need to apply a few adaptations arises.
\hypertarget{def:proj}{We make use of the \emph{consecutive} transfer operator $\Pi^n: \Vhn{n-1} \to \Vhn{n} $}, introduced in more detail in \Cref{sec:locproj}, to project initial data $u_h^{n-1}$ from one timestep to the next, and then the weak form reads: \\
Find $u_h^n \in \Vrn{1}{n},\; n = 1,...,N$ for a given $u_h^0 \in \Vrn{1}{0}$, such that
\begin{align}
\int_{\Omhn{n}} \frac{u_h^n - \Pin{n} u_h^{n-1}}{\Delta t} v_h ~dx + \bh{n}(u_h^n,v_h) + \gamma \srn{1}{n}(u_h^n,v_h) = \fhn{n}(v_h), \qquad \forall v_h \in \Vrn{1}{n}. 
\end{align}

\subsection{High order time discretization based on BDF schemes} \label{sec:fullhotime}
For high order approximation in time we apply
BDF schemes to the time derivative. We introduce the notation 
$\partial_{\dt}^{r}(...)$ for the BDF time stencils (for $r=1,2,3$): 
\begin{subequations} \label{eq:stencils}
\begin{align} \label{eq:stencil:BDF1}
  \partial_{\dt}(u_h^n, u_h^{n-1}) &:= \frac{u_h^n - u_h^{n-1}}{\Delta t}, &&  r=1; \\
  \label{eq:stencil:BDF2}
\partial_{\dt}^{2}(u_h^n, u_h^{n-1}, u_h^{n-2}) &:= \frac{3 u_h^n - 4 u_h^{n-1} + u_h^{n-2}}{2 \Delta t}, &&  r=2; \\
\label{eq:stencil:BDF3}
\partial_{\dt}^{3}(u_h^n, u_h^{n-1}, u_h^{n-2}, u_h^{n-3}) &:= \frac{11 u_h^n - 18  u_h^{n-1} + 9   u_h^{n-2} - 2  u_h^{n-3}}{6 \Delta t}, &&  r=3. 
\end{align}
\end{subequations}
In order to apply the stencils we need to take advantage of the projection operator. To do this across several time steps we define the consecutive application of projection operators over all intermediate time steps
% \footnote{To be precise the projection $\Pi^$ will have slightly modified range and domain of definition: $\Pi^k: V_{h,}$}
\begin{equation}\label{eq:defprojr}
\Pi_{n-r}^n: \Vhn{n-r} \to \Vhn{n}, \qquad v \mapsto \Pin{n} ~ \Pin{n-1} \cdots \Pin{n-r+1} v.
\end{equation}
Then, the weak form reads: \\
Find $u_h^n \in \Vrn{r}{n}, ~n = r,...,N$ for given $u_h^0 \in \Vrn{r}{0}$,..., $u_h^{r-1} \in \Vrn{r}{r-1}$, such that for $\forall v_h \in \Vrn{r}{n}$
% \begin{subequations} 
\begin{equation}\label{eq:disceq}
  \int_{\Omhn{n}} \dtBDFr{r}(u_h^n,  ..., \Pimn{n-r}{n} u_h^{n-r})  v_h ~dx  + \bh{n}(u_h^n,v_h) + \gamma \srn{r}{n}(u_h^n,v_h) = \fhn{n}(v_h).% \qquad \forall v_h \in \V_{h,r \dd}^n. 
\end{equation}
Note that the stabilization bilinear form $\srn{r}{n}$ now expands to a larger region extended by $r\dd$ distance plus $r$ additional element layers. 
\begin{remark}
  \label{rmk trans opr}
In an implementation it is not necessary to apply the whole chain of the projection $\Pimn{n-l}{n}$ for $1 < l \leq r$ as the terms involving $\Pin{m},~ m < n$ will be needed in previous time steps already and can be reused, i.e., there is actually only the projection $\Pin{n}$ to be evaluated at each time step (on possibly several terms though). 
\end{remark}

\section{Efficient higher-order projection for isoparametric unfitted FEM}~~
\label{sec:projection}
In this section we discuss the operator $\Pin{n}$ between consecutive time levels in detail. 

\subsection{Definition of a projection based on essentially local operations} \label{sec:locproj}
Let $v_{\Thn{m}} \in \Vhn{m}$, $m=n-1$ be a discrete function\footnote{Remember that we identify functions on restricted domain, e.g., in $\V_{h,\dd}^m$ with their finite element extensions by setting the remaining degrees of freedom to zero.} w.r.t. the mesh $\Thn{m}$. We aim to approximate it on $\Thn{n}$ % $\Omhn{n}\subset \Om_{h,\dd}^{m}$
with $v_{\Thn{n}} := \Pimn{m}{n} v_{\Thn{m}} \in \Vhn{n}$, i.e., a discrete function w.r.t. the (slightly different) mesh $\T_{h}^n$. This projection is achieved in three steps: 
\begin{enumerate}[label=(\roman*)]
\item Firstly, by exploiting that $v_{T^m}:=v_{\Thn{m}}|_{T^m}$, the restriction of $v_{\Thn{m}}$ to an element $T^m \in \T_{h}^m$, is smooth, we define an extension $v_{T^m}^{*} $ of $v_{T^m}$ to a small neighborhood $T^m_\eps$ of $T^m$ with $T^n \subset T^m_\eps$, such that $v_{T^m}^{*} \in \C^{\infty}(T^m_\eps)$;
\item Secondly, we project these extensions into $\bigoplus_{T^n\in\T_{h}^n}\Vhn{n}|_{T^n}$, i.e., the discontinuous (across element interfaces) version of $\Vhn{n}$, yielding $\tilde{v}_{\Thn{n}}$; 
\item Thirdly, we apply an Oswald-type interpolation of $\tilde{v}_{\Thn{n}}$ to obtain $v_{\Thn{n}} \in \Vhn{n}$.
\end{enumerate}
The first two steps are completely element-local and allowed for a trivial parallelization, especially as the access to neighboring elements is not required, whereas the third step is a high efficent vector operation (averaging). This is in contrast to an only seemingly simpler approach such as a global $\L^2$ projection.% which would involve non-local operations. 

\subsubsection{Element-local extensions}\label{sec:ext}
For an undeformed element $\hat{T} \in \Th$ we introduce the notation $T^i = \Theta_T^i(\hat{T}) \in \Thn{i}$ with $\Theta_T^i := \Theta^i|_{\hat{T}} \in [\P^q(\hat{T})]^d, ~i \in \{m,n\}$.

The restriction of $v_{\Thn{m}} \in \Vhn{m}$ to $T^m$, i.e. $v_{T^m} := v_{\Thn{m}}|_{T^m}$, is a mapped polynomial. % and hence a smooth function $v_{T_m} \in C^{\infty}(T_m)$.
We can map it back to the undeformed element $\hat{T} \in \Th$ and realize
that there is
$\hat{v}_{T} \in \P^k(\hat{T})$ such that $v_{T^m} = \hat{v}_{T} \circ \Theta_{T}^{-m}$.
Let $\hat{v}_{T}^{*} = \Ep \hat{v}_T \in \P^k(\Thneps{}{\eps})$ and $\Theta_{T}^{m*} = \Ep \Theta_T^m \in [\P^q(\Thneps{}{\eps})]^d$ be the canonical extension of this polynomial to the $\eps$-neighborhood $\Thneps{}{\eps}$ of $\hat{T}$. 
With $v_{T^m}^{*} := \hat{v}_{T}^{*} \circ \Theta_{T}^{-m*} = \Ep \hat{v}_T \circ (\Ep \Theta_T^m)^{-1}$ we have a smooth extension of $v_{T^m}$ from $T^m$ to \hypertarget{def:Tneps}{$T^{m}_\eps := \Theta_{T}^{m*}(\Thneps{}{\eps})$}, such that $v_{T^m}^*|_{T^m} = v_{T^m}$ still holds and furthermore $T^n \subset \Tneps{m}{\eps}$. A sketch of this extension is given in \Cref{fig:extproj}.

\begin{figure}
  \vspace*{-0.2cm}
  \begin{center}
    \includegraphics[trim=0.2cm 0.0cm 0.2cm 0.0cm, clip=true, width=0.9\textwidth]{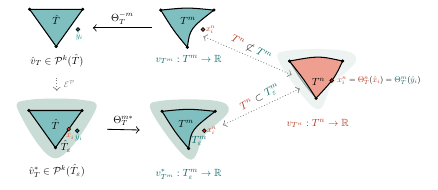} 
  \end{center}
  \vspace*{-0.5cm}
  \caption{
    Sketch of extension (\Cref{sec:ext}) and transfer operation (\Cref{sec:shiftedeval}). 
    For an element $T^m = \Theta_T^m(\hat{T}) \in \Thn{m}$ the extension $T^m_{\eps}$ covers $T^n = \Theta_T^n(\hat{T}) \in \Thn{n}$. For a Lagrange node $x_i^n$ in $T^n$ the mappings $\Theta_T^{-n}$ and $\Theta_T^{-m*}$, respectively, yield different points $\hat{x}_i$, $\hat{y}_i$ in $\hat{T}_{\eps}$. }
  \label{fig:extproj} 
\end{figure}

\subsubsection{Element-local interpolation (shifted evaluation)}\label{sec:shiftedeval}
With $T^n \subset \Tneps{m}{\eps}$, we can define the following element-local interpolation that for given $v_{T^m}^{*}$ as constructed in the previous section it yields $\tilde{v}_{T^n} \in \P^k(\hat{T}) \circ \Theta_{T}^{-n}$, or equivalently
$\hat{v}_{T^n} \in \P^k(\hat{T})$ with $\hat{v}_{T^n} = \tilde{v}_{T^n} \circ \Theta_T^n$, by nodal interpolation.

Let \hypertarget{def:L}{$L(\hat{T}) = \{\hat{x}_i\}_{i=1,...,M_L}, ~M_L=\# L(\hat T)$ be the set of Lagrange nodes} of $\Vhn{n}|_{\hat{T}} = \mathcal{P}^k(\hat{T})$ on $\hat T$ with corresponding set of Lagrange basis functions $\{\hat{\varphi}_i\}_{i=1,...,{M_{\LagL}}}$, s.t. $\hat{\varphi}_i(\hat{x}_j) = \delta_{ij},~i,j=1,..,{M_{\LagL}}$. The correspondingly mapped nodes and basis functions are $\LagL(T^n) := \{x_i^n\}_{i=1,...,{M_{\LagL}}}$ and $\{{\varphi}_i^n\}_{i=1,...,{M_{\LagL}}}$ with $x_i^n = \Theta_T^n(\hat{x}_i)$ and $\varphi_i^n = \hat{\varphi}_i \circ \Theta_T^{-n}$. 
We define
\begin{subequations}
\begin{align}
  \tilde{v}_{T^n}(x) & := \sum_{i=1}^{{M_{\LagL}}}  v_{T^m}^{*}(x_i^n) \varphi_i^n(x), \qquad \forall x \in T^n
  \intertext{
or equivalently, with $v^{*}_{T^m} = \hat v^{*}_{T} \circ \Theta_T^{-m*}$ and $x_i^n = \Theta_T^n(\hat{x}_i)$ we have}
  \hat{v}_{T}(\hat{x}) & := \sum_{i=1}^{{M_{\LagL}}} \hat{v}_{T}^{*} \underbrace{\left(\Theta_{T}^{-m*}\left(\Theta_{T}^n(\hat{x}_i)\right)\right)}_{\hat{y}_i} \hat{\varphi}_i(\hat{x}), \qquad \forall \hat{x} \in \hat{T}. 
\end{align}
\end{subequations}
Let us stress that \hypertarget{def:yihat}{$\hat{x}_i \ne \hat{y}_i := \Theta_{T}^{-m*}(x_i^n)$} and we hence call this step \emph{shifted evaluation}, cf. \Cref{fig:extproj} for a sketch of the relation between $\hat{x}_i$ and $\yihat$.
By setting $\tilde{v}_{\Thn{n}}|_{T^n} := \tilde{v}_{T^n}$ for all $T^n \in \T_{h}^n$ we obtain $\tilde{v}_{\Thn{n}} \in \bigoplus_{T^n\in\T_{h}^n} \Vhn{n}|_{T^n}$. 

\subsubsection{Projection into the space of continuous functions}
\label{sss:oswald}
After the previous steps we obtain a discontinous, element-wise (mapped) polynomial approximation on $\Thn{n}$. We then apply an Oswald-type quasi-interpolation $P_h: \bigoplus_{T^n\in\T_{h}^n} \C(T^n) \to \Vhn{n}$ in order to get a continuous function in $\Vhn{n}$.
% , i.e.  $v_{\Thn{n}} := \Pimn{m}{n} v_{\Thn{m}} := P_h (\tilde{v}_{\Thn{n}} \circ \Theta_h^n) \circ (\Theta_h^n)^{-1}$.
Let $\LagL(\Thn{n}) = \{x_i^n\}$ be the set of Lagrange nodes of $\Vhn{n}$ on $\Thn{n}$, and $\{\varphi_i^n\}$ the set of corresponding Lagrange basis functions. Let $\omx{x_i^n}$ be the set of elements containing the \extendedonly{Lagrange }node $x_i^n$.
The Oswald-type projector $P_h$ is the following generalization of the Lagrange interpolation for a discontinuous function $v$:
\begin{equation}
P_h v := \!\!\! \sum_{x_i^n \in \LagL(\Thn{n})} \!\!\! \Big( \# \omx{x_i^n}^{-1} \!\!\! \sum_{T^n \in \omx{x_i^n}} \!\!\! v|_{T^n}(x_i^n) \Big) \varphi_i^n.
\end{equation}

\subsection{Analysis of the projection} \label{ssec:error:proj}
We start with a simple observation for the norm evaluation w.r.t. one mesh of a function defined on another (slightly different) mesh. 
\begin{lemma} \label{lem:shiftednorm}
For $v_h \in \Vhn{n}$, $w_h \in \Vhn{m}$, $T_i = \Theta_{T}^i(\hat{T})$, $\hat{T} \in \Th$, $i\in\{m,n\}$, there holds
  \begin{equation}
	\| v_h + w_h \|_{T^n} \lesssim h^{\frac{d}{2}} \| v_h + w_h \|_{\L^{\infty}(\Tneps{n}{-\eps})} + h^{\frac52} \| \nabla v_h \|_{T^n} + h^{\frac52} \| \nabla w_h \|_{\omT{T^m}}. 
  \end{equation}
\end{lemma}

\begin{proof} Obviously we have
$\| v_h + w_h \|_{T^n} \lesssim \| v_h + w_h \|_{\Tneps{n}{-\eps}} + \| v_h + w_h \|_{T^n \setminus \Tneps{n}{-\eps}}$. The first term on the right hand side is simply bounded by $h^\frac{d}{2} \|v_h + w_h\|_{\L^\infty(\Tneps{n}{-\eps})}$. For the second term we use \Cref{lemma:dmax}, i.e. $\meas_d(T^n \setminus \Tneps{n}{-\eps}) \lesssim h^{d+1}$, and that for $x \in T^n \setminus \Tneps{n}{-\eps}$ there is 
$y \in \partial \Tneps{n}{-\eps}$ and $z \in \operatorname{conv}\{x,y\} \subset T \setminus \Tneps{n}{-\eps}$ such that $|(v_h+w_h)(x)| \leq |(v_h+w_h)(y)| + |(\nabla(v_h+w_h))(z)| ~ |y-x|$, and hence with $\eps \lesssim h^2$ we obtain
%\todo[inline]{CL: There is no such thing as norm equivalence on that space with constants that are independent of $n$, $m$, etc.. That's why we have to do this result so technical.}
  \begin{align*}
	\| v_h + w_h & \|_{T^n \setminus \Tneps{n}{-\eps}} \lesssim h^{\frac{d+1}{2}} \| v_h + w_h \|_{\L^{\infty}(T^n \setminus \Tneps{n}{-\eps})} \\
	& \lesssim h^{\frac{d+1}{2}} \left( \| v_h + w_h \|_{\L^{\infty}(\Tneps{n}{-\eps})} + \eps \| \nabla (v_h + w_h) \|_{\L^{\infty}(T^n \setminus \Tneps{n}{-\eps})} \right) \\
	& \lesssim h^{\frac{d+1}{2}} \left( \| v_h + w_h \|_{\L^{\infty}(\Tneps{n}{-\eps})} + h^2 \| \nabla v_h \|_{\L^{\infty}(T^n)} + h^2 \| \nabla w_h \|_{\L^{\infty}(\omT{T^m})} \right) \\
	& \lesssim h^{\frac{d+1}{2}} \| v_h + w_h \|_{\L^{\infty}(\Tneps{n}{-\eps})} + h^{\frac52} \| \nabla v_h \|_{T^n} + h^{\frac52} \| \nabla w_h \|_{\omT{T^m}}.
  \end{align*}
  \extendedonly{For the last step we made use of norm equivalences on a reference element after transformation\footnote{Note that we cannot directly apply such a result for $v_h + w_h$ as $w_h$ and $v_h$ are not from the same (mapped) polynomial space.}.}
\end{proof}

\begin{lemma}\label{lem:projstab}
  Let $\Pin{n}: \V_{h}^m \to \V_{h}^n, ~m=n-1$ be the projection for a discrete function $v_h \in \V_{h}^m$ from the mesh $\T_{h}^m$ to the mesh $\T_{h}^n$.
  Further, let $\tilde{\T}_{h} \subset \T_h$ be an arbitrary selection of elements and the corresponding deformed meshes $\tilde{\T}_{h}^m$, $\tilde{\T}_{h}^n$ with the corresponding domains $\tilde{\O}_{h}^m$, $\tilde{\O}_{h}^n$, respectively. 
  For $c_{\ref{lem:projstab}a} > 0$ and $c_{\ref{lem:projstab}b} > 0$ independent of $h$ and $\Delta t$ there holds for $\hat{T} \in \T_h,~T^n = \Theta_T^n(\hat{T}), T^m = \Theta_T^m(\hat{T})$: 
        \begin{subequations}
        \begin{align}
          \| \Pin{n} v_h \|_{T^n}^2 &\lesssim \| v_h \|_{\omT{T^m}}^2, &          \| \Pin{n} v_h \|_{\tilde{\O}_{h}^{n}}^2 &\leq c_{\ref{lem:projstab}a} \| v_h \|_{\tilde{\O}_{h,+}^{m}}^2, \label{eq:projstab1}\\
          \| \nabla \Pin{n} v_h \|_{T^n}^2 & \lesssim \| \nabla v_h \|_{\omT{T^m}}^2, &          \| \nabla \Pin{n} v_h \|_{\tilde{\O}_{h}^{n}}^2 &\leq c_{\ref{lem:projstab}b} \| \nabla v_h \|_{\tilde{\O}_{h,+}^{m}}^2. \label{eq:projstab2}
        \end{align}
      \end{subequations}
    \end{lemma}
    \begin{proof}
We have $\| \Pin{n} v_h \|_{T^n} \simeq \| \Pin{n} v_h \|_{\Tneps{n}{-\eps}}$, cf. \eqref{eq:normequiv}. For $x \in \Tneps{n}{-\eps}$ and $\hat{x} = \Theta^{-n}(x) \in \Thneps{}{-\eps}$ we have by definition \extendedonly{(recall the notation introduced in \Cref{sec:locproj})}
\begin{equation*}
	\Pin{n} v_h (x) =             
	\sum_{i=1}^{{M_{\LagL}}}
	\Big(
	\# \homega{x_i^n}^{-1}
	\sum_{\tilde{T} \in \homega{x_i^n}}
	\hat{v}_{\tilde{T}}^*(\yihat)
	\Big)  \hat{\varphi}_i(\hat{x})
\end{equation*}
for ${M_{\LagL}} = \# \LagL(T^n) \lesssim 1$\extendedonly{ the number of Lagrange nodes on $T^n$ and the patch w.r.t. to the undeformed mesh $\Th$, $\homega{x_i^n} = \Theta^{-m}(\omx{x_i^n})$}. Hence $\| \Pin{n} v_h \|_{\Tneps{n}{-\eps}} \lesssim h^{\frac{d}{2}} \| v_h \|_{\L^{\infty}(\omT{T^m})} \lesssim \| v_h \|_{\omT{T^m}}$. Summing over all elements in $\tilde{\T}_{h}^n$ yields the first result. 
For the second equation we proceed similarly after introducing $\bar{v}_h = \frac{1}{|\omT{T^m}|}\int_{\omT{T^m}} v_h ~ds$ with $\Pin{n} \bar{v}_h = \bar{v}_h$ on $\Tneps{n}{-\eps}$
\begin{align*}
	\| \nabla \Pin{n} v_h \|_{\Tneps{n}{-\eps}} & = \| \nabla \Pin{n} (v_h - \bar{v}_h) \|_{\Tneps{n}{-\eps}}
	\lesssim h^{\frac{d}{2}} h^{-1} \| v_h - \bar{v}_h \|_{\L^{\infty}(\omT{T^m})} \\
	& \lesssim h^{-1} \| v_h - \bar{v}_h \|_{\omT{T^m}}
	\lesssim \| \nabla v_h \|_{\omT{T^m}}.
\end{align*}    
    \end{proof}

\begin{lemma} \label{lemma:proj}
	Let $\Pin{n}: \V_{h}^m \to \V_{h}^n, ~m=n-1$ be the projection for a discrete function $v_h \in \V_{h}^m$ from the mesh $\T_{h}^m$ to the mesh $\T_{h}^n$.
	Then there holds for $T^n \in \T_{h}^n$ and $T^m = \Theta^m(\Theta^{-n}(T^n)) \in \T_{h}^m$
	\begin{subequations}
		\begin{align}\label{eq:proj:a}
                  \| (\operatorname{id} - \Pin{n}) v_h \|_{T^n}  &\lesssim h^2 \| \nabla v_h \|_{\omT{T^m}}  \lesssim h \| v_h \|_{\omT{T^m}}.
\\
                  \intertext{
\extendedonly{Here, $\omT{T^m}$ is the element patch to $T^m$. }Hence for an arbitrary selection of elements $\tilde{\T}_{h}$ and the corresponding deformed meshes $\tilde{\T}_{h}^m$, $\tilde{\T}_{h}^n$ with the corresponding domains $\tilde{\O}_{h}^m$, $\tilde{\O}_{h}^n$, respectively, there holds
                  }
\| (\operatorname{id} - \Pin{n}) v_h \|_{\tilde{\O}_{h}^n} &\lesssim h^2 \| \nabla v_h \|_{\tilde{\O}_{h,+}^m} \label{eq:proj:b} \lesssim h \| v_h \|_{\tilde{\O}_{h,+}^m}. %\label{eq:proj:c}.
		\end{align}
	\end{subequations}
\end{lemma}

\begin{proof}
	Let $T^n = \Theta_{T}^n(\hat{T}) \in \T_{h}^n,~ T^m = \Theta_{T}^m(\hat{T}) \in \T_{h}^m$ for $\hat{T} \in \Th$. 
	With \Cref{lem:shiftednorm} and \Cref{lem:projstab} we have 
	\begin{equation*}
		\| (\operatorname{id} - \Pin{n}) v_h \|_{T^n} \lesssim h^{\frac{d}{2}} \| (\operatorname{id} - \Pin{n}) v_h \|_{\L^{\infty}(\Tneps{n}{-\eps})} + h^{\frac52} \| \nabla  v_h \|_{\omT{T^m}}.
	\end{equation*}
	\extendedonly{We start with bounding the $\L^{\infty}$ norm. }For $\forall x \in \Tneps{n}{-\eps} \subset T^n \cap T^m$ and $\hat{x} := \Theta_T^{-n}(x)$, 
    we have by definition
	\begin{equation*}
	(v_h - \Pin{n} v_h)(x) =             
	\sum_{i=1}^{{M_{\LagL}}}
	\Big(
	\# \homega{x_i^n}^{-1}
	\sum_{\tilde{T} \in \homega{x_i^n}}
	\big( \hat{v}_{\tilde{T}}^*(\hat{x}_i) - \hat{v}_{\tilde{T}}^*(\yihat) \big)
	\Big) \hat{\varphi}_i(\hat{x})
	\end{equation*}
	for ${M_{\LagL}} = \# \LagL(T^n) \lesssim 1$\extendedonly{ the number of Lagrange nodes on $T^n$ and the element patch w.r.t. the undeformed mesh $\homega{x_i} = \Theta^{-m}(\omT{T^m}) \subset \Th$}. 
	With the second-order boundedness of $\Theta^i,~i\in\{m,n\}$, cf. \Cref{propertiesdh}, we have $| \hat{x}_i - \yihat |  \lesssim h^2$ and hence for every $\tilde{T}$ in the element patch $\homega{\hat{T}}$\extendedonly{ there holds}
	\begin{equation*}
	| \hat{v}_{\tilde{T}}^*(\hat{x}_i) - \hat{v}_{\tilde{T}}^*(\yihat) | \lesssim h^2 \| \nabla \hat{v}_{\tilde{T}}^* \|_{\L^\infty(\hat{T}_\eps)}
	\lesssim h^2 \| \nabla \hat{v}_{\tilde{T}} \|_{\L^\infty(\hat{T})} \lesssim h^{2-\tfrac{d}{2}} \| \nabla \hat{v}_{\tilde{T}} \|_{\hat{T}}.
	\end{equation*}
	\extendedonly{Here, we made use of the norm equivalence on finite dimensional spaces and scaling arguments.}
	This yields \eqref{eq:proj:a}. % have 
	Summing over $\tilde{\T}_{h}^n$ we obtain \eqref{eq:proj:b} by using finite overlap. 
\end{proof}

\begin{lemma}\label{cor:projstab}
Let $\Pin{n}: \V_{h}^m \to \V_{h}^n, ~m=n-1$ be the projection for a discrete function $v_h \in \V_{h}^m$ from the mesh $\T_{h}^m$ to the mesh $\T_{h}^n$.
For constants $c_{\ref{cor:projstab}a}$ and $c_{\ref{cor:projstab}b}$ independent of $h$ and $\Delta t$ there holds
        \begin{align}
          \srn{1}{n}(\Pi^{n} v_h, \Pi^{n} v_h) & \leq 
          c_{\ref{cor:projstab}a} \srn{2}{m}(v_h,v_h) +c_{\ref{cor:projstab}b} h^2 \| \nabla v_h \|^2_{\O_{h,2}^{n-1,\mathcal{S}}} . \label{eq:projstab3}
        \end{align}
\end{lemma}
\begin{proof}[Proof (sketch)]
  The proof relies on the application of an estimate of triangle inequality type for 
  each involved facet, $s_{F}^{n}(\Pin{n}v_h,\Pin{n}v_h) \lesssim s_{F}^{m}(v_h,v_h) + \Vert \Pin{n} v_h - v_h \Vert_{\omF{F}}^2$, and the estimates from \Cref{lemma:proj}. The details are technical and given only in the appendix for completeness,  cf. \Cref{proofs}.
\end{proof}

The previous lemmas describe "the worst case" scenarios as $v_h$ is allowed to be arbitrarily rough in $\V_h^m$. Assuming more smoothness helps to improve the bound.
\begin{lemma} \label{lemma:proj2}
  Let $\Pin{n}: \Vhn{m} \to \Vhn{n},~m=n-1$ be the projection for a discrete function $v_h \in \Vhn{m}$ from the mesh $\Thn{m}$ to the mesh $\Thn{n}$. For any $\hat{T} \in \Th$ and $T^n = \Theta^n(\hat{T}) \in \Thn{n}$, $v \in \W_\infty^{k+1}(\Tneps{n}{\eps}) \cap C^0(\omT{T^n})$ and \hypertarget{def:Ihn}{$\I^p_h$ the Lagrange interpolation w.r.t. $\Thn{p}, ~p \in \{m,n\}$}, there holds
		\begin{align} \label{eq:proj2 elem}
                  \| (\Ihn{n} - \Pin{n} \Ihn{m}) v \|_{T^n} & \lesssim \| \Theta^n - \Theta^m \|_{\L^{\infty}(\homega{\hat{T}})} ~ h^{k+\frac{d}{2}} |v|_{\W_\infty^{k+1}(T^n_\eps)}.
    \end{align}
\end{lemma}
\begin{proof}
  Let $\varphi_i^p := \hat{\varphi}_i \circ \Theta_T^{-p}$ be the Lagrange basis functions of $\Vhn{p}|_{T^p}$ w.r.t. the Lagrange nodes $x_i^p := \Theta_T^p(\hat{x}_i)$ on $T^p \in \Thn{p}$ such that $\varphi_i^p(x_j^p) = \delta_{ij}, ~i,j=1,...,{M_{\LagL}}:=\# \LagL(T^p)$.
Analogously to the local interpolation operator $\I_{T^p}^p$ on $T^p$ we define the nodal interpolation operator on the extension  $T^p_{\eps}$ of $T^p$, i.e. for $p\in\{m,n\}$ such that
\begin{equation*}
\I_{T^p}^p v (x) := \sum_{i=1}^{M_{\LagL}} v (x_i^p) \varphi_i^p (x), \quad\text{and}\quad
\I_{\Tneps{p}{\eps}}^p v (x) := \sum_{i=1}^{M_{\LagL}} v(x_i^p) \varphi_i^{p*} (x), \qquad v \in \L^\infty(T^p),
\end{equation*}
where $(\cdot)^* = \Ep \cdot$ canonically extends the basis functions $\varphi_i^{p}$ on $T^p \in \Thn{p}$ to $\Tneps{p}{\eps}$.
With the definition of the projection $\Pin{n}$, cf. \Cref{sec:locproj}, we have for $x \in T^n$
	\begin{equation*} %\label{eq:interp proj}
		\left( \Ihn{n} v \!-\! \Pin{n} \Ihn{m} v \right) (x) 
		\!=\! \sum_{j=1}^{M_{\LagL}} \bigg( \! \#\omx{x_i^m}^{-1} \!\!\!\!\!\sum_{\tilde{T}^m \in \omx{x_i^m}} \!\!\!\!\!\! \Big( (v-\I_{\tilde{T}^{m}_{\eps}}^m v) (x_j^n) \Big) \! \bigg) \varphi_j^n(x)
	\end{equation*}
for any $v \in C^0(\omT{T^m})$. 
Taking the $\L^2$-norm on $T^n \in \Thn{n}$ then this yields
\begin{equation} \label{eq:proj2 bound}
\left\| \Ihn{n} v - \Pin{n} \Ihn{m} v \right\|_{T^n}
\lesssim h^{\frac{d}{2}} \max_{\tilde{T}^m \in \omT{T^m}} \left\|v - \I_{\tilde{T}^m_{\eps}}^m v \right\|_{l_\infty(\LagL(T^n))}
\end{equation}
where we made use of $\|\sum_{j=1}^{M_{\LagL}} \varphi_j^n\|_{T^n} \lesssim h^{\frac{d}{2}}$.
\hypertarget{def:Taylor}{Let the Taylor polynomial of degree $k$ that expands a function $\hat{v}$ at node $\hat{x}_j^m$ be denoted by $\hat{\mathfrak{T}}_{\hat{x}_j^m}$.}%: \W_\infty^{k+1}(\hat{T}'^*) \to \P^k(\mathbb{R}^d)$.
We introduce the mapped Taylor polynomial (to the element $\tilde{T}$ in the patch $\omT{T^m}$) by transformation to the straight element applying the Taylor expansion there and transforming back, i.e., as $\Ty_{x_j^m} v := \Big( \hat{\Ty}_{\hat{x}_j^m} \big( v \circ \Theta_{\tilde{T}}^{m*} \big) \Big) \circ \Theta_{\tilde{T}}^{-m*}$ such that there holds
 $\Ty_{x_j^m} v = \I_{\tilde{T}^m_{\eps}}^m \Ty_{x_j^m} v$. For $j=1,...,{M_{\LagL}}$ we then have
\begin{align} \label{eq:seq bound}
\left| (v - \I_{\tilde{T}^m_{\eps}}^m v) (x_j^n) \right|
&= \left| (v - \Ty_{x_j^m} v)(x_j^n) + (\Ty_{x_j^m} v - \I_{\tilde{T}^m_{\eps}}^m v)(x_j^n) \right| \\
&\le \left| (v - \Ty_{x_j^m} v)(x_j^n) \right| + \left| \I_{\tilde{T}^m_{\eps}}^m (\Ty_{x_j^m} v - v)(x_j^n) \right| \nonumber
\end{align}
For the first part we simply have with $x_j^n = \Theta_{\tilde{T}}^n(\hat{x}_j)$, $x_j^m = \Theta_{\tilde{T}}^m(\hat{x}_j)$ and $x_j^n, x_j^m \in T^n_\eps$: 
\begin{align*}
 \left| (v - \Ty_{x_j^m} v)(x_j^n) \right| \lesssim  |x_j^n - x_j^m|^{k+1} |v|_{\W_\infty^{k+1}(T^n_\eps)}  \lesssim  \|\Theta_{\tilde{T}}^{n} - \Theta_{\tilde{T}}^{m}\|_{\L^\infty(\hat{T})}^{k+1} |v|_{\W_\infty^{k+1}(T^n_\eps)}. 
\end{align*}
For the second part we exploit $(\Ty_{x_j^m} v - v)(x_j^m) = 0$ to obtain \vspace*{-0.3cm}
\begin{align*} 
 & \big| \I_{\tilde{T}^m_{\eps}}^m (\Ty_{x_j^m} v - v)(x_j^n) \big| \!=\! \big|
  \sum_{i=1}^{M_{\LagL}} (\Ty_{x_j^m} v - v)(x_i^m)  \varphi_i^{m*} (x_j^n) \big| \!=\!
\big| \sum_{i\neq j} (\Ty_{x_j^m} v - v)(x_i^m)  \varphi_i^{m*} (x_j^n) \Big|                                
\\
               &\lesssim \max_{i\neq j} \big| \varphi_i^{m*} (x_j^n) \big| \max_{i\neq j} \big| (\Ty_{x_j^m} v - v)(x_i^m) \big|
                 \lesssim \max_{i\neq j} \big| \varphi_i^{m*} (x_j^n) \big| \underbrace{|x_i^m - x_j^m|^{k+1}}_{\lesssim h^{k+1}} |v|_{\W_\infty^{k+1}(T^n_\eps)}. \end{align*}\vspace*{-0.6cm}\\
Finally we bound $\left| \varphi_i^{m*} (x_j^n) \right|$ for $i\neq j$ by using $\varphi_i^{m*} (x_j^m) = 0$:
\begin{align*}
  \! \big| \varphi_i^{m*}\! (x_j^n) \big| \!=\!
  \big| \varphi_i^{m*}\! (x_j^n)\! -\! \varphi_i^{m*}\! (x_j^m) \big|  
  \! \lesssim \!   \| \nabla \varphi_i^{m*} \|_{\L^\infty(\hat{T}_\eps)} \| x_j^n \! - \! x_j^m
  \| \! \lesssim \! h^{-1} \|\Theta_{\tilde{T}}^{n} \!-\! \Theta_{\tilde{T}}^{m}\|_{\L^\infty(\hat{T})}\!.\!
\end{align*}
Altogether, by
summing over all $\tilde{T}^m$ in $\omT{T^m}$ we obtain the claim \eqref{eq:proj2 elem}. 
\end{proof}

\section{A priori error analysis}
\label{sec:ana}
The analysis follows a similar strategy as in \cite{LO_ESAIM_2019} and where possible, we will refer to the corresponding results. In the analysis of the method we will only treat the case of a BDF2 discretization, i.e., we fix $r=2$. Necessary changes compared to \cite{LO_ESAIM_2019} occur due to the projection $\Pimn{m}{n}$ and the second-order time difference stencil $\dtBDFb$, c.f. \Cref{sec:fullhotime}. 
To deal with the BDF2 scheme in the following stability analysis we define a suitable norm and make a simple but useful observation, cf. \cite[Lemma 6.33]{Ern04}.
\begin{lemma}\label{lem:bdfnorm}
Let $u_h, v_h, w_h \in \mathcal{V}_{h,2}^{n}$. 
For $S \subseteq \Orn{2}{n}$ we define the BDF2 tuple norm %$(\cdot,\cdot)$
  \begin{equation} \label{eq:bdf2norm}
   \| (w_h,v_h) \|_{S}^2 := \| w_h \|_{S}^2 + \| 2 w_h -  v_h \|_{S}^2 \text{ with } \| (w_h,v_h) \|_{S}^2 \leq 6 ( \Vert w_h \Vert_S^2 + \Vert v_h \Vert_S^2)
  \end{equation}
so that for $\dtBDFr{2}$ as in \eqref{eq:stencil:BDF2} there holds
  \begin{align}\label{eq:bdf2normest}
    \big( 4 \Delta t \dtBDFr{2}(w_h,v_h,u_h),w_h\big)_{\Omhn{n}}
    \ge \| (w_h,v_h) \|_{\Omhn{n}}^2 - \| (v_h, u_h) \|_{\Omhn{n}}^2.
  \end{align}
\end{lemma}
\extendedonly{
\begin{proof}
	Multiplying the terms out we find, cf. \cite[Lemma 6.33]{Ern04}:
  	\begin{align*}
    	\big( 4 \Delta t \dtBDFr{2}(w_h, v_h, u_h), w_h \big)_{\Omhn{n}} = & \hphantom{-} \| w_h \|_{\Omhn{n}}^2 + \| 2 w_h -  v_h \|_{\Omhn{n}}^2 \\[-3ex] & - \| v_h \|_{\Omhn{n}}^2 - \| 2 v_h -  u_h \|_{\Omhn{n}}^2
    	+ \overbrace{\| w_h - 2 v_h + u_h \|_{\Omhn{n}}^2}^{\geq 0}.
  	\end{align*}
\end{proof}
}
Let us note that similar norms with a corresponding estimate can also be defined for BDF3 and BDF4, cf. \cite[Section 2]{liu2013simple}. Similarly, the analysis below can also be transfered to $r=3$ and $r=4$. We however decided to treat only the case $r=2$ to keep the technicalities at a manageable level.

As a consequence of the  errors introduced by the mesh transfer operator, in the subsequent analysis, we make the following assumption on the ratio between space and time refinements.
\newc{
\begin{assumption}\label{ass:hhdt}
  We assume that as $h$ and $\Delta t$ go to zero, $\frac{h^5}{\Delta t^2}$ converges to zero too. Then, for any constant $c_{A\ref{ass:hhdt}} > 0$ there is an $h_0 > 0$ so that for all meshes with mesh size $h < h_0$ we have $\frac{h^5}{\nu \Delta t^2} < c_{A\ref{ass:hhdt}}$.
\end{assumption}
\begin{remark}
\Cref{ass:hhdt} is only a restriction on the efficiency for high order in space. Assuming an $\L^2(0,\finaltime;\H^1(\Omega))$-error bound of the form\footnote{see also the numerical examples for a motivation of this ansatz} $\mathcal{O}(\Delta t^2 + h^k)$. Only for $k > 5$ a scaling $h > \Delta t^{\frac25}$ would be benefitial for the efficiency of the scheme. 
\end{remark}
} 

We will need certain space-time norms, which we will abbreviate, e.g.,   
\hypertarget{def:LinfH}
{
$\L^\infty(\H^k)$
for 
$\L^\infty(0,\finaltime;\H^k(\Omega(t)))$
}
and 
\hypertarget{def:LinfW}
{
$\L^\infty(\W^k_{\infty})$
for 
$\L^\infty(0,\finaltime;\W^k_\infty(\Omega(t)))$
}. 

Furthermore, in the remainder we treat the asymptotic behavior, $h, \dt \to 0$ only, i.e., we implicitly assume $h$ and $\dt$ \emph{sufficiently small} at several occasions. 
\new{We also only consider an $h$-type analysis, assuming the polynomial degree $k$ to be fixed. Especially the constants occuring may in general depend on $k$.}

\subsection{Error splitting and error equation}
Let $u$ be the exact solution to \eqref{eq:eqn}. For $u^n := u(t_n), ~ n = 0,1,...,N$, we make use of the extension operator $\E: \L^2(\Omn{n}) \to \L^2(\Omepsn{n})$  introduced in \cite[Section 3.2.1]{LO_ESAIM_2019}. 
$\E u^n$ is well-defined on $\Omepsn{n} \supset \Om^{n+r}$ and $\E$ is uniformly continuous in standard Sobolev norms. In the following we assume that $\epsilon$ can be chosen sufficiently large so that $\Omepsn{n} \supset \O_{r}^n$ and identify $\E u^n$ with $u^n$.

First, we introduce an error splitting. We define \hypertarget{def:uIn}{$u_I^n := \Ihn{n} u^n \in \Vhn{n}$ the global Lagrange interpolant w.r.t. to $\Thn{n}$} such that the error can be split into
\begin{equation} \label{eq:split}
\errQ{n} := u^n - u_h^n = \erriQ{n} + \errdQ{n}, \quad \erriQ{n} := {u^n - \uIn{n}}, \quad \errdQ{n} := \uIn{n} - u_h^n,
\end{equation}
for the \emph{approximation error} $\erri{n}$ and the \emph{discrete error} $\errd{n}$, respectively. In addition, we split the approximation error after mesh transfer between $\Vhn{m}$ and $\Vhn{n}$:
\begin{equation} \label{eq:split proj}
\errpQ{m} := u^m - \Pimn{m}{n} u_h^m = \erripQ{m} + \Pimn{m}{n} \errd{m}, \quad \erripQ{m} := u^m - \Pimn{m}{n} \uIn{m}.
\end{equation}
\extendedonly{Here, $\errip{m}$ is the corresponding approximation error after projection, and $\Pimn{m}{n} \errd{m}$ is the projection of the corresponding discrete error.}

\hypertarget{def:lift}{We then introduce a lifting operator ${(\cdot)}^\ell: \H^1(\Omhn{n}) \to \H^1(\Omn{n}), ~v_h \mapsto v_h \circ \Theta \circ \PPsi^{-1}$, and $b^n(\cdot,\cdot)$, $f^n(\cdot)$ as the bilinear and linear forms $\bh{n}(\cdot,\cdot)$, $\fhn{n}(\cdot)$ with $\Omhn{n}$, $\Ghn{n}$ replaced by $\Omn{n}$, $\Gn{n}$.}
For the exact solution to \eqref{eq:eqn} there holds for $n = 0,1,...,N$ 
\begin{equation*}
  \int_{\Omn{n}} \partial_t u^n v ~dx + b^n(u^n,v) = f^n(v) \qquad \forall v \in \H^1(\Omn{n}),
\end{equation*}
from which we subtract \eqref{eq:disceq} to obtain the error equation
\begin{equation} \label{eq:erroreq}
  \int_{\Omhn{n}} \dtBDFr{2}(\err{n},\errp{n-1},\errp{n-2}) ~v_h ~dx + \bh{n}(\err{n},v_h)
  + \gamma s_{r}^n(\err{n},v_h) = \consistQ{n}(v_h)
\end{equation}
with the consistency error term that we decompose into four contributions
\begin{align} \label{eq:consist}
  \consistQ{n}(v_h) := & \overbrace{\int_{\Omhn{n}} \dtBDFr{2}(u^n,u^{n-1},u^{n-2}) ~v_h ~dx - \int_{\Omn{n}} \partial_t u^n v_h^{\lift} dx}^{\consist{n,1}(v_h)} \\ 
  & + \underbrace{\bh{n}(u^n,v_h) - b^n(u^n,v_h^{\lift})}_{\consist{n,2}(v_h)} + \underbrace{\gamma s_{r}^n (u^n, v_h)}_{\consist{n,3}(v_h)}
  + \underbrace{f^n (v_h^{\lift}) - f_{h}^n (v_h)}_{\consist{n,4}(v_h)}. \nonumber 
\end{align}
Further splitting $\err{n}$ and $\errp{n-1},\errp{n-2}$ in the error equation yields
\begin{align} \label{eq:discerreq}
  \int_{\Omhn{n}} \!\!\! \dtBDFr{2}(\errd{n},\Pin{n} \errd{n-1}\!\!,\Pimn{n-2}{n} \errd{n-2}) v_h dx \!+\! \bh{n}(\errd{n},v_h \!) & \!+\! \gamma s_{r}^n(\errd{n},v_h\!) 
  \!=\! \consist{n}(v_h) \!+\! \interpolQ{n}(v_h) \vspace*{-0.4cm}
\end{align}
with the interpolation error term \vspace*{-0.2cm}
\begin{equation} \label{eq:approx}
  \interpolQ{n}(v_h)
  := \underbrace{-\int_{\Omhn{n}} \dtBDFr{2}(\erri{n},\errip{n-1},\errip{n-2}) ~v_h ~dx}_{\interpolQ{n,1}(v_h)} \underbrace{\vphantom{\int_{\Omhn{n}}}- \bh{n}(\erri{n},v_h) - \gamma s_{r}^n(\erri{n},v_h)}_{\interpolQ{n,2}(v_h)}. \vspace*{-0.2cm}
  % = - \int_{\Omhn{n}} \partial_t^h(\erri{n},\erri{n-1}\!\!\!,...,\erri{n-l}) v_h dx - \bh{n}(\erri{n},v_h) - \gamma s_{h,l\delta}^n(\erri{n},v_h) 
  % + \int_{\Omhn{n}} q_t^h((\operatorname{id}-\Pin{n}) \uIn{n-1}\!\!\!,...,(\operatorname{id}-\Pi_{n-l}^n) u_I^{n-l}) v_h dx
\end{equation}

\subsection{Consistency and approximation bounds}
\begin{lemma} \label{lemma:consist}
Let $u \in \LinfH{k+1}$ with $\partial_t^3 u \in \L^\infty(\Q)$ be the exact solution to \eqref{eq:eqn} and let $g \in \LinfW{1}$, then the consistency error in \eqref{eq:consist} has the following bound for all $v_h \in \Vhn{n}$:
\begin{equation}
\left| \consist{n}(v_h) \right| \lesssim \left(\Delta t^2 + h^q + h^k \K^\frac12\right) \mathcal{R}_{\ref{lemma:consist}} \enorm{v_h}{n} \nonumber
\end{equation}
with $\mathcal{R}_{\ref{lemma:consist}} = \mathcal{R}_{\ref{lemma:consist}}(u,g) := \|\partial_t^3 u\|_{\L^{\infty}(\Q)} + \|u\|_{\LinfH{k+1}} + \|g\|_{\LinfW{1}}$. \extendedonly{Here, we recall the semi-norm $\enorm{v_h}{n}^2 = \|v_h\|_{\Omhn{n}}^2 + \frac{\nu}{2} \|\nabla v_h\|_{\Omhn{n}}^2 + \gamma \srn{2}{n}(v_h,v_h)$.}
\end{lemma}
\begin{proof}[Proof (sketch)]
  The proof follows along the lines of the proof of \cite[Lemma 5.11]{LO_ESAIM_2019} and for completeness is given in the appendix,  cf. \Cref{proofs}.
\end{proof}

\begin{lemma} \label{lemma:approx}
For any $u \in \LinfW{k+1}$ with $\partial_t u \in \LinfH{k}$, the interpolation errors in \eqref{eq:approx} have the following bound for all $v_h \in \mathcal{V}_{r}^{n}$: \vspace*{-0.2cm}
\begin{equation}
	\sum_{n=2}^N \left| \interpol{n,1}(v_h) \right|
	  \lesssim N^{\frac12} h^k
	~ \mathcal{R}_{\ref{lemma:approx}} ~ \|v_h\|_{l_2}, \quad 
	\left| \interpol{n,2}(v_h) \right| \lesssim \K^{\frac12} h^k \|u\|_{\H^{k+1}\left(\Omhn{n}\right)} \enorm{v_h}{n}
  \end{equation}
with $\displaystyle \mathcal{R}_{\ref{lemma:approx}} = \mathcal{R}_{\ref{lemma:approx}} (u)  :=  \|\partial_t u\|_{\LinfH{k}} + \|u\|_{\LinfW{k+1}}$
and \hypertarget{def:l2}{$\|v_h\|_{l_2}^2 := \sum_{n=2}^N \|v_h\|_{\Omhn{n}}^2$.}
\extendedonly{Here, we recall the semi-norm $\enorm{v_h}{n}^2 = \|v_h\|_{\Omhn{n}}^2 + \frac{\nu}{2} \|\nabla v_h\|_{\Omhn{n}}^2 + \gamma \srn{2}{n}(v_h,v_h)$.}
\end{lemma}
\begin{proof}
	We start with the first term $\interpol{n,1}(v_h)$ in \eqref{eq:approx} and sum it over all time steps from $n=2$ to $N$. Let \hypertarget{def:TnOm}{$T^n_\Om := T^n \cap \Omhn{n}$}, then due to Cauchy-Schwarz we obtain
	\begin{align*}% \label{eq:interpolfirst}
		\sum_{n=2}^N & \left| \interpol{n,1}(v_h) \right| 
		= \sum_{n=2}^N \Big| \int_{\Omhn{n}} \underbrace{\dtBDFr{r}(\erri{n},\errip{n-1},\errip{n-2})}_{=:z^n} v_h dx \Big| 
		\le \sum_{T^n \in \Thn{n}} \sum_{n=2}^N \int_{\TnOm{n}} |z^n| ~ |v_h| dx \\[-1ex]
		&\le \sum_{T^n \in \Thn{n}} \sum_{n=2}^N \|z^n\|_{\TnOm{n}} \|v_h\|_{\TnOm{n}} 
		\le \sum_{T \in \Thn{n}} \Big(\sum_{n=2}^N \|z^n\|_{\TnOm{n}}^2 \Big)^{\frac12} \Big(\sum_{n=2}^N \|v_h\|_{\TnOm{n}}^2 \Big)^{\frac12} \nonumber 
	\end{align*}
	Another triangle inequality yields the splitting
	\begin{equation*} %\label{eq:integrand}
  \begin{split} 
    \sum_{n=2}^N & \|z^n\|_{\TnOm{n}}^2  
		\le \sum_{n=2}^N \Big\|\frac{3\erri{n}-4\errip{n-1}+\errip{n-2}}{2 \Delta t}\Big\|_{\TnOm{n}}^2
		\lesssim \underbrace{\Delta t^{-2} \sum_{n=2}^N \left\|3\erri{n}-4\erri{n-1}+\erri{n-2}\right\|_{\TnOm{n}}^2}_{=:Z_1} \\[-2.5ex]
                             &+ \overbrace{\Delta t^{-2} \sum_{n=2}^N \left\|4(\Ihn{n} - \Pin{n} \Ihn{n-1}) u^{n-1} -
(\Ihn{n} - \Pimn{n-2}{n} \Ihn{n-2}) u^{n-2} \right\|_{\TnOm{n}}^2}^{=:Z_2}. 
    \end{split}
  \end{equation*}
	With $\mathbf{e}_I(t) := u(t) - \mathcal{I}_h u(t)$, the first part $Z_1$ can easily be estimated as follows
	\begin{align*}
		Z_1
		= \Delta t^{-2} \sum_{n=2}^N \Big\| 3\int_{t_{n-1}}^{t_{n}} \partial_t \mathbf{e}_I (t) ~dt - \int_{t_{n-2}}^{t_{n-1}} \partial_t \mathbf{e}_I (t) ~dt \Big\|_{\TnOm{n}}^2
		\lesssim N h^{2k}  \|\partial_t u\|_{\LinfH{k}(\TnOm{n})}^2 
	\end{align*}
        For the second part $Z_2$ we split the expression into the two parts
        \begin{align*}
          \Delta t^2 Z_2 & \leq\! %Z_2^a + Z_2^b  \\
   \underbrace{\sum_{n=2}^N \left\|(\Ihn{n}\! - \Pin{n} \Ihn{n-1})( 4 u^{n-1}\!\! -\! u^{n-2} ) \right\|_{\TnOm{n}}^2}_{=:\Delta t^2 Z_2^a}
 \!\!\!+\!  \underbrace{\sum_{n=2}^N \left\| \Pin{n} (\Ihn{n-1}\!\! -\! \Pin{n-1} \Ihn{n-2}) u^{n-2}\right\|_{\TnOm{n}}^2        }_{=:\Delta t^2 Z_2^b}   
        \end{align*}
        We start with $Z_2^a$ and apply \Cref{lemma:proj2}:
	\begin{align*}
	Z_2^a 
	&\lesssim \Delta t^{-2} \sum_{n=2}^N \left\|(\Ihn{n} - \Pin{n} \Ihn{n-1}) (4u^{n-1} - u^{n-2})\right\|_{\TnOm{n}}^2 \\
	&\lesssim \Delta t^{-2} \sum_{n=2}^N \|\Theta^n - \Theta^{n-1}\|_{\L^{\infty} \left(\homega{\hat{T}}\right)}^2 ~ h^{2k+d} ~ \|4u^{n-1} - u^{n-2}\|_{\W_\infty^{k+1}(\Tneps{n}{\eps})}^2 \nonumber \\ 
	& \lesssim \left(N + N_D \dt^{-2} h^4 \right)  h^{2k+d} \sup_{n = 0,..,N} \|u^n\|_{\W_\infty^{k+1}(\Tneps{n}{\eps})}^2 \nonumber
	\end{align*}
where we recall that by \Cref{ass:change} for most of the time steps there is $\|\Theta^n - \Theta^{m}\|_{\infty} \lesssim \dt$, while for a bounded number of time steps $N_D$ there is $\|\Theta^n - \Theta^{m}\|_{\infty} \lesssim h^2$. 
For $Z_2^b$ we additionally make use of \eqref{eq:projstab1} which extends the relevant region by one element layer yielding
	\begin{align*}
	Z_2^b \lesssim 
\big(N + N_D \dt^{-2} h^4 \big)  h^{2k+d} \sup_{n = 0,..,N} \|u^n\|_{\W_\infty^{k+1}(\omT{T^n}^*\cap \Omepsn{n})}^2
        \end{align*}
	As a consequence, with $N_D \lesssim N \Delta t$ and $\frac{h^4}{\Delta t} \lesssim 1$, cf. \Cref{ass:hhdt}, we arrive at $\big(N + N_D \dt^{-2} h^4 \big) \lesssim N$ and hence, after summing over the mesh, the bound
	\begin{align*}
		\sum_{T^n \in \Thn{n}} \sum_{n=2}^N \|z^n\|_{\TnOm{n}}^2 %\\
		\lesssim N h^{2k} \Big( \|\partial_t u\|_{\LinfH{k}}^2 
		+ \underbrace{(\# \Thn{n}) h^d}_{\lesssim 1} \|u\|_{\LinfW{k+1}}^2 \Big) .\nonumber
	\end{align*}
Next, we estimate the second term $\interpol{n,2}(v_h)$ in \eqref{eq:approx} (as in \cite[Lemma 5.12]{LO_ESAIM_2019}) by Cauchy-Schwarz inequality and interpolation estimates, i.e. 
	\begin{align*}
		|\bh{n}(\erri{n},v_h)| 
		&\le \|\erri{n}\|_{\H^1(\Omn{n})} \|v_h\|_{\H^1(\Omhn{n})} 
		\lesssim h^k \|u^n\|_{\H^{k+1}\left(\Omn{n}\right)} \big( \|v_h\|_{\Omhn{n}} + \frac{\nu}{2} \|\nabla v_h\|_{\Omhn{n}} \big), \\
		\gamma \srn{2}{n}(\erri{n},v_h) &\le \gamma \srn{2}{n}(\erri{n},\erri{n})^{\frac{1}{2}} \srn{2}{n}(v_h,v_h)^{\frac{1}{2}} \lesssim \K^\frac12 h^k \|u^n\|_{\H^{k+1}(\Omn{n})} \left(\gamma \srn{2}{n} (v_h, v_h)\right)^{\frac{1}{2}}. 
	\end{align*}
	Hence, we obtain the following bound for $\interpol{n,2}(v_h)$:
	\begin{equation*}
		\big| \interpol{n,2}(v_h) \big| \lesssim \big(1 + \K^\frac12\big)  h^k \|u^n\|_{\H^{k+1}(\Omn{n})} \big(\|v_h\|_{\Omhn{n}} + \frac{\nu}{2} \|\nabla v_h\|_{\Omhn{n}} + \gamma \srn{2}{n} (v_h,v_h)^{\frac12} \big). 
	\end{equation*}
  This concludes the proof.
\end{proof}

\subsection{The ghost penalty mechanism}
We introduce a slight generalization of an important result \cite{LO_ESAIM_2019} concerning the bound of the extensions obtained through the application of the ghost penalties.
\begin{lemma} \label{lem:stripbound}
  \begin{subequations}
    \newc{For $v \in H^1(\Omega_h^n)$ and  \hypertarget{def:Sdelta}{  $S_\delta(\Omhn{n}) := \{ x \in \Omhn{n} \mid \dist(x,\partial \Omhn{n}) < \delta \}$ the interior part of a tubular neighborhood of the domain boundary with $\delta > 0$ there holds for any $\beta > 0$}}  
    \begin{align}
    \newc{   \!\!\!\!\,\|\!\, v\|_{\Sdelta(\Omhn{n})}^2} &\newc{ \lesssim (1+\beta^{-1})\delta \|v\|_{\Omhn{n}}^2 + \delta \beta \|\nabla v\|_{\Omhn{n}}^2 } \label{eq:stripbounde} \\ 
      \intertext{
        For $v_h \in \mathcal{V}_{r}^{n}$, $r \in \{1,2,3\}$ and for all $\theta > 0$ there holds 
      }      
	\|v_h\|_{\O_{r}^n}^2 &\lesssim \|v_h\|_{\Omhn{n}}^2 + \K h^2 s_{r}^n(v_h,v_h)  \label{eq:stripbounda}\\
	\|\nabla v_h\|_{\O_{r}^n}^2 &\lesssim \|\nabla v_h\|_{\Omhn{n}}^2 + \K s_{r}^n(v_h,v_h) \label{eq:stripboundb} \\
%\newc{\|v_h\|_{\OdnS{r}{n}}^2}  &\newc{\leq c_{\ref{lem:stripbound}a} \dt \|v_h\|_{\Omhn{n}}^2 + c_{\ref{lem:stripbound}b} \nu \dt \|\nabla v_h\|_{\Omhn{n}}^2 \label{eq:stripboundd} +  c_{\ref{lem:stripbound}c} \dt \K \srn{r}{n}(v_h,v_h) } \\
	\|v_h\|_{\O_{r}^n}^2 &\leq (1 + c_{\ref{lem:stripbound}a} \dt ) \|v_h\|_{\Omhn{n}}^2 + c_{\ref{lem:stripbound}b} \nu \dt \|\nabla v_h\|_{\Omhn{n}}^2 \label{eq:stripboundc} +  c_{\ref{lem:stripbound}c} \dt \K \srn{r}{n}(v_h,v_h)\!\! 
\end{align}
  \end{subequations}
  with constants $c_{\ref{lem:stripbound}a} = c_{\ref{lem:stripbound}} r (1+\theta^{-1})$, $c_{\ref{lem:stripbound}b}(\theta) = c_{\ref{lem:stripbound}} r \theta \nu^{-1}$, $c_{\ref{lem:stripbound}c} = c_{\ref{lem:stripbound}} r (\theta + h^2 + h^2\theta^{-1})$ 
  for a constant $c_{\ref{lem:stripbound}} > 0$ independent of $h$ and $\dt$. \label{lem:stripnorm} 
\end{lemma}
\begin{proof}
 \newc{  
The first claim follows from \cite[Lemma 4.10]{elliott2012finite} with an additional scaling argument.} The remainder follows from \cite[Lemma 5.2, Lemma 5.5 and Lemma 5.7]{LO_ESAIM_2019} with minor modification for an extended strip size from $\O_{\d}^n$ (in \cite{LO_ESAIM_2019}) to $\O_{r}^n$ (here).
\end{proof}
\extendedonly{Note that the constant $c_{\ref{lem:stripbound}b}$ can be decreased at the price of an increase of $c_{\ref{lem:stripbound}a}$ by choosing $\theta$ accordingly. We will make use of this later in order to drive $c_{\ref{lem:stripbound}b}$ sufficiently small. In the following we will however only reflect that dependency on $\theta$ for $c_{\ref{lem:stripbound}b}$. }

Next, we bound the BDF2 tuple norm of a tuple of functions transfered from $\Omhn{n-1}$ and $\Omhn{n-2}$ to $\Omhn{n}$.
\begin{lemma} \label{lem:tupleest}
For any $c_{\ref{lem:tupleest}b}>0$ there holds for sufficiently small $h$ that for all $w_h \in \V_{2}^{n-1}$ and $v_h \in \V_{2}^{n-2}$ 
  \begin{align}
    \|& (\Pin{n} w_h, \Pin{n}\Pin{n-1} v_h) \|_{\Omhn{n}}^2
    \leq  (1 + c_{\ref{lem:tupleest}a} \Delta t) \| (w_h, \Pin{n-1}v_h) \|_{\Omhn{n-1}}^2 \nonumber \\
   & + c_{\ref{lem:tupleest}b} \nu \Delta t \big\{
     \| \nabla w_h \|_{\Omhn{n-1}}^2 + \| \nabla v_h \|_{\Omhn{n-2}}^2
     \big\} \label{eq:tupleest}
    + c_{\ref{lem:tupleest}c} \K  \Delta t \big\{  \srn{1}{n-1}( w_h, w_h) + \srn{2}{n-2}( v_h, v_h)
    \big\} \nonumber 
  \end{align}
for constants $c_{\ref{lem:tupleest}a}$, $c_{\ref{lem:tupleest}c}$ that are independent of $h$, $\Delta t$ or $n$.
\end{lemma}
\begin{proof}
 \newc{  
We will bound the l.h.s. of the claim step by step until only terms of the form of the r.h.s. remain. To simplify the notation we introduce the terms $R_1$, $R_2$ and $R_3$ for the contributions in the squared $L^2$ norm ($R_1$), the sqaure $H^1$ semi norm ($R_2$) and the ghost penalty part ($R_3$) so that the r.h.s. reads as $(1 + c_{\ref{lem:tupleest}a} \Delta t)R_1 + c_{\ref{lem:tupleest}b}\nu \Delta t R_2 + c_{\ref{lem:tupleest}c}\K \Delta t R_3$.  
% We use the notation $c_{\ref{lem:tupleest}a,i}$ (and correspondingly $c_{\ref{lem:tupleest}b,i}$ and $c_{\ref{lem:tupleest}c,i}$) for constants in front of corresponding terms that are in the form of the r.h.s., so that at the end the claim holds true with $c_{\ref{lem:tupleest}a} = \sum_i c_{\ref{lem:tupleest}a,i}$.

First, let us note that within $\Omhn{n}$ there only holds $w_h \neq \Pin{n} w_h$ on elements that are direct neighbors to cut elements at time $n-1$ or time $n$. We define the corresponding domain 
of the submesh of $\Orn{0}{n}$ as $\mathcal{O}^\ast$. We also introduce $\mathcal{O}^\ast_+$ which is the domain of the same submesh extended by one layer of neighbors, but deformed w.r.t. $\Theta^{n-1}$. Correspondingly $\mathcal{O}^\ast_{2+}$ adds an addditional layer and is deformed w.r.t. $\Theta^{n-2}$.
We observe
\begin{align*}
\Vert \Pin{n} w_h \Vert_{\Omhn{n}}^2
& \leq 
\Vert w_h \Vert_{\Omhn{n}}^2
+ \Vert w_h - \Pin{n} w_h \Vert_{\Omhn{n}}^2
+ 2 (w_h - \Pin{n} w_h, w_h)_{\Omhn{n}} \\ 
& \leq 
\Vert w_h \Vert_{\Omhn{n}}^2 + \beta
\Vert w_h \Vert_{\mathcal{O}^\ast \cap \Omhn{n}}^2 + (1+\beta^{-1}) \Vert w_h - \Pin{n} w_h \Vert_{\mathcal{O}^\ast}^2
\end{align*}
for any  $\beta > 0$. 
Using \eqref{eq:bdf2norm} to unroll the BDF2 norm, together with the previous inequality and \Cref{lemma:proj}
with 
$
\mathcal{O}^\ast_+ \subset \Orn{1}{n-1} 
$
yields
  \begin{align*}
     &\bdftwonorm{\Pin{n} w_h}{\Pin{n} \Pin{n-1} v_h}{\Omhn{n}}{2}
   \leq  \bdftwonorm{w_h}{\Pin{n-1} v_h}{\Omhn{n}}{2}
   + \beta \bdftwonorm{w_h}{ \Pin{n-1} v_h}{\mathcal{O}^\ast \cap \Omhn{n}}{2} \\
   &\qquad 
+ \!6 ~ c_{\eqref{eq:proj:b}} (1 +\beta^{-1} ) h^4 \!\left( \Vert \nabla w_h \Vert_{\Orn{1}{n-1}}^2 +  \Vert \nabla \Pin{n-1} v_h \Vert_{\mathcal{O}^\ast_+}^2 \right)
\!= \!\text{I}\!+\!\text{II}\!+\!\text{III}.      
 \end{align*}
We choose $\beta = \frac{\Delta t}{\delta}$
with $\delta$ sufficiently large so that there holds $\mathcal{O}^\ast \cap \Omhn{n} \subset \Sdelta(\Omhn{n})$ and $\delta \lesssim \Delta t + h$ and hence $\beta \simeq \frac{\Delta t}{\Delta t + h}$. We now succesively bound the terms III, II and I:\\
$\bullet$ \underline{III:} 
For III we recall \cref{lem:projstab} and use $\mathcal{O}^\ast_{2+} \subset \Orn{2}{n-2}$ so that
$$
\text{III} 
%\leq 12 c_{\eqref{eq:proj:b}} h^4 \left( \Vert \nabla w_h \Vert_{\Orn{0,+}{n}}^2 +  \Vert \nabla \Pin{n-1} v_h \Vert_{\Orn{0,+}{n}}^2 \right)
\lesssim (1 + h \Delta t^{-1} ) h^4 \left( \Vert \nabla w_h \Vert_{\Orn{1}{n-1}}^2 + c_{\ref{lem:projstab}b} \Vert \nabla v_h \Vert_{\Orn{2}{n-2}}^2 \right) % \lesssim h^4 
$$
Now applying \eqref{eq:stripboundb} and making use of \cref{ass:hhdt} with $h$ sufficiently small we can bound III with }
\begin{subequations}
\newc{   
\begin{equation} 
\text{III}\leq \frac{c_{\ref{lem:tupleest}b}}{4} \nu \Delta t R_2 + c_{\text{III}}\nu K\Delta t R_3     \label{eq:proof66a}   
\end{equation}
for a constant $c_{\text{III}} > 0$ that is indepedent of $h$, $\Delta t$ and $n$.\\ 
$\bullet$ \underline{II:} 
%... %yielding the first contributions 
% $c_{\ref{lem:tupleest}b,1} = 12  c_{\eqref{eq:proj:b}} c_{A\ref{ass:hhdt}} \max \{1, c_{\ref{lem:projstab}b}\}$ and $c_{\ref{lem:tupleest}c,1} = c_{\ref{lem:tupleest}b,i} / \nu$.
For II we have that for the chosen $\delta$ and $\beta$ there holds $\mathcal{O}^\ast \cap \Omhn{n} \subset \Sdelta(\Omhn{n})$ and with \eqref{eq:stripbounde} and any $\beta^\ast > 0$ (and \cref{lem:projstab}) further 
\begin{align*}   
\text{II} \leq  (1+{\beta^{\ast}}^{-1}   \nu^{-1}) & \Delta t \bdftwonorm{w_h}{\Pin{n-1} v_h}{\Omhn{n}}{2} + \tilde \beta^{\ast} \tilde c_{\text{II}} \nu \Delta t  \left( \Vert \nabla w_h \Vert_{\Orn{1}{n-1}}^2 \! + \! c_{\ref{lem:projstab}b} \Vert \nabla v_h \Vert_{\Orn{2}{n-2}}^2 \right)
\end{align*}  
Applying \eqref{eq:stripboundb} and choosing $\beta^{\ast}$ sufficiently small yields
\begin{align}   
  & \Longrightarrow   \text{II} \leq (1 + \frac{c_{\text{II}}}{c_{\ref{lem:tupleest}b}} ) \Delta t \cdot I  + \frac{c_{\ref{lem:tupleest}b}}{4} \nu \Delta t  \ (R_2 +  R_3) \label{eq:proof66b}  
\end{align}  
for constants $\tilde c_{\text{II}}, c_{\text{II}}$ that are independent of $h$, $\Delta t$ and $n$.\\ 
%
%It remains to bound $ (1 + \Delta t) \bdftwonorm{ w_h}{ \Pin{n-1} v_h}{\Omhn{n}}{2}$ with r.h.s. terms. 
$\bullet$ \underline{I:} }
To go from $\Omhn{n}$ to $\Omhn{n-1}$ we exploit $\Omhn{n}\subset \Orn{1}{n-1}$ and apply \eqref{eq:stripboundc} from \Cref{lem:stripnorm} with $r=1$:
 \begin{align*}
    \bdftwonorm{ w_h}{ \Pin{n-1} v_h}{\Omhn{n}}{2}
   % \Vert(w_h, \Pin{n-1} v_h)\Vert_{\Om_{h}^{n}}
   \leq & (1 + c_{\ref{lem:stripbound}a} \dt) \bdftwonorm{w_h}{\Pin{n-1} v_h}{\Om_{h}^{n-1}}{2} \\
    & + (c_{\ref{lem:stripbound}b} \nu  \dt + c_{\eqref{eq:proj:b}} h^4) \bdftwonorm{\nabla w_h}{\nabla \Pin{n-1} v_h}{\Orn{1}{n-1}}{2} \\
    & + c_{\ref{lem:stripbound}c} \dt (8 \K \srn{1}{n-1}(w_h,w_h) + 2 \K \srn{1}{n-1}(\Pin{n-1}v_h,\Pin{n-1}v_h) )
 \end{align*}
 In the last step we used $\srn{1}{n-1}(2a-b, 2a-b) \leq 8 \srn{1}{n-1}(a,a) + 2 \srn{1}{n-1}(b,b)$ for $a,b \in \mathcal{V}_2^{n-1}$. 
 %The first part on the r.h.s. is already in the desired form with $c_{\ref{lem:tupleest}a}:=c_{\ref{lem:stripbound}a}$. 
 %We hence continue with the second part. 
 After splitting the terms in the BDF2 norm, we use \eqref{eq:projstab2} 
 to bound $\| \nabla \Pin{n-1} v_h \|_{\Orn{1}{n-1}}^2$ with $\| \nabla v_h \|_{\Orn{2}{n-2}}^2$,
 and \eqref{eq:stripboundb} from \Cref{lem:stripbound}
 \begin{align*}
   \bdftwonorm{&\nabla w_h }{\nabla \Pin{n-1} v_h}{\Orn{1}{n-1}}{2} \!\!\!
   \leq 9  \| \nabla w_h\|_{\Orn{1}{n-1}}^2 \!\!+\!
  2 \| \nabla \Pin{n-1} v_h \|_{\Orn{1}{n-1}}^2 \\
\leq & 9  \| \nabla w_h\|_{\Orn{1}{n-1}}^2 \!\!+ 2 c_{\eqref{eq:projstab2}} c_{\eqref{eq:stripboundb}} \big(
       \| \nabla v_h \|_{\Omhn{n-2}}^2 \!\!+\!\! \K \srn{2}{n-2}(v_h,v_h) \big)\\
\leq & 9 c_{\eqref{eq:stripboundb}} (\| \nabla w_h\|_{\Omhn{n-1}}^2 \!\!+
\K \srn{1}{n-1}(w_h,w_h) )
\! +
\!\! 2 c_{\eqref{eq:projstab2}} c_{\eqref{eq:stripboundb}} \big(
      \| \nabla v_h \|_{\Omhn{n-2}}^2 \!\!+\! \K \srn{2}{n-2}(v_h,\! v_h) \! \big)\!.
 \end{align*}
 %{\color{red}Missing step from $\Orn{1}{n-1}$ to $\Omhn{n-1}$}
 Finally, it only remains to bound the ghost penalty stabilization term on $\Pin{n-1} v_h$. We use \Cref{cor:projstab} and \eqref{eq:stripboundb} from \Cref{lem:stripnorm}, so that
\begin{align*}
  \srn{1}{n-1}(&\Pin{n-1}v_h, \Pin{n-1}v_h) \leq 
                                                 c_{\ref{cor:projstab}a} \srn{2}{n-2}(v_h,v_h) + c_{\ref{cor:projstab}b} h^2 \|\nabla v_h\|_{\Orn{2}{n-2}}^2 \\
& \leq 
(c_{\ref{cor:projstab}a} + c_{\ref{cor:projstab}b} c_{\eqref{eq:stripboundb}} h^2 \K ) \srn{2}{n-2}(v_h,v_h) + c_{\ref{cor:projstab}b} c_{\eqref{eq:stripboundb}} h^2 \|\nabla v_h\|_{\Omhn{n-2}}^2.  
\end{align*}
Now, we need to collect all pieces together. 
\newc{
Ensured by \Cref{ass:hhdt} we can bound 
$(c_{\ref{lem:stripbound}b} \nu  \dt + c_{\eqref{eq:proj:b}} h^4) \leq 2 c_{\ref{lem:stripbound}b} \nu  \dt$
and choose $c_{\ref{lem:stripbound}b}$ sufficiently small to obtain
\begin{align}
 I \leq  (1 + c_{I,1}\Delta t) R_1 + \frac{c_{\ref{lem:tupleest}b}}{4}  \nu \Delta t R_2 + c_{I,3} K \Delta t R_3  \label{eq:proof66c}  
\end{align}
for constants $c_{\text{I},1}, c_{\text{I},3}$ that are independent of $h$, $\Delta t$ and $n$, but depend on $c_{\ref{lem:tupleest}b}$.

$\bullet$ \underline{I+II+III:} 
Merging \eqref{eq:proof66a}--\eqref{eq:proof66c} we obtain 
\begin{align*}
  \text{I}+\text{II}+\text{III} \leq &
  (1 + (1+\frac{c_{\text{II}}}{c_{\ref{lem:tupleest}b}}) \Delta t ) I + 
  \frac{c_{\ref{lem:tupleest}b}}{2}\nu \Delta t R_2 + (\frac{c_{\ref{lem:tupleest}b}}{4} + c_{\text{III}} )\nu K \Delta t  R_3    \\
   = &(1 + (1+\frac{c_{\text{II}}}{c_{\ref{lem:tupleest}b}}) \Delta t ) ( 1 + c_{\text{I},1} \Delta t) R_1\\
    & +  (3 + (1+\frac{c_{\text{II}}}{c_{\ref{lem:tupleest}b}}) \Delta t ) \frac{c_{\ref{lem:tupleest}b}}{4}\nu \Delta t R_2\\
    & +  ((1 + (1+\frac{c_{\text{II}}}{c_{\ref{lem:tupleest}b}}) \Delta t ) c_{\text{I},3} + ( \frac{c_{\ref{lem:tupleest}b}}{4} + c_{\text{III}} )\nu   ) K \Delta t R_3. 
\end{align*}  
Hence, for sufficiently small $\Delta t$ the claim follows. 
% the conditions
% \begin{align*}
% \max\{
%   18 c_{\ref{lem:stripbound}b} c_{\eqref{eq:stripboundb}}, ~~
%   4 c_{\ref{cor:projstab}b}  c_{\ref{lem:stripbound}b}
%   + 2 c_{\ref{lem:stripbound}c} c_{\ref{lem:stripbound}b} c_{\ref{cor:projstab}b} c_{\eqref{eq:stripboundb}} K/\nu h^2 \!
% \} \!\!\leq \!\! c_{\ref{lem:tupleest}b}, \\
% \max\{
%   18 c_{\ref{lem:stripbound}b} c_{\eqref{eq:stripboundb}} \nu \!+\! 8 c_{\ref{lem:stripbound}c} , 
%   4 c_{\ref{lem:stripbound}b} c_{\ref{cor:projstab}b} c_{\eqref{eq:stripboundb}} \nu  \!+\! 4 c_{\ref{lem:stripbound}c} (c_{\ref{cor:projstab}a} \!+\! c_{\ref{cor:projstab}b} c_{\eqref{eq:stripboundb}} h^2\! K) \!
% \} \!\! \leq \!\! c_{\ref{lem:tupleest}c}. 
% \end{align*}
% We note that the latest terms in the $\max$ can always be neglected for sufficiently small $h$. In the remaining terms only $c_{\ref{lem:stripbound}b}$
% and $c_{\ref{lem:stripbound}c}$ depend on the choice of $\theta$ in \cref{lem:stripbound}. Hence, for every choice of $c_{\ref{lem:tupleest}b}$ we can choose $\theta$ and $c_{\ref{lem:tupleest}c}$ to meet the claim.
}
\end{subequations} 
\end{proof}

\subsection{Stability analysis}
\begin{theorem} \label{thm:stability}
	The solution $\{u_h^n\}$ of \eqref{eq:disceq} with $r=2$ satisfies the stability bound 
	\begin{align*}
		\bdftwonorm{u_h^N&}{\Pin{N} u_h^{N-1}}{\Omhn{N}}{2}\!+\dt \! \sum_{n=2}^N \! \big( \nu \|\nabla u_h^n\|_{\Omhn{n}}^2 + 2 \gamma \srn{2}{n}(u_h^n,u_h^n) \big) \lesssim \exp(c_{\ref{thm:stability}} t_N) R^0\text{ with}\\
		 R^0 \! := \! \bdftwonorm{u_h^{1}&}{\Pin{1} u_h^{0}}{\Omhn{1}}{2} \! + \dt \Big( \sum_{n=2}^N \|g^n\|_{\Omhn{n}}^2 + %\\ 
		 \sum_{n=0}^1 \big(\nu\|\nabla u_h^{n}\|_{\Omhn{n}}^2 + \K \srn{2}{n}(u_h^{n},u_h^{n})\big)\! \Big)
     ,
	\end{align*}
  for $c_{\ref{thm:stability}} := c_{\ref{lem:tupleest}a} + \frac12 + 4 \xxi$ with $\xxi$ as in \eqref{eq:timestep}, i.e. indepedent of $\finaltime$, $h$ or $\dt$.
\end{theorem}

\begin{proof}
	We test \eqref{eq:disceq} with $v_h = 4 \dt u_h^n$ and apply \eqref{eq:bdf2normest} which yields
	\begin{align} \label{eq:}
		\bdftwonorm{u_h^n}{\Pin{n} u_h^{n-1}}{\Omhn{n}}{2} &+ 4 \dt \bh{n}(u_h^n,u_h^n) + 4 \dt \gamma \srn{2}{n}(u_h^n,u_h^n) \nonumber	\\
		&\le \bdftwonorm{\Pin{n} u_h^{n-1}}{\Pin{n} \Pin{n-1} u_h^{n-2}}{\Omhn{n}}{2} + 4 \dt \fhn{n}(u_h^n). 
	\end{align}
	Recall the lower bound of $\bh{n}(\cdot,\cdot)$ from \eqref{eq:bilinear}. We apply \Cref{lem:tupleest} on the r.h.s. followed by Young's inequality with $\beta > 0$ and Cauchy-Schwartz applied to $\fhn{n}$:
	\begin{align*}
		(1 &- 4 \dt \xxi) \bdftwonorm{u_h^n}{\Pin{n} u_h^{n-1}}{\Omhn{n}}{2} \!+\! 2 \dt \nu \|\nabla u_h^n\|_{\Omhn{n}}^2 \!+\! 4 \dt \gamma \srn{2}{n}(u_h^n,u_h^n) \\
		&\le (1 \!+\! c_{\ref{lem:tupleest}a} \dt) \bdftwonorm{u_h^{n-1}}{\Pin{n-1} u_h^{n-2}}{\Omhn{n-1}}{2} \nonumber
		\!+\! c_{\ref{lem:tupleest}b} \nu \dt \big(\|\nabla u_h^{n-1}\|_{\Omhn{n-1}}^2 \!+\! \|\nabla u_h^{n-2}\|_{\Omhn{n-2}}^2\big) \nonumber \\
		&\quad \!+\! c_{\ref{lem:tupleest}c} \dt \big(\K \srn{2}{n-1}\!(u_h^{n-1}\!\!,u_h^{n-1}) \!+\! \K \srn{2}{n-2}\!(u_h^{n-2}\!\!,u_h^{n-2})\big) 
     \!+\! 2 \dt \big( \beta^{-1} \|g^n\|_{\Omhn{n}}^2 \!+\! \beta \|u_h^n\|_{\Omhn{n}}^2\big). \nonumber
	\end{align*}
	Summing over $n=2,...,N \le N$, choosing $c_{\ref{lem:tupleest}b} \le \frac12$, and assuming $\gamma \ge c_{\ref{lem:tupleest}c} \K$ yields
	\begin{align}
		(1  - 4 \dt& \xxi - 2 \dt \beta) \bdftwonorm{u_h^N}{\Pin{N} u_h^{N-1}}{\Omhn{N}}{2} + \dt \sum_{n=2}^N \nu \|\nabla u_h^n\|_{\Omhn{n}}^2 + 2 \dt \sum_{n=2}^N \gamma \srn{2}{n}(u_h^n,u_h^n) \nonumber \\
		&\le \bdftwonorm{u_h^{1}}{\Pin{1} u_h^{0}}{\Omhn{1}}{2} \!+\! \dt \sum_{n=0}^1 \Big( \nu \|\nabla u_h^{n}\|_{\Omhn{n}}^2 \!+\! c_{\ref{lem:tupleest}c} \K \srn{2}{n}(u_h^{n},u_h^{n}) \Big)  \label{eq:stab:summedup}\\
		&\quad \!+\! \left(c_{\ref{lem:tupleest}a} \!+\! 4 \xxi \!+\! 2 \beta \right) \dt \sum_{n=2}^{N} \bdftwonorm{u_h^{n-1}}{\Pin{n-1} u_h^{n-2}}{\Omhn{n-1}}{2} \!+\! 2 \dt \beta^{-1} \sum_{n=2}^N \|g^n\|_{\Omhn{n}}^2. \nonumber
	\end{align}	
	Finally by choosing $\beta=\frac14$ and applying the discrete Gronwall's lemma with $\dt \xxi \le \frac{1}{16}$ we obtain the result. 
\end{proof}
Let us note that parts of the stability analysis of the unfitted BDF2 method have been treated in a much simplified setting in \cite[Section 5.2.1]{jinmaster}.

\subsection{Error estimates}
\begin{theorem} \label{thm:error}
	For $u \in \LinfW{k+1}$ with $\partial_t u \in \LinfH{k}$ and $\partial_t^3 u \in \L^\infty(\Q)$ the solution to \eqref{eq:eqn} with source term $g \in \LinfW{1}$, the numerical solution $\{u_h^n\}$ of \eqref{eq:disceq} with $r=2$ fulfills the error estimate
	\begin{align*}
		\| \err{N} \|_{\Omhn{N}}^2 &+ \dt \sum_{n=2}^N \left(\frac{\nu}{2} \| \nabla \err{n} \|_{\Omhn{n}}^2 + \gamma \srn{2}{n}(\err{n},\err{n})\right) \\ 
		&\lesssim \exp{(c \finaltime)} \Big( \big(\dt^4 + \K h^{2k} + h^{2q}\big) \mathcal{R}_{\ref{lemma:consist}}^2 + h^{2k} \mathcal{R}_{\ref{lemma:approx}}^2  \Big)
	\end{align*}
	with $c := c_{\ref{lem:tupleest}a} + 4 \xxi + 4$ independent of $h$, $\dt$ and $\finaltime$. 
\end{theorem}
\begin{proof}
	By the error splitting \eqref{eq:split} we have with the interpolation error estimate
	\begin{equation} \label{eq:interror}
		\| \err{N} \|_{\Omhn{N}}^2 \lesssim \|\erri{N}\|_{\Omhn{N}}^2 + \|\errd{N}\|_{\Omhn{N}}^2 \lesssim h^{2k} \|u^{N}\|_{\H^{k+1}(\Omn{N})}^2 + \|\errd{N}\|_{\Omhn{N}}^2, 
	\end{equation}
	hence we only need to bound the last term $\|\errd{N}\|_{\Omhn{N}}$. 
	The error equation for $\errd{n}$, \eqref{eq:discerreq}, coincides with \eqref{eq:disceq} when replacing $u_h^n$ with $\errd{n}$ and $\fhn{n}(\cdot)$ with $\consist{n}(\cdot) + \interpol{n}(\cdot)$. Except for the treatment of the r.h.s. term, we proceed as in \Cref{thm:stability}.
	We turn our attention to $\consist{n}(\cdot) + \interpol{n}(\cdot)$, recall \Cref{lemma:consist} and \Cref{lemma:approx} and sum up:
	\begin{align}
		\sum_{n=2}^{N} \Big(\consist{n}(\errd{n}) &+ \interpol{n}(\errd{n})\Big) \lesssim (\dt^2 + \K^\frac12 h^k + h^q) \mathcal{R}_{\ref{lemma:consist}} \sum_{n=2}^N \enorm{\errd{n}}{n}
		+ N^{\frac12} h^k 
		\mathcal{R}_{\ref{lemma:approx}}  \|\errd{n}\|_{\ltwo} \nonumber \\
		&\lesssim N (\dt^4 + \K h^{2k} + h^{2q}) \mathcal{R}_{\ref{lemma:consist}}^2 + N h^{2k} \mathcal{R}_{\ref{lemma:approx}}^2  \nonumber \\
		&\qquad+ \sum_{n=1}^{N} \Big( \bdftwonorm{\errd{n}}{\Pin{n} \errd{n-1}}{\Omhn{n}}{2} + \frac{\nu}{2} \|\nabla \errd{n}\|_{\Omhn{n}}^2 + \gamma \srn{2}{n}(\errd{n},\errd{n}) \Big) \nonumber
	\end{align}
	Analogously to \eqref{eq:stab:summedup} we then arrive at
	\begin{align*}
		\big(&(1 - 4 \dt (1+\xxi)\big) \bdftwonorm{\errd{N}}{\Pin{N} \errd{N-1}}{\Omhn{N}}{2} + \dt \sum_{n=2}^{N} \frac{\nu}{2} \|\nabla \errd{n}\|_{\Omhn{n}}^2 + \dt \sum_{n=2}^{N} \gamma \srn{2}{n}(\errd{n},\errd{n}) \\[-1ex]
		\lesssim & \bdftwonorm{\errd{1}}{\Pin{1} \errd{0}}{\Omhn{1}}{2} + \dt \sum_{n=0}^1 \Big( \nu \|\nabla \errd{n}\|_{\Omhn{n}}^2 + \gamma \srn{2}{n}(\errd{n},\errd{n}) \Big) \nonumber 
		+ \finaltime (\dt^4 + \K h^{2k} + h^{2q}) \mathcal{R}_{\ref{lemma:consist}}^2 \\[-1ex]
    & + \finaltime h^{2k} \mathcal{R}_{\ref{lemma:approx}}^2 + \left(c_{\ref{lem:tupleest}a} + 4 \xxi + 4\right) \dt \sum_{n=2}^{N} \bdftwonorm{\errd{n-1}}{\Pin{n-1} \errd{n-2}}{\Omhn{n-1}}{2}. \nonumber
	\end{align*}	
	 We apply discrete Gronwall's inequality with $\dt (1+\xxi) \le \frac18$ and make use of $\|\errd{N}\|_{\Omhn{N}} \le \bdftwonorm{\errd{N}}{\Pin{N} \errd{N-1}}{\Omhn{N}}{}$, \eqref{eq:interror} and set $c := c_{\ref{lem:tupleest}a} + 4 (\xxi+1)$ to obtain the claim.
\end{proof}
\new{
\begin{remark}
  Let us stress that the result in \cref{thm:error} is not robust in the diffusivity $\nu$ because of the exponential term $c$ which depends on $\xi$, cf. also \cref{rem:nu}. 
  Furthermore, we only considered an $h$-version analysis here. An extension to a $p$ or even $hp$-version is outside the scope of this study due to the excessive use of inverse inequalities, e.g. in the analysis of the transfer operator and the ghost penalty mechanism.
\end{remark} 
}

\begin{remark}[Impact of the anisotropy factor $\K$]
The previous error estimate involves the factor $\K h^{2k} \lesssim h^{2k} + \Delta t \cdot h^{2k-1}$. Hence, at first glance it seems that an anisotropy between space and time grid resolution, i.e. when $\Delta t / h \to \infty$ for $h, \Delta t \to 0$, can have a negative impact on the convergence rate. For $k>1$ we can estimate
  $ \Delta t \cdot h^{2k-1} \leq \Delta t^2 \cdot h^k + h^{3k-2} \lesssim \Delta t^4 + h^{2k}$ and can hence conclude that the factor $\K$ has no influence on the convergence rates. However, for $k=1$ we can indeed have that $\Delta t \cdot h$ converges slower than $\Delta t^4 + h^2$ for $\Delta t / h \to \infty$ for $h, \Delta t \to 0$.
\end{remark}

\section{Numerical experiments}
\label{sec:experiments}
In this section we present numerical examples for the method proposed above. The results verify the order of accuracy corresponding to the error analysis, demonstrate the stability with respect to the variation of discretization parameters, and show the robustness to handle evolving domains even in complex configurations. 
 All examples are implemented with the finite element package \texttt{NGSolve}\cite{ngsolve} and its Add-on \texttt{ngsxfem}\cite{ngsxfem}. For reproducibility of all subsequent numerical results, we provide the generating code and instructions on how to obtain them in the repository \url{https://gitlab.gwdg.de/lehrenfeld/repro-isop-unf-bdf-fem}.

In slight contrast to the analysis, we make two changes in the discretization. Firstly, use the following bilinear form instead of $\bh{n}(u_h,v_h)$:
%\begin{align}
$\int_{\Omhn{n}} \nu \nabla u_h \cdot \nabla v_h \; dx + \int_{\Omhn{n}} (\w^{\ext} \cdot \nabla u_h)v_h \; dx + \int_{\Omhn{n}} (\operatorname{div} \w^{\ext}) u_h v_h \; dx$.  Secondly, we do not include the additional element layers for the extension, i.e., we skip the "+" layers defined in \Cref{sec:discrete neigh}. However, in all subsequent numerical examples we ensured that the desired inclusion properties hold.
In the examples we consider an additional source term in order to simplify the construction of manufactured solutions. We hence add a linear form 
%\begin{align}
$\fhn{n}(v_h) := \int_{\Omhn{n}} g v_h dx$
%\end{align}
to the right hand side of the discrete variational formulation.

%\subsubsection{Error measurement}
Related to the previous error analysis, we investigate the numerical errors in terms of the following discrete space-time norms
\begin{align}
% \|u_h - u^{\ext}\|_{\L^2(\L^2)}^2 &:= \sum_{n=1}^N \Delta t \|u_h - u^{\ext}\|_{\L^2(\Omhn{n})}^2, \\
\| \err{} \|_{\L^2(\H^1)}^2 \!:=\! \sum_{n=1}^N \!\! \Delta t \|\nabla \err{}\|_{\L^2(\Omhn{n})}^2, \ 
\|\err{}\|_{\L^\infty(\L^2)} \!\!:=\!\!\! \max_{n=1,...,N} \|\err{}\|_{\L^2(\Omhn{n})}, \ \err{} = u - u_h,
% \|u_h \!\!-\! u^{\ext}\|_{\L^\infty(\H^1)} &:= \max_{n=1,...,N} \|\nabla (u_h - u^{\ext})\|_{\L^2(\Omhn{n})}, \\
% \|u_h - u^{\ext}\|_{\L^2(\finaltime)} &:= \|u_h - u^{\ext}\|_{\L^2(\Om^N_h)}, \\
% \|u_h - u^{\ext}\|_{\H^1(\finaltime)} &:= \|\nabla (u_h - u^{\ext})\|_{\L^2(\Om^N_h)}, 
\end{align}
where we recall that $u$ is identified with its extension $u^{\ext}$ to the suitable neighborhood.
For each example we start with an initial (quasi-uniform) spatial and (uniform) temporal mesh with mesh size $h_0$ and $\Delta t_0$, respectively, and then apply successive regular refinements in space and time. We denote the corresponding refinement levels in space by $L_x$,  and those in time by $L_t$, s.t. $h = h_0 \cdot 2^{- L_x}$, $\Delta t = \Delta t_0 \cdot 2^{- L_t}$.

\subsection{Kite geometry}
\label{sec:ex2}

%(P2-BDF2, P3-BDF3)
In this first example we consider a disk deforming towards a kite shape, cf. \Cref{fig:bubble}. 
\begin{figure}\label{fig:bubble} 
  \centering
  %\todo[inline]{deactivated for now to reduce compile time...}
  \begin{center}
  \small $-1$ \includegraphics[width=0.50\textwidth, trim = 8.5cm  67.5cm 8.5cm 0.5cm, clip = True]{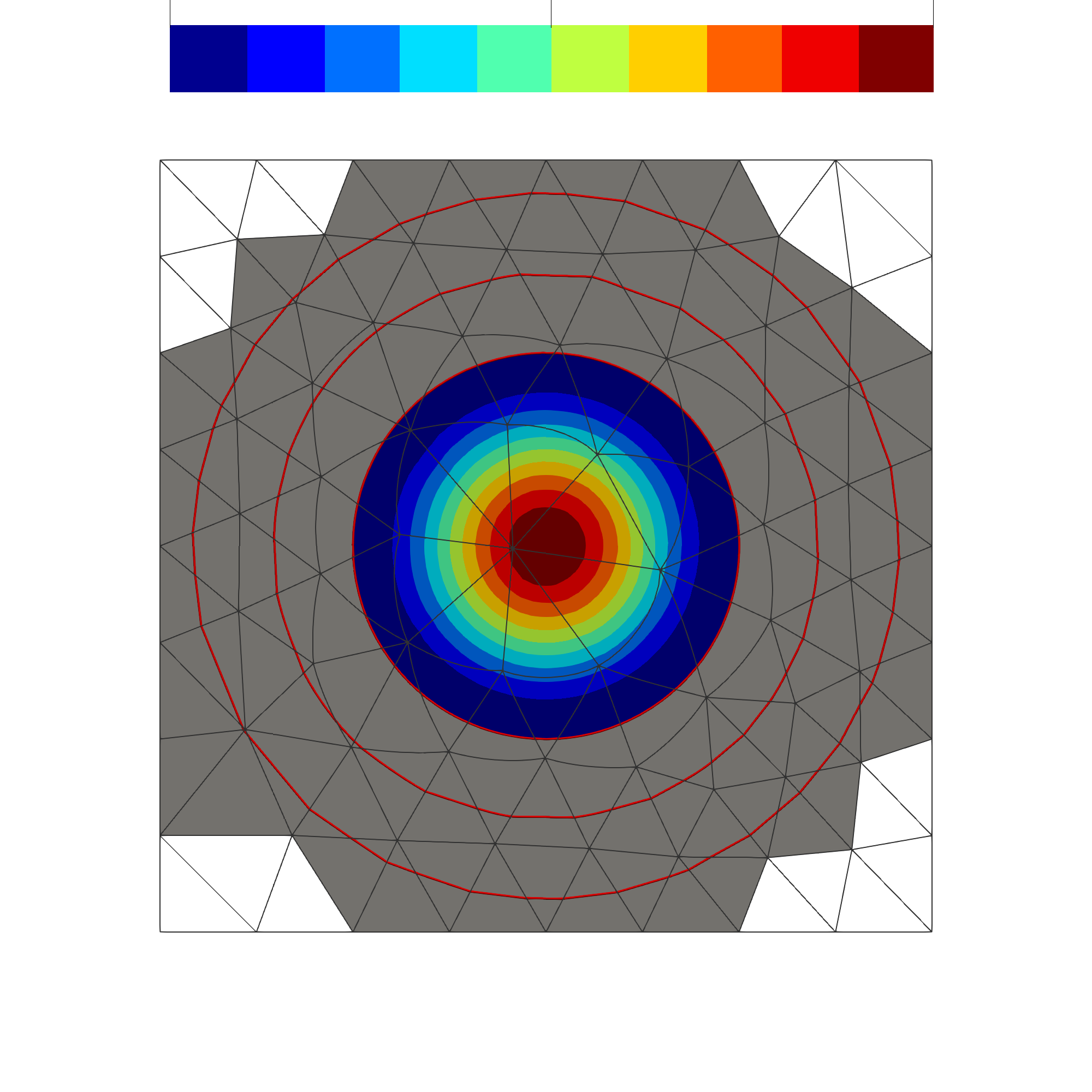} $1$\\
  \includegraphics[width=0.315\textwidth, trim = 10.25cm 8.5cm 10.25cm 9.5cm, clip = True]{kite1.png}
  \includegraphics[width=0.315\textwidth, trim = 10.25cm 8.5cm 10.25cm 9.5cm, clip = True]{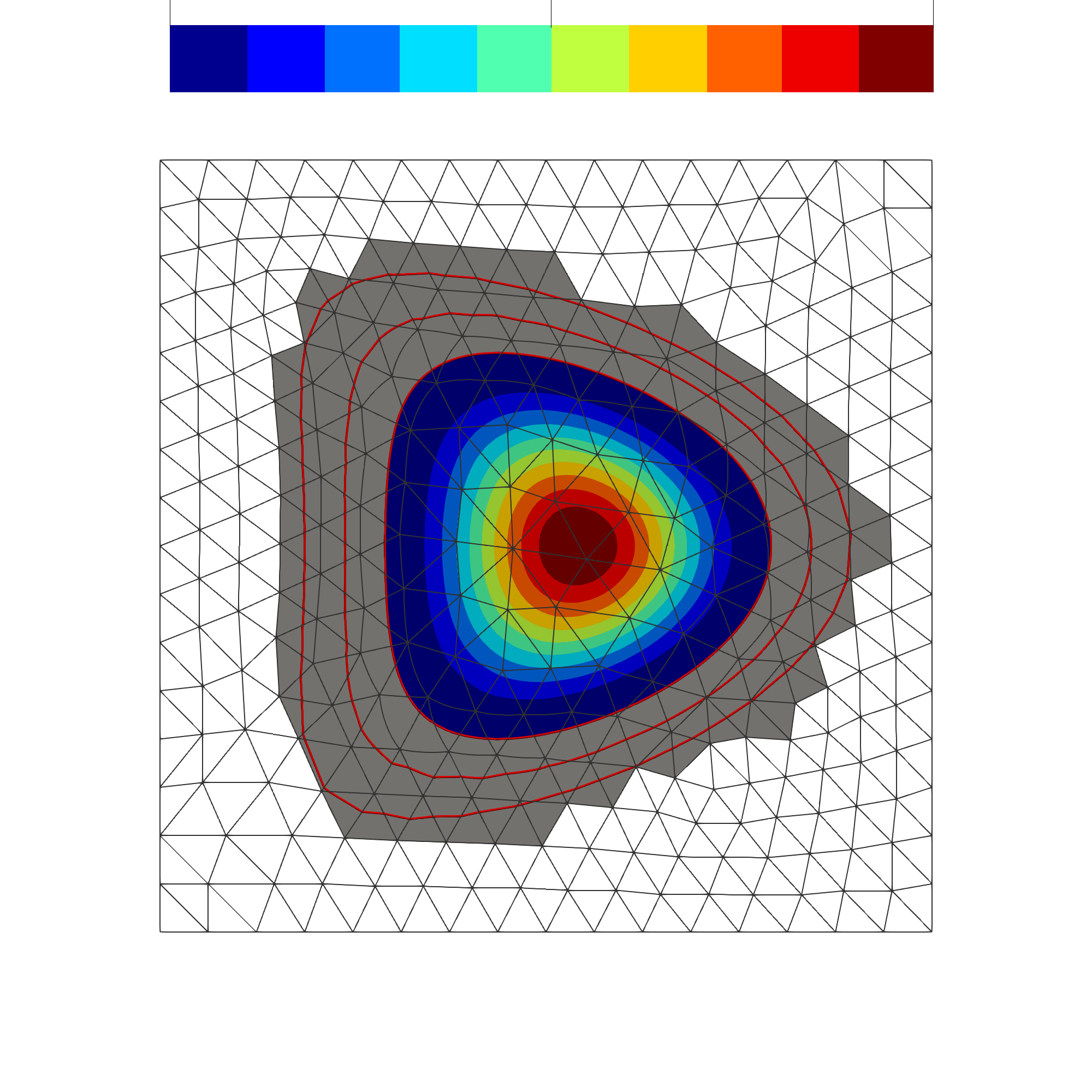}
  \includegraphics[width=0.315\textwidth, trim = 10.25cm 8.5cm 10.25cm 9.5cm, clip = True]{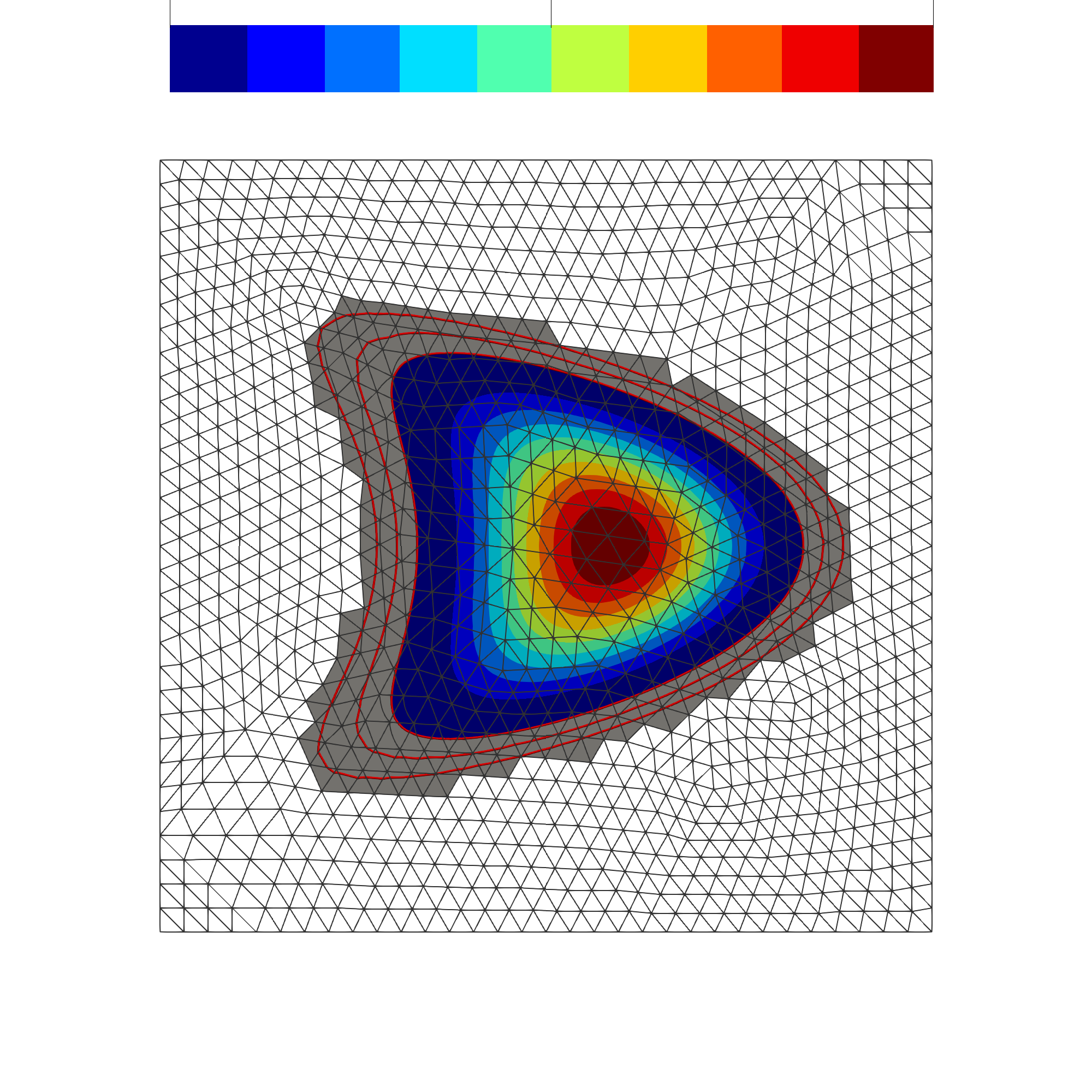}
  \end{center}
  \vspace*{-0.4cm}
  \caption{Numerical example from \Cref{sec:ex2} with $k=q=3$ and $r=2$ (BDF-2): Mesh, active mesh (grey), neighborhood extension (red line) and discrete solutions on $\Omega(t)$ for $L_x=0$, $L_t=3$, $t=0$ (left), for $L_x=1$, $L_t=4$, $t=\finaltime/2$ (center) and $L_x=2$, $L_t=5$, $t=\finaltime$ (right).}
\end{figure}
%\subsubsection{Configuration}
The background domain is fixed to $\B = (-1,1)^2$ and the final time in the simulation to $\finaltime=1$. 
The evolution of the geometry is represented by the following level set function $\phi$:
% for the shifted position and $R \equiv 0.5$ for a constant radius
\begin{align*}
\phi(\x,t) = \|\x - \rho(\x,t)\|_2 - \frac12 \text{ with }\rho(\x,t) = \w(\x)\!\cdot\! t, ~\w \!=\! \big( 1/6 - 5/3 y^2, 0 \big)^{\!T}. 
\end{align*}
Note that $\phi$ is not a signed distance function for $t>0$. 
The r.h.s. function $g$ is set so that $u(\x,t) = \cos{\frac{\pi}{R}\|\x-\rho(\x,t)\|_2}$ (which has $-\nabla u \cdot \n = 0$, on $\G(t), ~t \in [0,\finaltime]$) solves the PDE. The initial spatial and temporal resolutions are $h_0 = 0.25$ and $\dt_0 = \finaltime$. 

In \Cref{fig:b_p2bdf2} the convergence of the error norms $\L^{\infty}(\L^2)$ and $\L^2(\H^1)$ are displayed for $r=k=q=2$ and we observe that the convergence rate in time is $\Delta t^2$ in both norms whereas the convergence rate in space is $h^2$ in $\L^{2}(\H^1)$ and $h^3$ in $\L^{\infty}(\L^2)$. These results are in agreement with the previous analysis except for the $h^3$ convergence in $\L^{\infty}(\L^2)$ which is even one order better than predicted. 
 
\begin{figure}\label{fig:b_p2bdf2}
  \centering
  \includegraphics[trim=0.075cm 0cm 15.1cm 0.5cm, clip=true, width=0.97\textwidth]{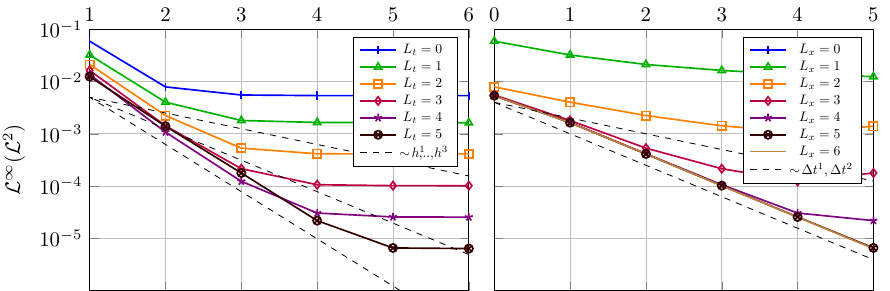} \\
  \includegraphics[trim=15.175cm 0cm 0.0cm 0.0cm, clip=true, width=0.97\textwidth]{plot_kite_P2BDF2.pdf}  
  \vspace*{-0.3cm}
  \caption{$\L^\infty(\L^2)$- and $\L^2(\H^1)$-errors of BDF-2 discretization for $k=q=2$ and $c_\gamma=0.1$ for several levels of mesh and time refinements in the example from \Cref{sec:ex2}.}
\end{figure}

\subsection{Two-phase mass transport}
\label{sec:ex3}
We finally consider a more complex example that is beyond the scope of the previous analysis and third order methods in space and time. In this example we consider a two-phase interface problem as in \cite{L_PHD_2015}.
A circular interface that separates two materials is moving within the background domain $\Omega = (0,2)^2$ for $t \in [0,\finaltime]$, $\finaltime = \frac12$. The two domains are described by the level set function
$\phi(\x,t) = \Vert \x - \rho(t) \Vert_2 - \frac13$ with $\rho(t) = (\frac12+\frac{1}{4\pi}\sin(2\pi t),1)$ which describes a circle moving with time-dependent velocity in the horizontal direction. We define $\Omega_1(t) = \{ \phi(\x,t) < 0 \}$, $\Omega_2(t) = \{ \phi(\x,t) > 0 \}$. 
On both subdomains the convection-diffusion equation $\partial_t u + \w \cdot \nabla u - \nu_i \Delta u = g$, $\w = \partial_t \rho(t)$ is solved. We prescribe Dirichlet data on $\partial \Omega$ and writing $[\![ w ]\!] = w_1 - w_2$ for the jump, and impose the interface conditions $[\![-\nu \nabla u \cdot n]\!]=0$, $[\![\beta u]\!]=0$ on $\Gamma(t) = \overline{\Omega}_1(t)\cap\overline{\Omega}_2(t)$ for $\nu_i,~i=1,2$ the diffusion constants and $\beta_i,~i=1,2$ the Henry weights for the Henry interface condition. We choose $(\nu_1,\nu_2) = (10,20)$ and  $(\beta_1,\beta_2) = (2,1)$.
The corresponding boundary data and r.h.s. data $f$ is prescribed such that $u(\x,t) = \sin(\pi t) u_i(\Vert \x - \rho(t) \Vert)$ with $u_1(r) = a + b r^2$ and $u_2(r) = \cos(\pi r)$, $a \approx 1.1569$, $b \approx -8.1621$, is the unique solution. 

For the discretization we use an unfitted discretization similar to \eqref{eq:disceq}, but for both subdomains with a Nitsche formulation for the coupling through the interface conditions. 
The discretization takes the form: 

Find $u_h^n = (u_{1,h}^n,u_{2,h}^n) \in \Vrn{r,1}{n} \times \Vrn{r,1}{n}, ~n = r,...,N$ for given $u_{i,h}^0 \in \Vrn{r,i}{0}$,..., $u_{i,h}^{r-1} \in \Vrn{r,i}{r-1},i=1,2$, s.t. for all $v_h = (v_{1,h},v_{2,h}) \in \Vrn{r,1}{n} \times \Vrn{r,2}{n}$ there holds
\begin{align}
  & \sum_{i=1,2} \beta_i \Big\{ \int_{\Omihn{i}{n}} \dtBDFr{r}(u_{i,h}^n,  ..., \Pimn{n-r}{n} u_{i,h}^{n-r})  v_h ~dx  + \bih{i}{n}(u_h^n,v_h) + \gamma \sirn{i}{r}{n}(u_h^n,v_h) \Big\} \nonumber \\
  &\! +\!\! \int_{\Gamma} \!\!-\{\!\!\{ \nu \nabla u_h \!\cdot\! n_{\Gamma}\!\}\!\!\} [\![\beta v_h]\!] \!-\! \{\!\!\{ \nu \nabla v_h \!\cdot\! n_{\Gamma}\!\}\!\!\} [\![\beta u_h]\!] \!+\! \frac{ \!\{\!\!\{\nu\}\!\!\}\! \lambda_N}{h} [\![\beta u_h]\!] [\![\beta v_h]\!] ds \label{eq:disceq:2ph} %\\ 
  = \!\!\!\sum_{i=1,2} \! \beta_i \!\! \int_{\Omihn{i}{n}} \!\!\!\!\! g_i  v_{i,h} dx\!,\! \nonumber
\end{align}
where 
$\Omihn{i}{n}$ is the discrete geometry approximation to $\Om_i(t^n)$ as previously,
$g_i$ is a proper extension of $g|_{\Omega_i}$ to $\Omihn{i}{n}$,
and $\bih{i}{n}(\cdot,\cdot)$, $\sirn{i}{r}{n}(\cdot,\cdot)$ are the bilinear forms analogously to before, but now acting on $\Omihn{i}{n}$ and on a corresponding set of facets $\Firn{i}{r}{n}$, respectively. In the Nitsche terms, we used $\{\!\!\{w\}\!\!\} := \frac{w_1 + w_2}{2} $ and for $\lambda_N$ we choose $40$.
The spatial and temporal resolution are $h_0 = 0.5$ and $\dt_0 = \finaltime$ initially.

%\subsubsection{Convergence in space and time}

In \Cref{fig:b_p3bdf3} the convergence of the error norms $\L^{\infty}(\L^2)$ and $\L^2(\H^1)$ (which are composed as the summation of the corresponding norms on the subdomains $i=1,2$) are displayed for $r=k=q=3$ and we observe that the convergence rate in time is $\Delta t^3$ in both norms whereas the convergence rate in space is $h^3$. The convergence rate in space in the $\L^{\infty}(\L^2)$ norm is even higher. Again, these results are in agreement with the expectations from the case of only one moving domain.

\begin{figure}\label{fig:b_p3bdf3}
  \centering
  \includegraphics[trim=0.075cm 0cm 15.1cm 0.5cm, clip=true, width=0.97\textwidth]{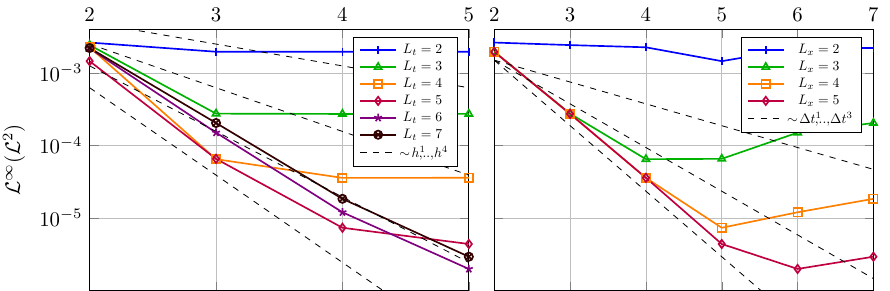} \\
  \includegraphics[trim=15.175cm 0cm 0.0cm 0.0cm, clip=true, width=0.97\textwidth]{plot_intf_P3BDF3.pdf}  
  \vspace*{-0.3cm}
  \caption{$\L^\infty(\L^2)$- and $\L^2(\H^1)$-errors of BDF-3 discretization for $k=q=3$ and $c_\gamma=10$ for several levels of mesh and time refinements in the example from \Cref{sec:ex3}.}
  \vspace*{-0.3cm}
\end{figure}

\section{Conclusion}
\label{sec:conclusion}
In this paper we have extended the numerical method from \cite{LO_ESAIM_2019} for solving PDEs on evolving domains towards \emph{provably} higher order of accuracy in space and time. 
The development of the method in contrast to \cite{LO_ESAIM_2019} consists of the higher-order in space isoparametric unfitted FEM and its analysis on the one hand, and the analysis of the time stepping from an implicit Euler to a high-order BDF scheme on the other hand. 
The method shows several advantages over its competitors aforementioned in \Cref{sec:intro}. 
In comparison to the space-time approach, the proposed method based on the standard isoparametric unfitted FEM involves only spatial integrals and finite element spaces, which leads to an implementational and computational complexity that is comparable to stationary unfitted problems. 
While compared to the characteristic Galerkin scheme, the method trades the extra cost and complexity of the Lagrangian tracking of the geometry (to high-order accuracy) with a comparably simple ghost-penalty-based discrete extension. 

\section*{Acknowledgments}
The authors gratefully acknowledge funding by the DFG (German Research Foundation) within the project ``LE 3726/1-1''. 

\shortonly{\newpage}
\appendix
\section{Selected proofs}\label{proofs}

\begin{proof}[\underline{Proof of \Cref{cor:projstab}}]\label{proof:cor:projstab}
  Let us consider a facet $F^{n} \in \Frn{1}{n}$ with the corresponding uncurved facet $\hat{F} = \Theta^{-n}(F^n)$ and hence $F^{m} = \Theta^{m}(\hat{F}) \in \Frn{2}{m}$. Their adjacent elements of the patches are denoted as $\omF{F^n} = T^{n}_1 \cup T^{n}_2$, $\homega{\hat{F}} = \hat{T}_1 \cup \hat{T}_2$, and $\omF{F^{m}} = T^{m}_1 \cup T^{m}_2$, respectively. Note that $\Theta_{T_1}^m$, $\Theta_{T_1}^n$, $\Theta_{T_2}^m$ and $\Theta_{T_2}^n$ are involved here. Let $u_h = \Pin{n} v_h \in \Vhn{n}$ be the discrete function after projection and recall the notation from \Cref{sec:fulllinear}, i.e. 
  $u_i  = \hat{u}_i \circ \Theta_{{T}_i}^{-n*},~i=1,2$ with $\hat{u}_i := \Ep \big(u_h|_{T^n_i} \circ \Theta_{{T}_i}^n\big)$  and
  $v_i  = \hat{v}_i \circ \Theta_{{T}_i}^{-m*},~i=1,2$ with $\hat{v}_i := \Ep \big(v_h|_{T^{m}_i} \circ \Theta_{{T}_i}^{m}\big)$.
  We further introduce the notation for the properly extended neighbor functions:
  $\hat{u}_j^* := \hat{u}_j \circ \Upsilon^n(\hat{x})$ and 
  $\hat{v}_j^* := \hat{v}_j \circ \Upsilon^m(\hat{x})$ with
  $\Upsilon^n = \Theta_{T_j}^{-n*} \circ \Theta_{T_i}^{n}$ and 
  $\Upsilon^m = \Theta_{T_j}^{-m*} \circ \Theta_{T_i}^{m}$.
  Taking the definition of the ghost penalty from \eqref{eq:ghostpenalty} and exploiting the smallness of the deformations, we find after transformation to $\hat{T}_i$, $\hat{T}_j$
  \begin{equation*}
    h^2 s_{F}^{n}(u_h,u_h) 
    \leq 2 \sum_{i,j=1,2} \int_{\hat{T}_i} (\hat{u}_i - \hat{u}_j^*)^2 d\hat{x} \text{ and } \sum_{i,j=1,2} \int_{\hat{T}_i} (\hat{v}_i - \hat{v}_j^*)^2 d\hat{x} \leq 2 h^2 s_{F}^{m}(v_h,v_h).
  \end{equation*}
   Let $a := \hat{u}_i - \hat{u}^*_j$ and $b := \hat{v}_i - \hat{v}^*_j$. Then with 
  $A := \hat{u}_i - \hat{v}_i$ and $B := \hat{u}^*_j - \hat{v}^*_j$ we have $a-b = A-B$ so that there holds
  \begin{align*}
    a^2 - b^2 = \underbrace{(a-b)}_{(A-B)}(a+b) \leq (A-B)^2 +\frac{1}{4}(a+b)^2
    \leq 2(A^2+B^2) +\frac{1}{2}(a^2+b^2)  \\[-2ex]
    \Longrightarrow \frac12 a^2 - \frac32 b^2 \leq 2(A^2+B^2)
    \Longrightarrow a^2 \leq 3b^2 + 4 (A^2+B^2).
  \end{align*}
  Hence, we obtain 
  \begin{align*}
    s_{F}^{n}(u_h,u_h) & \lesssim s_{F}^{m}(v_h,v_h)  + h^{-2}\sum_{\substack{i,j=1,2 \\ i\neq j}} \Big( \|\underbrace{\hat{u}_i - \hat{v}_i}_{=A}\|_{\hat{T}_i}^2 + \|\underbrace{\hat{u}_i^* - \hat{v}_i^*}_{=B}\|_{\hat{T}_i}^2 \Big)
  \end{align*}
  As a polynomial in $\mathcal{P}^k(\mathbb{R}^d)$ we can retreat $A(\hat{x})$ to $\Thneps{}{i,-\eps}$ and find \vspace*{-0.35cm}
  \begin{align*}
    \| A \|_{\hat{T}_i}  
    &
    \simeq 
    \| A \|_{\Thneps{}{i,-\eps}} 
    \lesssim  
    \| v_h \circ \Theta_{{T}_i}^n - \overbrace{u_h \circ \Theta_{{T}_i}^n}^{\hat{u}_i} \|_{\Thneps{}{i,-\eps}} 
    + 
    \| v_h \circ \Theta_{{T}_i}^n - \overbrace{v_h \circ \Theta_{{T}_i}^{m}}^{\hat{v}_i} \|_{\Thneps{}{i,-\eps}}  \\
    &
    \lesssim 
    \| (\id - \Pi^{n}) v_h \|_{T^n_{i}} + 
    \| v_h \circ (\Theta_{{T}_i}^n - \Theta_{{T}_i}^{m}) \|_{\Thneps{}{i,-\eps}}  
    \lesssim 
    h^2 \| \nabla v_h \|_{\omT{T^m_i}}
  \end{align*}
  where in the last step we made use of \Cref{lemma:proj} and the closeness of $\Theta_{{T}}^{n}$ and $\Theta_{{T}}^{m}$. 
  For $B(\hat{x})$ we first note that $\|\Upsilon^n - \Upsilon^m\|_{\L^{\infty}(\hat{T}_i)} \lesssim h^2$ and recall that $\hat{u}_j - \hat{v}_j \in \mathcal{P}^k(\mathbb{R}^d)$ s.t. with standard scaling arguments we have $\| \hat{u}_j - \hat{v}_j \|_{\Upsilon^n(\hat{T}_i)} \lesssim \| \hat{u}_j - \hat{v}_j \|_{\Thneps{}{j,-\eps}}$ and
  \begin{align*}
    \| B \|_{\hat{T}_i}
    &\lesssim \| (\hat{u}_j - \hat{v}_j) \circ \Upsilon^n\|_{\hat{T}_i} \!\!+\! \| \hat{v}_j \circ \Upsilon^n - \hat{v}_j \circ \Upsilon^m\|_{\hat{T}_j}
  %  &\lesssim \| \hat{u}_j - \hat{v}_j \|_{\Upsilon^n(\hat{T}_i)} + h^{\frac{d}{2}} \| \nabla \hat{v}_j \|_{\L^{\infty}(\hat{T}_i)}  \underbrace{\| \Upsilon^n - \Upsilon^m\|_{\L^{\infty}(\hat{T}_i)}}_{\lesssim h^2} \\
    \lesssim \| \hat{u}_j - \hat{v}_j \|_{\Upsilon^n(\hat{T}_{i})} + h^2 \| \nabla v_h \|_{T_j^m}
  \\
    &\lesssim \| u_h \circ \Theta_{T_j}^n - v_h \circ \Theta_{T_j}^{m} \|_{\Thneps{}{j,-\eps}} + h^2 \| \nabla v_h \|_{T_j^m}     
  \\
    &\lesssim
      \| (u_h - v_h) \circ \Theta_{T_j}^{n} \|_{\Thneps{}{j,-\eps}}
      + \| v_h \circ \Theta_{T_j}^n - v_h \circ \Theta_{T_j}^{m} \|_{\Thneps{}{j,-\eps}}
      + h^2 \| \nabla v_h \|_{T_j^m}     
  \\
    &\lesssim
      \| \Pi^{n} u_h - v_h \|_{T_j^n}
      + h^2 \| \nabla v_h \|_{T_j^m}     
  \lesssim
  h^2 \| \nabla v_h \|_{\omT{T_j^m}}.
  \end{align*}
  Putting all together yields the proof.
  \end{proof}

%\subsection{Proof of \Cref{lemma:consist}}\label{proofs}
\begin{proof}[\underline{Proof of \Cref{lemma:consist}}]\label{proof:lemma:consist}
  Recall \eqref{eq:consist}, i.e. $\consist{n}(v_h) = \sum_{i=1}^4 \consist{n,i}(v_h)$. 
  We start with $\consist{n,1}(v_h)$ by triangle inequality \vspace*{-0.5cm}
  \begin{equation*}
  \! | \consist{n,1} (v_h) | \! 
    \le \overbrace{\left| \int_{\Omhn{n}} \!\!\!
      \left(
      \dtBDFr{2}(u^n,u^{n-1}\!\!\!,u^{n-2})
      - \partial_t u^n \right) v_h dx \right|}^{\consist{n,a}} + \overbrace{\left| \int_{\Omhn{n}} \!\!\! \partial_t u^n v_h dx - \int_{\Omn{n}} \!\!\! \partial_t u^n v_h^{\lift} dx \right|}^{\consist{n,b}}. 
  \end{equation*}
  The first term is similar to the one in \cite[Lemma 5.11]{LO_ESAIM_2019} but differs due to the high-order time difference stencil. By elementary calculations based on partial integration on $[t_{n-2},t_{n-1}]$ and $[t_{n-1},t_n]$ one obtains
  \begin{equation*}
  \consist{n,a} = \int_{\Omhn{n}} \int_{t_{n-2}}^{t_{n}} z(t) ~\partial_{t}^3 u ~dt ~ v_h ~dx
  \end{equation*}
  for $z \in C^1\left([t_{n-2},t_{n}]\right)$ satisfying $z(t)|_{[t_{n-2},t_{n-1}]} = -\frac{1}{4\dt} (t-t_{n-2})^2$ and $z(t)|_{(t_{n-1},t_{n}]} = \frac{1}{12\dt} \left( (3t-t_{n-1}-\Delta t)^2 - 4 \Delta t^2\right)$ and hence $\|z\|_{\L^{\infty}\left([t_{n-2},t_{n}]\right)} \leq \frac13 \dt$. We have with
  $
  \|\partial_t^3 u\|_{\L^\infty(\Omhn{n} \times [t_{n-2},t_n])} \leq  \|\partial_t^3 u\|_{\L^\infty(\Qe)}
  \leq \|\partial_t^3 u\|_{\L^\infty(\Q)}$ the bound
  \begin{align*}
  \consist{n,a} \lesssim \Delta t^2 \|\partial_t^3 u\|_{\L^\infty(\Q)} \|v_h\|_{\Omhn{n}}.
  \end{align*}
  The second term $\consist{n,b}$ involves the high-order geometrical approximation error, i.e. the error from $\Omhn{n} := \Theta(\Omphih{\hat{\phi}_h}{n})$. By notating $\Phi:= \PPsi \circ \Theta^{-1}: \Om_{h}^n \to \Omn{n}$, i.e., a mapping from the discrete domain to the exact domain, and applying an integral transformation we have
  \begin{align*}
  |\consist{n,b}| 
    = & \left| \int_{\Omhn{n}} \left( \partial_t u^n  - (\partial_t u^n \circ \Phi) \det(D \Phi) \right) v_h dx \right|
  \end{align*}
  Splitting the integrand (expect for the factor $v_h$) into the sum of 
  $\partial_t u^n  - (\partial_t u^n \circ \Phi)$ and 
  $(\partial_t u^n \circ \Phi)(1 - \det(D \Phi)) $ yields
  \begin{align*} 
  |\consist{n,b}| 
  &\le \left| \int_{\Omhn{n}} |\id - \Phi| ~ \| \nabla \partial_t u \|_{\L^\infty(\Om_{\epsilon}^n)} v_h dx \right| + \left| \int_{\Omhn{n}} |1- \det(D \Phi)| ~ \|\partial_t u\|_{\L^\infty(\Om_{\epsilon}^n)} v_h dx \right| \nonumber \\[-2ex]
  &\lesssim h^{q+1} \|u\|_{\W^{2}_{\infty}(\Q)} \|v_h\|_{\L^1(\Omhn{n})} + h^q \|u \|_{\W^{1}_{\infty}(\Q)} \|v_h\|_{\L^1(\Omhn{n})} 
  \lesssim h^q \|u\|_{\W^{2}_{\infty}(\Q)} \|v_h\|_{\Omhn{n}}, %\label{eq:shift and map}   %\nonumber
  \end{align*}
  where we made use of
  \begin{equation*}
    \left| \partial_t u^n(x)  - (\partial_t u^n \circ \Phi) (x) \right| \le \| \nabla \partial_t u \|_{\L^\infty(\Om_{\epsilon}^n)} \left| x - \Phi (x) \right|
  \end{equation*}
  and \Cref{propertiesdh} to bound the geometrical errors. 
  
  The bound for the second part follows similar lines in \cite[Lemma 5.11]{LO_ESAIM_2019} and yields %\cite[Lemma 5.11]{LO_ESAIM_2019}
  \begin{equation*}
  |\consist{n,2}(v_h)| = |\bh{n}(u^n,v_h) - b^n(u^n,v_h^{\lift})| \lesssim h^q \|u\|_{\W^{2}_{\infty}(\Q)} \|v_h\|_{\H^1(\Omhn{n})}. 
  \end{equation*}
  
  The third part is obtained with a Cauchy-Schwarz inequality and the estimate from \cite[Lemma 5.8]{LO_ESAIM_2019}, which is still valid on deformed meshes. 
  \begin{align*}
  \consist{n,3}(v_h) &= \gamma \srn{2}{n} (u^n, v_h) 
  \le \gamma \srn{2}{n} (u^n, u^n)^{\frac{1}{2}} \srn{2}{n} (v_h, v_h)^{\frac{1}{2}} \nonumber \\
  &\lesssim \gamma h^k \|u^n\|_{\H^{k+1}(\O_{h,2}^n)} \srn{2}{n} (v_h, v_h)^{\frac{1}{2}} 
  \lesssim \K^\frac12 h^k \|u^n\|_{\H^{k+1}(\Omn{n})} \left( \gamma \srn{2}{n} (v_h, v_h) \right)^{\frac{1}{2}}, 
  \end{align*}
  where for the last inequality we make use of $\O_{h,2}^n \subset \Om^{n}_{\epsilon}$ and continuity of the extension. 

  Finally, the fourth part is estimated analogously to $|\consist{n,b}|$ as follows
  \begin{align*}
  |&\consist{n,4}(v_h)|
  = \left| \int_{\Omhn{n}} \left( (g \circ \Phi) \det(D \Phi) - g \right) v_h ~dx \right| \le \quad  \|v_h\|_{\L^1(\Omhn{n})} \quad \cdot \\
  & \quad \qquad \cdot \big( \|1 - \det(D \Phi)\|_{\L^\infty(\Om_{h,\eps}^n)} \|g\|_{\L^\infty(\Om_{h,\eps}^n)}   \nonumber %\\
  + \|\id - \Phi\|_{\L^\infty(\Om_{h,\eps}^n)} \|\nabla g\|_{\L^\infty(\Om_{h,\eps}^n)} \big) \\
  &\lesssim \big( h^q \|g\|_{\L^\infty(\Omhn{n})} + h^{q+1} \|g\|_{\W^1_\infty(\Omhn{n})} \big) \cdot \|v_h\|_{\Omhn{n}} \nonumber 
  \lesssim \sup_{t \in [0,\finaltime]} h^q \|g\|_{\W^1_\infty(\Om(t))} \|v_h\|_{\Omhn{n}}. \nonumber
  \end{align*}  
\end{proof}  

\extendedonly{

}

% \bibliography{literatur}{}
\bibliographystyle{siam}

\bibliography{literatur}
\end{document}